\documentclass[11pt,a4paper,final,titlepage]{amsart}
\usepackage[a4paper,left=30mm,right=30mm,top=30mm,bottom=30mm,marginpar=23mm]{geometry} 
\usepackage{amsmath}
\usepackage{amssymb}
\usepackage{amsthm}
\usepackage{amscd}
\usepackage[ansinew]{inputenc}
\usepackage{cite}
\usepackage{bbm}
\usepackage{color}
\usepackage[english=american]{csquotes}
\usepackage[final]{graphicx}
\usepackage[square,numbers]{natbib}
\usepackage[colorlinks=true,linkcolor=blue,citecolor=blue,urlcolor=blue,pdfborder={0 0 0}]{hyperref}
\usepackage{calc}
\usepackage{mathptmx}
\usepackage{titletoc}
\usepackage{amsaddr}

\numberwithin{equation}{section}

\newtheoremstyle{thmlemcorr}{10pt}{10pt}{\itshape}{}{\bfseries}{.}{10pt}{{\thmname{#1}\thmnumber{ #2}\thmnote{ (#3)}}}
\newtheoremstyle{thmlemcorr*}{10pt}{10pt}{\itshape}{}{\bfseries}{.}\newline{{\thmname{#1}\thmnumber{ #2}\thmnote{ (#3)}}}
\newtheoremstyle{remexample}{10pt}{10pt}{}{}{\bfseries}{.}{10pt}{{\thmname{#1}\thmnumber{ #2}\thmnote{ (#3)}}}

\theoremstyle{thmlemcorr}
\newtheorem{theorem}{Theorem}
\numberwithin{theorem}{section}
\newtheorem{lemma}[theorem]{Lemma}
\newtheorem{corollary}[theorem]{Corollary}
\newtheorem{proposition}[theorem]{Proposition}

\newtheorem{definition}[theorem]{Definition}
\newtheorem{notation}[theorem]{Notation}

\theoremstyle{thmlemcorr*}
\newtheorem{theorem*}{Theorem}
\newtheorem{lemma*}[theorem]{Lemma}
\newtheorem{corollary*}[theorem]{Corollary}
\newtheorem{proposition*}[theorem]{Proposition}
\newtheorem{problem*}[theorem]{Problem}
\newtheorem{conjecture*}[theorem]{Conjecture}
\newtheorem{definition*}[theorem]{Definition}

\theoremstyle{remexample}
\newtheorem{remark}[theorem]{Remark}

\newcommand{\R}{\mathbb{R}}
\newcommand{\N}{\mathbb{N}}

\renewcommand\appendix{\par
        \renewcommand\thesection{A}
        \renewcommand\thesubsection{A\arabic{subsection}}
        \renewcommand\thetable{A\arabic{table}}}

\def\Xint#1{\mathchoice 
{\XXint\displaystyle\textstyle{#1}}%
{\XXint\textstyle\scriptstyle{#1}}%
{\XXint\scriptstyle\scriptscriptstyle{#1}}%
{\XXint\scriptscriptstyle\scriptscriptstyle{#1}}%
\!\int} 
\def\XXint#1#2#3{{\setbox0=\hbox{$#1{#2#3}{\int}$} 
\vcenter{\hbox{$#2#3$}}\kern-.5\wd0}}

\newcommand{\W}{\mathrm{W}}
\newcommand{\LL}{\mathrm{L}}
\newcommand{\st}{\ensuremath{\mathrm{ \hspace*{1mm}\textup{\textbf{:}}\hspace*{1mm} }}}
\newcommand{\dx}{\, \mathrm{d}x}
\newcommand{\dy}{\, \mathrm{d}y}
\newcommand{\dt}{\, \mathrm{d}t}
\newcommand{\dv}[1]{\, \mathrm{d}#1}

\renewcommand{\phi}{\varphi}

\def\Xint#1{\mathchoice
{\XXint\displaystyle\textstyle{#1}}%
{\XXint\textstyle\scriptstyle{#1}}%
{\XXint\scriptstyle\scriptscriptstyle{#1}}%
{\XXint\scriptscriptstyle\scriptscriptstyle{#1}}%
\!\int}
\def\XXint#1#2#3{{\setbox0=\hbox{$#1{#2#3}{\int}$}
\vcenter{\hbox{$#2#3$}}\kern-.5\wd0}}
\def\mint{\Xint=}
\def\mint{\Xint-}

\newcommand{\measurerestr}{%
  \,\raisebox{-.127ex}{\reflectbox{\rotatebox[origin=br]{-90}{$\lnot$}}}\,%
}

\let\oldtocsection=\tocsection

\let\oldtocsubsection=\tocsubsection

\let\oldtocsubsubsection=\tocsubsubsection

\renewcommand{\tocsection}[2]{\hspace{0em}\oldtocsection{#1}{#2}}
\renewcommand{\tocsubsection}[2]{\hspace{1em}\oldtocsubsection{#1}{#2}}
\renewcommand{\tocsubsubsection}[2]{\hspace{2em}\oldtocsubsubsection{#1}{#2}}

\begin{document}


\title[Boundary regularity and sufficient conditions for strong local minimizers]{Boundary regularity and sufficient conditions for strong local minimizers}
\author{Judith Campos Cordero}

\address{Institut f\"ur Mathematik, Universit\"at Augsburg, Universit\"atsstrasse 14, 86159, Augsburg, Germany}

\curraddr{ Universidad Aut\'onoma Metropolitana -- Iztapalapa, Departamento de Matem\'aticas, Av. San Rafael Atlixco No. 186 Col. Vicentina, C.P. 09340, Ciudad de M\'exico, M\'exico.}
\email{judith.campos@xanum.uam.mx}


\hypersetup{
  pdfauthor = {},
  pdftitle = {...},
  pdfsubject = {},
  pdfkeywords = {}
}


\begin{abstract}
In this paper we present a new proof of the sufficiency theorem for strong local minimizers concerning $C^1$-extremals at which the second variation is strictly positive. The results are presented  in the quasiconvex setting, in accordance with the original statement by Grabovsky and Mengesha (2009). The strategy that we follow relies on a decomposition theorem that allows to split a sequence of variations into its oscillating and its concentrating parts, as well as on a sufficiency result according to which smooth extremals are spatially-local minimizers. Furthermore, we prove partial regularity up to the boundary for strong local minimizers in the non-homogeneous case and a full regularity result for Lipschitz extremals with gradient of vanishing mean oscillation. As a consequence, we also establish a sufficiency result for this class of extremals.  The regularity results are established via a blow-up argument.
\noindent\textsc{}
\vspace{0.1cm}

 \vspace{0.1cm}

\vspace{0.1cm}

\end{abstract}

%

    \global\let\newpagegood\newpage
 \global\let\newpagegood\newpage
    \global\let\newpage\relax

\maketitle

\setcounter{tocdepth}{3} 
\tableofcontents

\global\let\newpage\newpagegood


\section{Introduction}

It is well known that minimizers of variational integrals satisfy the Euler-Lagrange equations and the nonnegativity of the second variation. On the other hand, for non-convex integrands we are faced with the possibility of encountering local minimizers that are not globally minimizing the functional. Hence, an underlying problem in the Calculus of Variations has been to find sufficient conditions guaranteeing that an extremal at which the second variation is positive is in fact a strong local minimizer. 

This problem was solved by Weierstrass in the 19$^{\mathrm{th}}$ century for the case in which the admissible scalar functions have one single independent variable. Levi then provided a proof with a different method, which does not use the field theory originated from Weierstrass' work \cite{levisui}. The sufficiency result was later generalized by Hestenes for functions of several variables \cite{Hestenes}.

On the other hand, Meyers showed that the notion of quasiconvexity developed by Morrey \cite{Morrey} is in a suitable sense a necessary condition for  strong local minimizers \cite{Meyers}. Furthermore, Ball and Marsden established the concept of \textit{quasiconvexity at the free boundary} in \cite{BallMarsden}, where this was also shown to be a necessary condition for strong local minimizers satisfying mixed boundary conditions, that is, by allowing the minimizers to take free values on part of the boundary of their domain.  

The importance of obtaining an adequate set of sufficient conditions for strong local minima in the vectorial case was highly motivated by applications coming from materials science. Ball foresaw  that the natural way to extend Weierstrass condition to the vectorial case are the notions of quasiconvexity at the interior and at the free boundary. Furthermore, in \cite[Section 6.2]{Ballcv} it is conjectured that if a solution to the weak Euler-Lagrange equation is sufficiently smooth, then the strict positivity of the second variation together with suitable versions of the quasiconvexity in the interior and at the boundary should guarantee that the extremal furnishes a strong local minimizer. 

Further generalizations regarding Weierstrass problem were obtained by Taheri in \cite{Taheri}, where Hestenes' strategy is extended to the case of $\LL^p$-local minimizers. 

Regarding the vectorial problem, Zhang was able to exploit the quasiconvexity of the integrand to establish that smooth extremals are absolute minimizers in sufficiently small subsets of their domain \cite{Zhang}. In contrast, Kristensen and Taheri provided an example, motivated by the work in \cite{MyS}, in which it becomes clear that the Lipschitz regularity of an extremal is not enough to ensure strong local minimality, even under the assumptions of strong quasiconvexity and strict positivity of the second variation \cite[Corollary 7.3]{Kristensen}. 

It was only until the work of Grabovsky and Mengesha \cite{GyM} that a sufficiency theorem for $C^1$ extremals in the quasiconvex setting was settled in the multi-dimensional Calculus of Variations (see also \cite{GyM07}). Their strategy relies, on the one hand, on a decomposition theorem and the method presented in \cite{FonsPed} to prove it (see also \cite{Kdeclem,KristensenLSC}). In addition, a localization principle is established to prove that the concentrating part of a sequence of variations acts on the functional in a localized way, so that the quasiconvexity condition can come into play to show that this part of the sequence does not decrease the functional. That the same occurs with the oscillating part of the sequence of variations is shown using the positivity of the second variation. 

In this context, after a preliminary section the first objective of this paper is developed in Section \ref{SecGMI}. Here, we present an alternative strategy for the sufficiency theorem for $C^1$-extremals. The main purpose is to present the ideas behind the new proof and, hence, we restrict ourselves to the homogeneous case. However, we work under essentially the same assumptions than in \cite{GyM}, the only difference being that we remove the coercivity restriction for the case of Dirichlet boundary conditions when the integrand has quadratic growth. The method that we introduce here consists in exploiting the spatially-local minimality obtained by Zhang \cite{Zhang} and then, via a covering argument and a version of the Decomposition Theorem based in \cite[Lemma 1.7]{KristensenLSC}, we obtain the conclusion using an indirect approach. The use of Zhang's Theorem in the new proof replaces then the localization principle used in \cite{GyM} and enables us to show that both the oscillating and the concentrating part of a sequence of variations increase the functional. We remark that the strategy presented here can be adapted to more general functionals with lower order terms. Details on this and further generalizations of the sufficiency result concerning non-smooth domains can be found in \cite{JCCKK}. 

On the other hand, the regularity assumed a priori on the extremal must play a crucial role in any set of sufficient conditions for strong local minima. This becomes clear by the aforementioned Corollary 7.3 in \cite{Kristensen} which, in turn, is a consequence of the partial regularity theorem for strong local minimizers established in \cite[Theorem 4.1]{Kristensen}. Kristensen-Taheri's result extends Evans' partial regularity theorem for global minimizers  in \cite{Evanspr} to the case of strong local minima. The strategy of their proof is based in Evans' blow-up argument and the assumptions on the integrand are stated in great generality, motivated by the works by Acerbi and Fusco, \cite{AcFus,AcFusLocRegNonconvex}, Evans and Gariepy \cite{EvansGarblowup}, Fusco and Hutchinson  \cite{FusHut}, among others. However, Kristensen and Taheri perform a remarkable modification to establish the strong convergence of the blown-up sequence, given that the estimates available for global minimizers cannot be obtained in the same way for the case of strongly-local small variations. We observe that, considering this regularity result, as well as celebrated examples of singular minmizers (see \cite{NecasExample,SverYan}), it is still an underlying question in what way the regularity assumption on the extremal could be relaxed in the sufficiency theorem. 

Furthermore, Kristensen and Taheri establish in \cite[Theorem 6.1]{Kristensen} that extremals of quasiconvex functionals at which the second variation is strictly positive, are BMO-local minimizers for sets of variations uniformly bounded in $\W^{1,\infty}$. This strengthens the well known result that such extremals actually furnish weak local minimizers, and generalizes previous observations given by \cite{Firoozye} in the same direction. 

Motivated by these results, Sections \ref{SectLipExtBMO}-\ref{SectGralizedSuff} of this paper aim at providing a sufficiency theorem for which the extremal is not assumed, in principle, to be smooth, but only Lipschitz with gradient of vanishing mean oscillation. Given that the nature of an extremal as a solution to the weak Euler-Lagrange equation is not enough to improve its regularity, the course of action that we follow is to generalize first Theorem 6.1 from \cite{Kristensen} by removing the restriction on the boundedness of the allowed variations. We devote Section \ref{SectLipExtBMO} to this purpose, hence establishing that extremals at which the second variation is strictly positive are BMO-local minimizers. 

In addition, in Section \ref{SectBdryReg} we extend up to the boundary the partial regularity results from \cite{Kristensen} for $\W^{1,q}$-local minimizers. Furthermore, we also establish that Lipschitz BMO-local minimizers with derivative of vanishing mean oscillation are of class $C^1$ up to the boundary. In both situations we address the case of non-homogeneous integrands. 

Boundary regularity results for absolute minimizers of quasiconvex variational integrals had already been established in \cite{BeckVarInt}. Beck also treats the case of $\W^{1,p}$-local minimizers for integrands with $p$-growth \cite[Theorem 1.3]{BeckVarInt}. However, for $\W^{1,q}$-minimizers with $q>p$, the problem cannot be reduced to that of absolute minimizers (see \cite[Proposition 2.1]{Kristensen}). Hence, a suitable adaptation of the blown-up technique has to be introduced for the boundary regularity in this case. This is the core problem  that we address in Section \ref{SectBdryReg}.

Furthermore, the characterization of the regular subset of the boundary in terms of Lebesgue points that we give in Section \ref{SectRegularSet} can be applied to improve previous characterizations given in the literature for minimizers or solutions to nonlinear elliptic systems, as in \cite{Grotowskinonlin,Kronzbdryreg,BeckVarInt}. There, the regular points are characterized in terms of the mean oscillations of the total derivative with respect to the normal derivative of the solution, and not in terms of Lebesgue points (see also \cite[Remark 1.2]{BeckVarInt}). Other interesting developments regarding boundary regularity theory have been carried out in \cite{DGK,Beckregweaksols,KristMingBdryreg}.

The final section of the paper compiles the previous results to establish that, although  Lipschitz extremals may not be strong local minimizers, they are so if we assume, in addition, that their derivative is of vanishing mean oscillation.

\section{Preliminaries}
We consider functionals of the type
\begin{equation}
\mathcal{F}(u):=\underset{\Omega}{\int}F(x,\nabla u)\dx,
\end{equation}
where $\Omega$ is a suitably smooth bounded domain, $u$ belongs to a given Sobolev space and $F\colon\overline{\Omega}\times\R^{N\times n}\rightarrow\R$ is such that:
\begin{itemize}
\item[($\overline{\mathrm{H}}$0)]  $F(x,\cdot)$ is of class  $C^2$  in  $\R^{N\times n}$  for every  $x\in\overline{\Omega}$  and  $F_{zz}$  is continuous in  $\overline{\Omega}\times\R^{N\times n}$.
\item [($\overline{\mathrm{H}}$1)] There are $p\in (1,\infty)$ and  a constant $c_1>0$ such that, for every $(x,z)\in\overline{\Omega}\times\R^{N\times n}$,
\begin{equation*}
|F(x,z)|\leq c_1(1+|z|^p).
\end{equation*}
\item[($\overline{\mathrm{H}}$2)] $F$ is strongly $p$-quasiconvex, meaning that there is a constant $c_2>0$ such that, for every $(x_0,z)\in\Omega\times\R^{N\times n}$ and every $\varphi\in \W^{1,\infty}(Q,\R^N)$, it holds that
\begin{equation*}
c_2\underset{Q}{\int}|V(\nabla\varphi)|^2\dx\leq\underset{Q}{\int}F(x_0,z+\nabla \varphi)-F(x_0,z)\dx. 
\end{equation*}
\end{itemize}

The function $V$ is given by $V=V_{\frac{p}{2}}$, where for $\beta>0$ arbitrary and for $k\in\N^+,$ the auxiliary function $V_\beta$ is defined by
\begin{equation*}
V_\beta(\xi):=\left( 1+|\xi|^2\right)^{\frac{\beta-1}{2}}\xi.
\end{equation*}

We emphasize, however, that most results in the following sections will only be stated for the superquadratic case $p\geq 2$ and, furthermore, the proof of the sufficiency theorem appearing in Section \ref{SecGMI} is given only for homogeneous integrands, that is, with no $x$-dependence on $F$. The rest of the results are established in the non-homogeneous setting. 

In the following lemma we compile some standard properties of the function $V_\beta$. We remark that  we make no distinction in the notation  between different dimensions of the domain in which  $V_\beta$ is defined (as in part (ii) below). 

\begin{lemma}\label{lemmapropsV} Let $\beta>0,$ $2\leq p<\infty$ and $M>0$. Then, there is a constant $c>0$, depending only on $\beta$, such that for every $\xi,\eta\in\R^k$ and every $t\geq 0$, 
\begin{itemize}
\item[$\mathrm{(i)}$] $V_\beta(t)$ is non-decreasing in $[0,\infty)$;
\item[$\mathrm{(ii)}$] $|V_\beta(\xi)|=V_\beta(|\xi|)$;
\item[$\mathrm{(iii)}$]  $|V_\beta(\xi+\eta)|\leq\,c\left(|V_\beta(\xi)|+|V_\beta(\eta)|\right)$; 
\item[$\mathrm{(iv)}$]$|V(t\xi)|\leq\max\{t,t^{\frac{p}{2}}\}|V(\xi)|$;
\item[$\mathrm{(v)}$] $c(p)|\xi-\eta|\leq\frac{|V(\xi)-V(\eta)|}{\left(1+|\xi|^2+|\eta|^2\right)^{\frac{p-2}{4}}}\leq c(p,k)|\xi-\eta|$;
\item[$\mathrm{(vi)}$] $\left(1+|\xi|^2+|\eta|^2\right)^{\frac{p}{2}}\leq\,c\left(1+|V(\xi)|^2+|V(\eta)|^2\right)$;
\item[$\mathrm{(vii)}$] $\left|V_{p-1}(\xi)\right||\eta|\leq|V(\xi)|^2+|V(\eta)|^2$;
\item[$\mathrm{(viii)}$] Young's inequality is satisfied in the sense that, for every $\varepsilon>0$, there exists $c_\varepsilon>0$ such that $|V_{p-1}(\xi)||\eta|\leq\,\varepsilon|V(\xi)|^2+c_\varepsilon|V(\eta)|^2$;
\item[$\mathrm{(ix)}$] $\max \{|\xi|,|\xi|^{\frac{p}{2}}\}\leq|V(\xi)|\leq2^{\frac{p-2}{4}}\max\{|\xi|,|\xi|^{\frac{p}{2}}\}$;
\item[] $\frac{1}{2} \left( |\xi|^2+|\xi|^p\right)\leq|V(\xi)|^2\leq\,2^{\frac{p-2}{4}}\left( |\xi|^2+|\xi|^p\right)$;
\item[$\mathrm{(x)}$] $|V(\xi-\eta)|\leq c(p)|V(\xi)-V(\eta)|$;
\item[$\mathrm{(xi)}$] $|V(\xi)-V(\eta)|\leq c(k,p,M)|V(\xi-\eta)|$, provided $|\eta|\leq M$.
\end{itemize}  
\end{lemma}
\begin{remark}
With the exception of (ix)-(xi), all the above properties also hold if we consider $1<p<2$ in the definition of $V_\beta$. This will be used for some results in Section \ref{SectLipExtBMO}.   
\end{remark}

The proof of this lemma requires essentially elementary computations and consider  separately the cases $|\xi|,|\eta|<1$ and $|\xi|,|\eta|\geq 1$, with all the combinations that arise from these possibilities. We refer the reader to \cite{AcFussubquad}, \cite[Lemma 2.1]{CarrFuscMing} and  \cite[Lemma 2.1]{BeckVarInt} for further details.

We now compile  a list of symbols and conventions that will be used throughout the text. 
\begin{notation}
 $\R^{N\times n}$ denotes of real matrices of $N\times n$. When $N=1$, we simply write $\R^n$. 
\\
Given $x,y\in \R^{N\times n}$, $xy$ denotes the standard matrix multiplication between $x$ and $y$. 
\\
For $x,y\in\R^{N\times n}$ the scalar product is given by $\langle x,y\rangle=\mathrm{tr}(xy^T)$ and the Euclidean norm in $\R^{N\times n}$ is denoted by $|x|:=\langle x,x\rangle^\frac{1}{2}$. 
\\
The symbol $\mathrm{e_i}$ denotes the canonical vector in $\R^n$ whose $i$-th entry is $1$ and the rest are $0$. We use $\mathrm{I}_n$ to denote the identity matrix in $\R^{n \times n}$.  
\\
For a given vector $x\in\R^n$, $x_i:=\langle x,\mathrm{e_i}\rangle$. 
\\
For a set $A\subseteq\R^n$, $\overline{A}$ denotes its closure in $\R^n$, $\mathrm{int}(A)$ is the interior of $A$ and $\mathcal{L}^n(A)$ denotes its Lebesgue measure in $\R^n$. In addition, $\mathcal{H}^k(A)$ denotes its Hausdorff measure of order $k$. 
\\
$B(x_0,r)$ represents the open ball in $\R^n$ with center $x_0$ and radius $r>0$. If $x_0\in\R^{n-1}\times \{0\}$, we denote the upper half ball centred at $x_0$ and with radius $r>0$ by  
\begin{equation*}
B^+(x_0,r):=\{x\in B(x_0,r)\st \langle x,\mathrm{e_n}\rangle>0\}.
\end{equation*}
In addition, we denote $B[x_0,r]:=\overline{B(x_0,r)}$ and $B^+[x_0,r]:=\overline{B^+(x_0,r)}$.
\\
Given a bounded set $\omega\subseteq\R^n$ with $\mathcal{L}^n(\omega)>0$ and $f\in \LL^1(\omega,\R^N)$, we write
\begin{equation*}
(f)_\omega:=\underset{\omega}{\mint} f\dx=\frac{1}{\mathcal{L}^n(\omega)}\underset{\omega}{\int}f\dx.
\end{equation*}
If $\omega=B(x_0,r)$,  we abbreviate $(f)_{B(x_0,r)}=(f)_{x_0,r}$. This notation will also be used with one variant: when $\omega=\Omega\cap B(x_0,r)$ for some $x_0\in\overline{\Omega}$, we write $\Omega(x_0,r):=\Omega\cap B(x_0,r)$ and for $p\in [1,\infty)$, $f\in \LL^1(\Omega(x_0,r),\R^N)$ we will denote
\begin{equation*}
\underset{\Omega(x_0,r)}{\mint}|f-(f)_{x_0,r}|^p\dx:=\frac{1}{\mathcal{L}^n(\Omega(x_0,r))}\underset{\Omega(x_0,r)}{\int}\left|f-\frac{1}{\mathcal{L}^n(\Omega(x_0,r))}\underset{\Omega(x_0,r)}{\int}f\dy\right|^p\dx.
\end{equation*}
\end{notation}

Concerning $\Omega$, we assume that $\Gamma_D\subseteq\partial\Omega$ is a subset of $\partial\Omega$ such that the $(n-1)$-Hausdorff measure of the relative interior of $\Gamma_D$ is positive (so, in particular, $\Gamma_D\neq\emptyset$).

We define the admissible functions as follows. 

\begin{definition}\label{Defadmfunct}
Given $p\in (1,\infty)$,\footnote{The exponent $p$ will be related, in this case, to the growth condition imposed on the integrand $F$.} we define the \textbf{set of admissible functions} as
\begin{equation*}
\mathcal{A}:=\overline{\left\{u\in C^1({\overline{\Omega}},\R^N) \st u(x)=g(x)\,\forall x\in\Gamma_D\right\}^{\W^{1,p}}},
\end{equation*}
where the closure is taken in $\W^{1,p}(\Omega,\R^N)$ and $g$ is of class $C^1$ in some open set in $\R^n$ containing $\overline{\Gamma_D}$.
\end{definition}
We henceforth assume that $\Gamma_D$ is the interior of $\overline{\Gamma_D}$, relative to $\partial\Omega$. By defining $\Gamma_N:=\partial\Omega-\overline{\Gamma_D}$, it holds that $\Gamma_N$ is a relatively open subset of $\partial\Omega$ and $\partial\Omega=\Gamma_D\cup\overline{\Gamma_N}$. Indeed, if $x\in \partial\Omega-\overline{\Gamma_N}$, then $x$ has an open neighbourhood in $\partial\Omega$ that does not intersect $\Gamma_N$. Therefore, this neighbourhood must belong to the interior of $\overline{\Gamma_D}$, which is $\Gamma_D$. Assuming further that $\Gamma_D$ is Lipschitz and has, in turn, a Lipschitz boundary in $\partial\Omega$, we can conclude that if $u\in\mathcal{A}\cap C^1(\overline{\Omega},\R^N)$, then $u(x)=g(x)$ for all $x\in\overline{\Gamma_D}$.\footnote{See also \cite[Proposition A.2]{ADMDS}.}   

We emphasize that for the results in Sections \ref{SectLipExtBMO}-\ref{SectGralizedSuff} we assume $\Gamma_N=\emptyset$ and so, in that case, $\mathcal{A}=\W^{1,p}_g(\Omega,\R^N)$. 

\begin{definition}
We define the \textbf{space of variations} as the set 
\begin{equation*}
\mathrm{Var}(\mathcal{A}):=
\overline{\left\{
\varphi\in C^1(\overline{\Omega},\R^N)\st \varphi(x)=0 \,\forall x\in\Gamma_D
\right\}^{\W^{1,p}}}.
\end{equation*}
\end{definition}
We call $\mathrm{Var}(\mathcal{A})$ the space of variations because, for any $y_1,y_2\in\mathcal{A}$, we have $y_1-y_2\in \mathrm{Var}(\mathcal{A})$. More generally, given an open set $\omega\subseteq\R^n$ such that $\omega\cap\Omega\neq\emptyset$, we consider the following space of variations defined in $\omega$, which is naturally embedded in $\mathrm{Var}(\mathcal{A})$. 

\begin{equation*}
\mathrm{Var}(\omega,\R^N):=\overline{\left\{
\varphi\in C^1(\overline{\omega},\R^N)\st \varphi(x)=0\,\forall x\in (\Gamma_D\cap{\overline{\omega}})\cup(\partial\omega\cap\Omega)
\right\}^{\W^{1,p}}}.
\end{equation*}
Observe that, given $\varphi\in\mathrm{Var}(\omega,\R^N)$, by extending $\varphi$ to $\Omega$ so that it takes the value of $0$ in $\Omega\backslash\omega$, we can assume that $\varphi\in\mathrm{Var}(\mathcal{A})$. We remark that, with this notation, $\mathrm{Var}(\mathcal{A})=\mathrm{Var}(\Omega,\R^N)$.

We now recall the notions of weak and strong local minimizer. 
\begin{definition}\label{DefWeStlocmin}
Let $u\in \W^{1,p}(\Omega,\R^N)$. We say that $u$ is a \textbf{weak local $F$-minimizer} if and only if there is a $\delta>0$ such that
\begin{equation}\label{minimality}
\underset{\Omega}{\int} F(x,\nabla u)\dx\leq \underset{\Omega}{\int} F(x,\nabla u+\nabla \varphi)\dx
\end{equation}
for every $\varphi$ in $\mathrm{Var}(\mathcal{A})$ with $\|\nabla\varphi\|_{\LL^\infty}<\delta$.
\\
\\
On the other hand, we say that $u$ is a \textbf{strong local $F$-minimizer} if and only if there is a $\delta>0$ such that (\ref{minimality}) holds for every $\varphi$ in $\mathrm{Var}(\Omega,\R^N)$ with $\|\varphi\|_{\LL^\infty}<\delta$. For $1<p<\infty$ the notion of strong local minimizer can be generalized by considering the $\LL^p$ or the $\W^{1,p}$ norm of the variations. In those cases we say, respectively, that $u$ is an $\LL^p$\textbf{-local minimizer} or a $\W^{1,p}$\textbf{-local minimizer}.
\end{definition}

We also use the following terminology concerning  solutions to the Weak Euler-Lagrange equation. 
\begin{definition}\label{defextremalMixedBdry}
Let $u\in \W^{1,p}(\Omega,\R^N)$. We say that $u$ is an \textbf{$F$-extremal} if and only if 
\begin{equation*}
\underset{\Omega}{\int} \left\langle F_z(x,\nabla u),\nabla\varphi\right\rangle\dx=0
\end{equation*}
for every $\varphi\in \mathrm{Var}(\Omega,\R^N)$.
\end{definition}
It is well known in the Calculus of Variations that local minimizers are $F$-extremals under the assumed conditions. 
The last part of this section is devoted to fixing the terminology concerning Young measures, that we will use extensively in Section \ref{SecGMI}. 
\begin{notation}
Let $\mathcal{B}(E)$ be the Borel $\sigma$-algebra of all the Borel sets contained in the set with finite Lebesgue measure $E$. We denote
\begin{equation*}
\mathcal{M}(E,\R^d):=\{\mu\colon\mathcal{B}(E)\rightarrow\R^d\st \mu  \text{ is a bounded Borel measure}\}
\end{equation*}
and
\begin{equation*}
\mathcal{M}^+(E)=\mathcal{M}^+(E,\R):=\{\mu\in \mathcal{M}(E,\R)\st \mu(B)\geq 0 \text{ for every Borel set } B\subseteq\Omega\}.
\end{equation*}
In addition, we use the following notation for the space of \textbf{probability measures} on $E$.
\begin{equation*}
\mathcal{M}^+_1(E):=\{\mu\in \mathcal{M}^+(E)\st \mu(E)=1\}.
\end{equation*}
\end{notation}
The following definitions finally aim at establishing the concept of Young measure.
\begin{definition}
Let $\nu\colon E\rightarrow \mathcal{M}(\R^d,\R)$. We say that $\nu$ is \textbf{weakly$^*$-measurable} if and only if the mapping
\begin{equation*}
x\in E\mapsto \langle\phi,\nu(x)\rangle=\underset{\R^d}{\int}\Phi(y)\dv{}\nu(x)\dy
\end{equation*}
is Borel measurable for all $\Phi\in C^0_0(\R^d)$, where $C^0_0(\R^d)$ is the space of continuous functions that vanish at infinity.
\end{definition}
\begin{definition}
Let $\mu\in\mathcal{M}^+(E)$ and $\nu\colon E\rightarrow\mathcal{M}^+_1(\R^d)$ be weakly$^*$-measurable. Then the generalized product measure $\mu\otimes\nu(x)$ given by
\begin{align*}
\langle\Phi,\mu\otimes\nu(x)\rangle&:=\underset{\Omega\times\R^d}{\int}\Phi \dv{}(\mu\otimes\nu(x))
=\underset{\Omega}{\int}\underset{\R^d}{\int}\Phi(x,z)\dv{}\nu(x)(z)\dv{}\mu(x)
\end{align*}
for $\Phi\in C^0_0(E\times\R^d)$, is called a \textbf{$\mu$-Young measure on $E$ with target $\R^d$}. 
\end{definition}
\begin{remark}
Henceforth, we write a weakly$^*$ measurable map $\nu\colon\Omega\rightarrow\mathcal{M}^+_1(\R^d)$ as a  parametrized family of probability measures, for which we use the notation $(\nu_x)_{x\in\Omega}$, with $\nu_x:=\nu(x)$. 
On the other hand, we will suppress $\mu$ from the notation and we will just call the family $(\nu_x)_{x\in\Omega}$ a Young measure. 
\end{remark}
We conclude this preliminary section with the following terminology regarding Young measures generated by sequences of measurable functions. 
\begin{definition}
Let $(f_j)$ be a sequence of measurable maps defined on a Lebesgue measurable set $\Omega$ such that $\mathcal{L}^n(\Omega)<\infty$ and let $(\mu_x)_{x\in\Omega}\subseteq\mathcal{M}_1^+(\R^d)$ be a measurable family. We say that $f_j$ \textbf{generates the Young Measure $(\mu_x)_{x\in\Omega}$} if and only if $\mathcal{L}^n\otimes\delta_{f_j(x)}\overset{*}{\rightharpoonup}\mathcal{L}^n\otimes\mu_x$ in $C_0^0(\Omega\times\R^d)^*$, where $\delta_{f_j(x)}$ stands for the Dirac measure concentrated at $f_j(x)$. We denote this by $f_j\overset{Y}{\longrightarrow}(\mu_x)_{x\in\Omega}$.
\end{definition}
We refer the reader to \cite{FonsMM,muller1999variational} for nice compilations of the theory of Young measures. 

\section{A new proof of Grabovsky-Mengesha sufficiency theorem for strong local minimizers in the homogeneous case}\label{SecGMI}

In this section we provide an alternative strategy for the sufficiency theorem established by Grabovsky and Mengesha in \cite{GyM}. 

The proof relies mainly on two results. The first one concerns a theorem by Zhang according to which smooth extremals minimize the functional when restricted to small subsets of $\Omega$ \cite{Zhang}. On the other hand, we use Kristensen's Decomposition Theorem  \cite{Kdeclem,KristensenLSC} (see also \cite{FonsPed}) to exploit the positivity of the second variation and relate it to the behaviour of a shifted functional acting on a suitable sequence of variations. This strategy is discussed in detail in Section  \ref{SectNewProofSuff}.

We emphasize that the result presented here concerns exclusively the case of homogeneous integrands, as it was originally developed in \cite{JCCThesis}. However, the ideas involved can be adapted to consider more general functionals of the type
\begin{equation*}
\mathcal{F}(u):=\underset{\Omega}{\int}F(x,u,\nabla u)\dx,
\end{equation*}
as discussed in Remark \ref{RemarkSuffNonhom}.

The assumptions that we make on the integrand in this section are the homogeneous version of hypotheses $(\mathrm{\overline{H}}0)$-$(\mathrm{\overline{H}}2)$. We restate them here as follows. 
\begin{itemize}
\item[$\mathrm{(H0)}$] $F\in C^2(\R^{N\times n})$;
\item[$\mathrm{(H1)}$] $F$ has $p$-growth, i.e., for some fixed $p\in(1,\infty)$ there exists a constant $c_1>0$ such that, for every $z\in\R^{N\times n}$, $|F(z)|\leq c_1(1+|z|^p)$ and 
\item [$\mathrm{(H2)}$] $F$ is strongly $p$-quasiconvex, meaning that there is $c_2>0$ such that, for every $z\in \R^{N\times n}$ and every $\varphi \in \W^{1,\infty}_0(Q, \R^N)$, it holds that
\begin{equation*}
c_2\underset{Q}{\int}|V(\nabla\varphi)|^2 \dx\leq \underset{Q}{\int}\bigl(F(z+\nabla{\varphi})-F(z)\bigr)\dx.
\end{equation*}
\end{itemize}

\subsection{Quasiconvexity at the free boundary}\label{SectQCbdry}
For a quasiconvexity-based sufficiency result comprising admissible functions with free boundary values, the notion of quasiconvexity at the boundary becomes necessary.  

This concept was defined by Ball and Marsden in \cite{BallMarsden}, where it was established  that, under mixed boundary conditions allowing a minimizer  $\overline{u}$ to take free values on part of the boundary, a necessary condition related to Morrey's notion of quasiconvexity must be satisfied. 
%

For this reason, in addition to (H0)-(H2), we assume in this section that the integrand $F$ is \textbf{strongly quasiconvex on the free boundary}, meaning that, for a constant $c_2>0$,
\begin{equation}
\tag{H2'}
c_2\underset{B^-_{{\mathrm{n}(x_0)}}}{\int}|V\left(\nabla\varphi\right)|^2\dx\leq \underset{B^-_{{\mathrm{n}(x_0)}}}{\int}(F(\nabla u(x_0)+\nabla{\varphi})-F(\nabla u(x_0)))\dx
\end{equation}
for all $\varphi\in \mathrm{V}_{{\mathrm{n}(x_0)}}$, where
\begin{equation}\label{admissvars}
\mathrm{V}_{\mathrm{n}(x_0)}:=\left\{\varphi\in C^\infty \left(\overline{B^-_{\mathrm{n}(x_0)}},\R^N\right)\st\varphi(x)=0\mbox{ on } (\partial B(x_0,1))\cap \overline{B^-_{\mathrm{n}(x_0)}}\right\}
\end{equation}
and ${\mathrm{n}(x_0)}$ is the outer unit normal to $\partial\Omega$ at $x_0\in \Gamma_N$. Here, $B^-_{\mathrm{n}(x_0)}$ is the half of the ball $B(x_0,1)$ that lies in the half plane $\{z\in\R^n\st \langle z-x_0,{\mathrm{n}(x_0)}\rangle <0\}$. 

The spirit in which this condition is shown to be necessary for strong local minima is the same in which the quasiconvexity at the interior is also proven to hold when in presence of strong local minimizers.\footnote{$F$ is quaxiconvex at an interior point $x_0\in \Omega$ if $\underset{B}{\int}F(\nabla u(x_0)+\nabla \varphi)-F(\nabla u(x_0))\dx\geq0$ for all $\varphi\in \W^{1,p}_0(B,\R^N)$.} We observe, however, that the quasiconvexity at the boundary differs from the one in the interior in the sense that it is not anymore a convexity notion. It is enough to recall, for example, that for a convex integrand $F$ the quasiconvexity in the interior can be seen as a straightforward consequence of Jensen's inequality for probability measures. However, given $x_0\in \Gamma_N$ and $\varphi\in \mathrm{V}_{{\mathrm{n}(x_0)}}$, we can follow the same ideas if we consider the probability measure defined on the space of matrices $\R^{N\times n}$ by 
\begin{equation*}
\langle \Phi,\nu_{\varphi,x_0} \rangle:=\underset{B_{\mathrm{n}(x_0)}^-}{\mint}\Phi(\nabla u(x_0)+\nabla\varphi(x))\dx.
\end{equation*}
We then observe that the centre of mass of this probability measure is given by
\begin{equation*}
\overline{\nu}_{\varphi,x_0}=\underset{R^{N\times n}}{\int}z\dv{}\nu_{\varphi,x_0}(z)=\nabla u(x_0)+\underset{B_{\mathrm{n}(x_0)}^-}{\mint}\nabla\varphi(x)\dx.
\end{equation*}
Here, $\underset{B_{\mathrm{n}(x_0)}^-}{\mint}\nabla\varphi(x)\dx\neq 0$ in general, as $\varphi\in \mathrm{V}_{\mathrm{n}(x_0)}$ and it is not necessarily $0$ at the boundary. On the other hand,  Jensen's inequality implies, for a convex integrand $F$, that
\begin{equation*}
F(\overline{\nu}_{\varphi,x_0})\leq \underset{\R^{N\times n}}{\int}F(z)\dv{}\nu_{\varphi,x_0}=\underset{B_{\mathrm{n}(x_0)}^-}{\mint}F(\nabla u(x_0)+\nabla\varphi(x))\dx.
\end{equation*}
This means, in particular, that the notion of quasiconvexity at the free boundary doesn't follow from convexity in the same way that quasiconvexity at the interior does. 

Following the spirit in which the calculations above are made, we can consider the following specific example of a convex function that is not quasiconvex at the free boundary. Let $F:\R^2\rightarrow\R$ be given by 
\begin{equation*}
F(u,v):=v
\end{equation*}
and let $\Omega:=B_{\mathrm{n}(0,1)}^-$ be the half of the unit ball centred at zero that lies below the ${x}$-axis. We consider the mixed boundary conditions according to which the admissible test functions are precisely those in the set $\mathrm{V}_{\mathrm{n}(0)}$  defined in (\ref{admissvars}). It is then clear that the function $\varphi:\R^2\rightarrow\R$ defined as
\begin{equation*}
\varphi(x,y)=x^2+y^2-1
\end{equation*}
is such that $\varphi\in \mathrm{V}_{\mathrm{n}(0)}$. However, the quasiconvexity at the boundary condition is not satisfied for the convex function $F$ at the point $(0,0)$, which lies on the free boundary, since this would imply that, for the particular $\varphi$ that we defined above, 
\begin{equation*}
0=F(0,0)\leq \underset{B_{\mathrm{n}(0,1)}^-}{\mint}F(\nabla\varphi(x,y))\dx\dy=\underset{B_{\mathrm{n}(0,1)}^-}{\mint}2y\dx\dy,
\end{equation*}
which is a contradiction by definition of $B_{\mathrm{n}(0,1)}^-$.

Regarding how to interpret the notion of quasiconvexity at the free boundary, we observe that, by differentiating $t\mapsto\underset{B_{\mathrm{n}(x_0)}^-}{\int}F(\nabla u(x_0)+t\nabla \varphi(x))\dx$ we obtain, from the \textit{quasiconvexity at the boundary} condition, that it implies 
\begin{align}\label{nonlinearNeum}
0=&\underset{B_{\mathrm{n}(x_0)}^-}{\int}\left\langle F'(\nabla u(x_0)),\nabla\varphi(x)\right\rangle\dx=\left\langle F'(\nabla u(x_0)),\underset{\partial B_{\mathrm{n}(x_0)}^-}{\int}\varphi\otimes {\mathrm{n}(x)}\dv{}\sigma(x)\right\rangle\notag\\
=&\left\langle F'(\nabla u(x_0)),\underset{\partial B_{\mathrm{n}(x_0)}^-}{\int}\varphi\dv{}\sigma(x)\otimes {\mathrm{n}(x_0)}\right\rangle,
\end{align}
where the second identity above follows from the Divergence Theorem. 

This enables us to understand the quasiconvexity at the boundary as a \textit{non-linear variational Neumann condition.} Indeed, since (\ref{nonlinearNeum}) holds for every $\varphi\in  \mathrm{V}_{{\mathrm{n}(x_0)}}$, it in turn implies the Neumann boundary condition $\left\langle F'(\nabla u(x_0)), \mathrm{n}(x_0)\right\rangle=0$ in $\R^N$. However, as pointed out by Ball and Marsden with an example in \cite{BallMarsden}, the quasiconvexity at the boundary is still a stronger notion. 

We remark here that, as it turns out, the quasiconvexity on the free boundary is one of the sufficient conditions first established in \cite{GyM} ensuring that a $C^1$ extremal furnishes a strong local minimizer. This was conjectured by Ball in \cite[Section 6.2]{Ballcv}. See also \cite{kruzikBQC,kruzikBNL,mielkeBQC,JCCKK} for further works related to the notion of quasiconvexity at the boundary.

\subsection{The sufficiency theorem for strong local minimizers}
The precise statement of the sufficiency theorem that we establish in this section is the following. 
\begin{theorem}\label{theoGM}
Let ${F\colon\R^{N\times n}\rightarrow \R}$ such that it satisfies $\mathrm{(H0)-(H2)}$ for some $p\in[2,\infty)$. Let $u\in C^1\left(\overline{\Omega},\R^N\right)$ be an $F$-extremal and assume that the second variation at $u$ is strictly positive, meaning that there is a constant $c_3>0$ such that
\begin{equation}\label{secvar}
c_3\underset{\Omega}{\int}|\nabla\varphi|^2\dx\leq\underset{\Omega}{\int} F''(\nabla u)[\nabla\varphi,\nabla\varphi]\dx
\end{equation}
for all $\varphi\in \mathrm{Var}(\Omega,\R^N)$. In addition, if the free portion of the boundary is such that $\Gamma_N\neq\emptyset$, suppose that $\textup{(H2')}$ holds. In this case, or if $p>2$, further assume that, for some constants $c_4,c_5>0$, 
\begin{equation}
\tag{H3}c_4\underset{\Omega}{\int}|\nabla \varphi|^p\dx-c_5\underset{\Omega}{\int}|\nabla\varphi|^2\dx\leq\underset{\Omega}{\int}\left(F(\nabla u+\nabla\varphi)-F(\nabla u)\right)\dx
\end{equation} for every $\varphi\in\mathrm{Var}(\Omega,\R^N)$.
Then, $u$ is an $\LL^p$-local $F$-minimizer. 
\end{theorem}
\begin{remark}
If $p=2$ and $\Gamma_N=\emptyset$, the statement above provides a slightly stronger version of the original sufficiency theorem in \cite{GyM}, since it does not impose the coercivity condition (H3) on the integrand. However, it is not obvious whether this assumption can also be removed for the general superquadratic case.  

On the other hand, as discussed in \cite[S.3.2]{GyM}, it can be shown that (H3) also follows if we assume that $F$ is pointwise coercive: $c_4|z|^p-c_5\leq F(z)$ for all $z\in\R^{N\times n}$. However, (H3) is a more general assumption.
\end{remark}


\begin{remark}\label{RemarkSuffNonhom}
Under the hypotheses $(\mathrm{\overline{H}}0)$-$(\mathrm{\overline{H}}2)$ and a uniform continuity assumption (more general than $(\mathrm{HC})$ from Section \ref{SectBdryReg}), as well as by assuming the corresponding coercivity notion, namely 
\begin{equation}
\tag{$\mathrm{\overline{H}}3$}c_4\underset{\Omega}{\int}|\nabla \varphi|^p\dx-c_5\underset{\Omega}{\int}|\nabla\varphi|^2\dx\leq\underset{\Omega}{\int}\left(F(x,\nabla u+\nabla\varphi)-F(\nabla u)\right)\dx
\end{equation} for every $\varphi\in\mathrm{Var}(\Omega,\R^N)$,
the above sufficiency result also holds for  non-homogeneous integrands of the type $F(x,z)$ and extremals at which the second variation is positive, as in (\ref{uextremalVMO}) and (\ref{secvarposVMO}). The original statement of the result, as developed in \cite{GyM}, holds true for $\LL^\infty$-local minimizers of integrands of the type $F(x,u,z)$, so that $u$-dependence is also allowed. Further generalizations comprising integrands admitting lower order terms and non-smooth domains were recently obtained in \cite{JCCKK}. 
\end{remark}

We devote the rest of this section to prove Theorem \ref{theoGM} using the aforementioned alternative strategy. The first step will be to establish the two main ingredients of the proof, namely, a suitable version of Kristensen's Decomposition Theorem  and of Zhang's Theorem concerning spatially-local minimizers.

\subsection{The Decomposition Theorem}
Here we state the following result based on Kristensen's Decomposition Theorem \cite{KristensenLSC}. In its original form, the result states that given a uniformly bounded sequence in $\W^{1,p}$, it can be decomposed (up to a subsequence), into one part that carries the \textit{oscillations} and another carrying the \textit{concentrations}. The decomposition that we present here follows Kristensen's proof and provides a subtle extension by showing that if a bounded sequence $(u_j)$ in $\W^{1,2}$ is modified by multiplying each term by a (possibly different) scalar belonging to the unit interval $(0,1)$ and if the resulting subsequence, say $(\zeta_j)$, is bounded in $\W^{1,p}$ for some $p\geq 2$, then the respective decompositions in the spaces $\W^{1,2}$ and $\W^{1,p}$ can be obtained in such a way that they follow the same linear relations that $(u_j)$ and $(\zeta_j)$ satisfy. 

We remark that Kristensen's Decomposition Theorem is valid also for the case $1<p<2$, that we skip here although the proof remains the same than the one we will perform for the sequence $(u_j)$ in $\W^{1,2}$. In addition, we refer the reader to \cite[Theorem 8.1]{GyM}, in which Grabovsky and Mengesha modify the decomposition result from \cite{FonsPed} and on which the statement that we provide below is motivated. 

Finally, we emphasize that the reason we require the following version of the Decomposition Theorem is to control a sequence in $\W^{1,p}$ while normalizing it by both its $\W^{1,2}$ and its $\W^{1,p}$-norms in the proof of Theorem \ref{theoGM}.

\begin{theorem}[Decomposition Theorem]\label{DecLemma}
Let $\Omega\subseteq\R^n$ be a bounded Lipschitz domain and let $2\leq p<\infty$. Let $(u_j)$ be a sequence such that $u_j\rightharpoonup u$ in $\W^{1,2}(\Omega,\R^N)$, assume that $(r_j)$ is a sequence in $(0,1)$ and that $\zeta_j:=r_ju_j$ is bounded in $\W^{1,p}$. Then, there exist a subsequence $(u_{j_k})$ and sequences $(g_k)\subseteq C_c^\infty(\Omega,\R^N)$, $(b_k)\subseteq \W^{1,2}(\Omega,\R^N)$ such that 
\begin{itemize}
\item[\textup{(a)}] $g_k\rightharpoonup 0$ and $b_k\rightharpoonup 0$ in $\W^{1,2}(\Omega,\R^N)$;
\item[\textup{(b)}] $(|\nabla g_k|^2)$ is equiintegrable;
\item[\textup{(c)}] $\nabla b_k\rightarrow 0$ in measure and
\item[\textup{(d)}] $u_{j_k}=u+g_k+b_k$.
\end{itemize}
In addition, $(g_k)$ and $(b_k)$ can be taken so that, for a subsequence $(r_{k_j})$, if $s_k:=r_{k_j}g_k$ and $t_k:=r_{j_k}b_k$, then 
\begin{itemize}
\item[\textup{(a')}] $s_k\rightharpoonup 0$ and $t_k\rightharpoonup 0$ in $\W^{1,p}(\Omega,\R^N)$;
\item[\textup{(b')}] $(|\nabla s_k|^p)$ is equiintegrable and 
\item[\textup{(c')}] $\nabla t_k\rightarrow 0$ in measure.
\end{itemize}
\end{theorem}
In order to establish this result following Kristensen's proof, we state first the Helmholtz Decomposition Theorem and an auxiliary lemma. 
\begin{theorem}[Helmholtz Decomposition Theorem]
Let $1<p<\infty$ and denote
\begin{equation*}
\overset{\circ}{W}\,^{1,p}:=\left\{\varphi\in \W^{1,p}_{loc}(\R^n)\st \nabla\varphi\in \LL^p(\R^n,\R^n)\right\}
\end{equation*}
the homogeneous Sobolev space. 
Let
\begin{equation*}
E^p:=\left\{\nabla\varphi\st\varphi\in \overset{\circ}{W}\,^{1,p}\right\}
\end{equation*}
and
\begin{equation*}
B^p:=\left\{\sigma\in \LL^p(\R^n,\R^n)\st \mathrm{div}\sigma=0\mbox{ in the distributional sense}\right\}.
\end{equation*}
Then, $E^p$ and $B^p$ are closed subspaces of $\LL^p=\LL^p(\R^n,\R^n)$ such that $E^p\cap B^p=\{0\}$. Furthermore, there exist bounded projections $\mathbb{E}: \LL^p\rightarrow E^p$, $\mathbb{B}:\LL^p\rightarrow B^p$, so that $\mathrm{Id}_{\LL^p}=\mathbb{E}+\mathbb{B}$. In other words, for any $v\in \LL^p$, there exist $\varphi\in \overset{\circ}{W}\,^{1,p}$ and $\sigma\in \LL^p$ with $\mathrm{div}\sigma=0$, with the property that
\begin{equation}
v=\nabla\varphi+\sigma
\end{equation}
and where $\|\nabla\varphi\|_{\LL^p}\leq C_p\|v\|_{\LL^p}$, $\|\sigma\|_{\LL^p}\leq C_p\|v\|_{\LL^p}$ for some constant $C_p>0$. 

In addition, if $v\in \LL^2(\R^n,\R^n)\cap \LL^p(\R^n,\R^n)$, then the decomposition of $v$ in the space $\LL^2(\R^n,\R^n)$ coincides with the decomposition of $v$ in the space $\LL^p(\R^n,\R^n)$.
\end{theorem}
The following result follows easily from the definition of $p$-equiintegrability. 
\begin{lemma}\label{lemmaequiint}
Let $\Omega\subseteq\R^n$ such that $\mathcal{L}^n(\Omega)<\infty$, $1\leq p< q\leq\infty$ and let $(f_j)\subseteq \LL^p(\Omega,\R^m)$ be bounded in $\LL^p$. Then, $(f_j)$ is $p-equiintegrable$ if and only if, for each $\varepsilon>0$, there exist a sequence $(g_j)\subseteq \LL^q(\Omega,\R^m)$ and a constant $c_\varepsilon>0$ such that, for all $j\in\N$, $\|f_j-g_j\|_{\LL^p}<\varepsilon$ and $\|g_j\|_{\LL^q}\leq c_\varepsilon$.
\end{lemma}
We now proceed with the proof of the Decomposition Theorem following  \cite[Lemma 1.7]{KristensenLSC}.
\\
\begin{proof}[Proof of Theorem \ref{DecLemma}]
By considering suitable subsequences, that we do not relabel, we assume without loss of generality that ${\nabla u_j\overset{Y}{\rightarrow}(\nu_x)}$ and ${\nabla \zeta_j\overset{Y}{\rightarrow}(\mu_x)}$.

In addition, we observe that part (c') of the Theorem follows directly from part (c) and the fact that $r_j\in (0,1)$. 

Furthermore, we observe (as in \cite[Theorem 8.1]{GyM}) that (b') implies that the sequence $(s_k)$ is bounded in $\W^{1,p}$ and, since $\zeta_k=r_{j_k}u+s_k+t_k$, then we also have that $(t_k)$ is bounded in $\W^{1,p}$, following our initial assumption on $\zeta_k$. This, together with (a), means that there exist subsequences of $s_k$ and $t_k$, that we do not relabel,  such that they satisfy condition (a'). 

Given this, we are only left with establishing parts (a)-(d) and (b') of the theorem. 
\\
\\
\textit{Step 1.} Observe that, by working with the sequence $u_j-u$ instead of $u_j$, we can assume that $u=0$. 
\\
\\
\textit{Step 2.} \textit{Reduction of the problem to $(u_j)\subseteq \W^{1,2}_0(\Omega,\R^N)$.} We begin by taking a sequence of smooth subdomains $\Omega_k\Subset\Omega_{k+1}\Subset\Omega$ such that $\bigcup_{k\in\N}\Omega_k=\Omega$. In addition, we consider cut-off functions $\rho_k\colon\Omega\rightarrow[0,1]$ with $\rho_k\in C^1_c(\Omega)$, $\mathbbm{1}_{\Omega_k}\leq\rho_k\leq \mathbbm{1}_{\Omega}$ and $|\nabla\rho_k|\leq\frac{1}{d_k}$, where $d_k:=\mathrm{dist}(\Omega_k,\Omega)$.
\\
Now, observe that for any $j,k\in\N$, $u_j=\rho_k u_j+(1-\rho_k)u_j$. Hence, for $q\geq 1$, 
\begin{equation}\label{ectruncdecLemma}
\underset{\Omega}{\int}|\nabla((1-\rho_k)u_j)|^q\dx\leq c_q\underset{\Omega-\Omega_k}{\int}\left(|\nabla u_j|^q+\frac{|u_j|^q}{d_k^q}\right)\dx.
\end{equation}
Since $u_j\rightharpoonup 0$ in $\W^{1,2}(\Omega,\R^N)$, in particular we know that $u_j{\rightarrow}\,0$ in $\LL^2$. Hence, we can find a sequence $k_j\rightarrow\infty$ such that 
\begin{equation*}
\underset{\Omega-\Omega_{k_j}}{\int}\frac{|u_j|^2}{d_{k_j}^2}\dx\rightarrow 0
\end{equation*}
when $j\rightarrow\infty$. 
Furthermore, since $(\nabla u_j)$ is bounded in $\LL^2$, by  adjusting the sequence $k_j$ if necessary, from (\ref{ectruncdecLemma}) we can also ensure that $(1-\rho_{k_j})u_j$ is bounded in $\W^{1,2}$. In an analogous way we can also assume that $(1-\rho_{k_j})r_ju_j$ is bounded in $\W^{1,p}$. 

In addition, we also infer that $(1-\rho_{k_j})u_j\rightarrow 0$ in $\LL^2$. The last two facts imply together that $(1-\rho_{k_j})u_j\rightharpoonup 0$ in $\W^{1,2}(\Omega,\R^N)$.

What is more, using again that $|\nabla((1-\rho_{k_j})u_j)|\leq \left(|(1-\rho_{k_j})\nabla u_j|+\frac{|u_j|}{d_{k_j}} \right) \mathbbm{1}_{\Omega-\Omega_{k_j}} $, we infer that this sequence converges to $0$ in measure.  Given this, we focus on decomposing $(\rho_{k_j}u_j)$, as we can then incorporate $(1-\rho_{k_j})u_j$ into the sequence $(b_k)$. 
\\
\\
\textit{Step 3.} \textit{Truncation.} We now define, for each $k\in\N$, the truncation $T_k\colon\R^{N\times n}\rightarrow\R^{N\times n}$ at level $k$ by
\begin{equation*}
T_k(z):=\left\{
\begin{array}{ccc}
z& \mbox{if} &|z|\leq k\\
k\frac{z}{|z|} &\mbox{if}& |z|>k.
\end{array}
\right.
\end{equation*}
It is clear that $T_k$ is continuous and $|T_k(z)|\leq k$ for all $z\in\R^{N\times n}$. 
\\
Observe that, by the Fundamental Theorem for Young Measures and the Monotone Convergence Theorem, 
\begin{align*}
\underset{k\rightarrow\infty}{\lim}\,\underset{j\rightarrow\infty}{\lim}\underset{\Omega}{\int}|T_k(\nabla u_j)|^2\dx &=\underset{k\rightarrow\infty}{\lim} \underset{\Omega}{\int}\underset{\R^{N\times n}}{\int} |T_k(z)|^2\dv{}\nu_x(z)\dx\\
&=\underset{\Omega}{\int}\underset{\R^{N\times n}}{\int} |\cdot|^2\dv{}\nu_x\dx.
\end{align*}
We also have that
\begin{align*}
\underset{k\rightarrow\infty}{\lim}\,\underset{j\rightarrow\infty}{\lim}\underset{\Omega}{\int}|T_k(\nabla u_j)-\nabla u_j|\dx&\leq \underset{k\rightarrow\infty}{\lim}\,\underset{j\in\N}{\sup}\underset{|\nabla u_j|>k}{\int}2|\nabla u_j|\dx=0.
\end{align*}
Given these, we can now take a subsequence $j_k\rightarrow \infty$ such that
\begin{equation}\label{vkequiint}
\underset{k\rightarrow\infty}{\lim}\underset{\Omega}{\int}|T_k(\nabla u_{j_k})|^2\dx=\underset{\Omega}{\int}\langle|\cdot|^2,\nu_x\rangle\dx
\end{equation}
and 
\begin{equation}\label{vkgenerates}
\underset{k\rightarrow\infty}{\lim}\underset{\Omega}{\int}|T_k(\nabla u_{j_k})-\nabla u_{j_k}|\dx=0.
\end{equation}
Let $v_k:=T_k\circ\nabla u_{j_k}$. It follows from equation (\ref{vkequiint}) and the Fundamental Theorem for Young Measures, that $(v_k)$ is $2$-equiintegrable and, from (\ref{vkgenerates}), we deduce that $(v_k)$ also generates the Young Measure $(\nu_x)$. Furthermore, observe that using (\ref{vkequiint}), (\ref{vkgenerates}), Vitali's Convergence Theorem and de la Vall\'{e}e Poussin criterion for equiintegrability, we can also conclude that $(T_k\circ\nabla u_{j_k}-\nabla u_{j_k})\rightarrow 0$ strongly in $\LL^q(\Omega,\R^{N\times n})$ for $q<2$. 
\\
\\
\textit{Step 4.} \textit{Helmholtz decomposition.} Since $u_{j_k}\in \W^{1,2}_0(\Omega,\R^N)$, we can extend $u_{j_k}$ to $\R^{n}$ by $0$ while still having $u_{j_k}\in \overset{\circ}{W}\,^{1,2}(\R^n,\R^N)$. We also extend $v_k\colon\R^n\rightarrow\R^{N\times n}$ so that it is $0$ outside of $\Omega$.

Now, we apply row-wise Helmholtz decomposition in $\LL^2(\R^n,\R^n)$ and we obtain functions ${\tilde{g}_k\in \overset{\circ}{W}\,^{1,2}}$, $\tilde{\sigma}_k\in \LL^2(\R^n,\R^{N\times n})$ such that:
\begin{itemize}
\item[$\mathrm{(i)}$] $\mathbb{E}(v_k)=\nabla\tilde{g}_k$, $\mathbb{B}(v_k)=\tilde{\sigma}_k$ defined row-wise and, hence, $v_k=\nabla\tilde{g}_k+\tilde{\sigma}_k$;
\item[$\mathrm{(ii)}$] $\mathrm{div}\,\tilde{\sigma}_k=0$ and
\item[$\mathrm{(iii)}$] $\underset{\Omega}{\int}\tilde{g}_k=0$, which we can achieve by subtracting a constant from $\tilde{g}_k$ if necessary. 
\end{itemize}
We now claim that $\tilde{\sigma}_k\rightarrow 0$ in measure when $k\rightarrow\infty$. Indeed, since $v_k-\nabla u_{j_k}\rightarrow 0$ in $\LL^q(\Omega,\R^{N\times n})$ for every $q\in(1,2)$, we get that
\begin{equation}
\|\tilde{\sigma}_k\|_{\LL^q}=\|\mathbb{B}(v_k)\|_{\LL^q}=\|\mathbb{B}(v_k-\nabla u_{j_k})\|_{\LL^q}\leq c_q\|v_k-\nabla u_{j_k}\|_{\LL^q}\rightarrow 0, 
\end{equation}
where $c_q>0$ is a constant depending only on the continuity of $\mathbb{B}$. 

We now proceed to prove that $\nabla\tilde{g}_k$ is $2$-equiintegrable on $\Omega$. Observe first that, since $v_k$ is $2$-equiintegrable, Lemma \ref{lemmaequiint} implies that, for a given $\varepsilon>0$, we can find a sequence $(w_k)\subseteq \LL^{3}(\Omega,\R^{N\times n})$ such that, for all $k\in\N$, $\|v_k-w_k\|_{\LL^2}<\varepsilon$ and $\|w_k\|_{\LL^{3}}\leq c_\varepsilon$. On the other hand, by making $v_k=w_k=0$ off $\Omega$, we can also apply Helmholtz decomposition to $w_k$ to conclude that
\begin{equation*}
\|\nabla\tilde{g}_k-\mathbb{E}(w_k)\|_{\LL^2}\leq c\|v_k-w_k\|_{\LL^2}<c\varepsilon
\end{equation*}
and $\|\mathbb{E}(w_k)\|_{\LL^{3}}\leq c_\varepsilon$ for every $k\in\N$. Since $\varepsilon>0$ was arbitrary, by using Lemma \ref{lemmaequiint} once again, we conclude the proof of our claim. 

Observe that, since $\nabla\tilde{g}_k-\nabla u_{j_k}=v_k-\nabla u_{j_k}-\tilde{\sigma}_k=(v_k-\nabla u_{j_k})-\tilde{\sigma}_k$, we have that $\nabla\tilde{g}_k-\nabla u_{j_k}\rightarrow 0$ in measure. In addition, (\ref{vkequiint}) implies that $\nabla\tilde{g}_k$ is bounded in $\LL^2(\Omega,\R^{N\times n})$. Hence,  we can assume that $\nabla\tilde{g}_k-\nabla u_{j_k}\rightharpoonup 0$ in $\LL^2(\Omega,\R^{N\times n})$ and, therefore, the same holds for $(\nabla\tilde{g}_k)$.

On the other hand, because $\underset{\Omega}{\int}\tilde{g}_k=0$, by Poincar\'e's inequality and Rellich-Kondrachov Embedding Theorem we can conclude that there is $g\in \W^{1,2}(\Omega,\R^N)$ such that $\tilde{g}_k\rightharpoonup g$ in $\W^{1,2}(\Omega,\R^N)$ and, by the observations above, we further have $\tilde{g}_k\rightharpoonup 0$ in $\W^{1,2}(\Omega,\R^N)$. 

Now, we consider again the sequence of domains $(\Omega_l)$. Observe that, since $\tilde{g}_k\rightarrow 0$ in $\LL^2(\Omega,\R^N)$, we can find a subsequence $\Omega_{l_k}$ such that, if $g_k:=\rho_{l_k}\tilde{g}_k$, then 
\begin{equation*}
\nabla g_k=\rho_{l_k}\nabla\tilde{g}_k+\tilde{g}_k\otimes\nabla\rho_{l_k}
\end{equation*}
is $2$-equiintegrable.
In addition, since $\nabla\rho_{l_k}=0$ in $\Omega_{l_k}$, it follows that, if $b_k:=u_{j_k}-g_k$, then 
\begin{align*}
\nabla b_k&=\nabla u_{j_k}-\nabla g_k=(\nabla u_{j_k}-v_k)+(v_k-\rho_{l_k}v_k)+\rho_{l_k}(v_k-\nabla\tilde{g}_k)-\tilde{g}_k\otimes\nabla\rho_{l_k}\\
&=(\nabla u_{j_k}-v_k)+(1-\rho_{l_k})v_k+\rho_{l_k}\tilde{\sigma}_k-\tilde{g}_k\otimes\nabla\rho_{l_k}
\end{align*} 
converges to $0$ in measure. 
Then, $g_k$ and $b_k$ are the desired functions and, since  $g_k\in \W^{1,2}_0(\Omega,\R^N)$, we can further assume that $(g_k)\in C^\infty_c(\Omega,\R^N)$, with which we complete the proof of parts (a)-(d) of the Theorem. 

We now proceed with the proof of part (b'). Arguing exactly as we did in (\ref{vkequiint}) from Step 3 and performing, once again, a slight abuse on the notation by not relabelling the corresponding subsequence, we can assume that 
\begin{equation}\label{forrkvkequiint}
\underset{k\rightarrow\infty}{\lim}\underset{\Omega}{\int}|T_k(\nabla \zeta_{j_k})|^p\dx=\underset{\Omega}{\int}\langle|\cdot|^p,\mu_x\rangle\dx.
\end{equation}
In addition, we observe that, for every $x\in\Omega$ and with $v_k=T_k\circ\nabla u_{j_k}$ as in Step 3,  it holds that
\begin{equation}\label{comprkvkwithTk}
\left|r_{j_k}v_k(x)\right|=\left|r_{j_k}T_k\circ \nabla u_{j_k}(x)\right|=\left|T_{r_{j_k}k}(r_{j_k}\nabla u_{j_k}(x))\right|\leq \left|T_k(r_{j_k}\nabla u_{j_k})\right|=\left|T_k(\nabla\zeta_{j_k})\right|.
\end{equation}
We are using here the elementary identity $rT_k(\xi)=T_{rk}(r\xi)$ and the facts that $r_{j_k}k\leq k$ and  $k\mapsto |T_k(\xi)|$ is non-decreasing for every $\xi\in\R^{N\times n}$. 

It follows from (\ref{forrkvkequiint}) and (\ref{comprkvkwithTk}) that the sequence $(r_{j_k}v_k)$ is $p$-equiintegrable and, in particular, it is also bounded in $\LL^p$. 

Furthermore, if $\mathbb{E}(v_k)=\nabla \tilde{g}_k$, then clearly $\mathbb{E}(r_{j_k}v_k)=r_{j_k}\nabla \tilde{g}_k$  by decomposing both $v_k$ and $r_{j_k}v_k$ in $\LL^2(\Omega,\R^N)$. In addition, since $r_{j_k}v_k\in \LL^2(\Omega,\R^N)\cap\LL^p(\Omega,\R^N)$, Helmholtz Decomposition Theorem enables us to ensure that, for some constant $c_p>0$, 
\begin{equation*}
\|r_{j_k}\nabla \tilde{g}_k\|_{\LL^p}=\|\mathbb{E}(r_{j_k}v_k)\|_{\LL^p}\leq c_p\|r_{j_k}v_k\|_{\LL^p}.
\end{equation*}
This inequality, together with the fact that $(r_{j_k}v_k)$ is bounded in $\LL^p(\Omega,\R^N)$, enables us to conclude that, for a subsequence that we do not relabel, $r_{j_k}\nabla \tilde{g}_k$ converges weakly in $\W^{1,p}(\Omega,\R^N)$ and, arguing exactly as we did with $\tilde{g}_k$, we can further deduce that 
\begin{equation}\label{rkgtildekconv}
r_{j_k}\tilde{g}_k\rightharpoonup 0 \mathrm{\hspace{7mm}in} \hspace{7mm}\W^{1,p}(\Omega,\R^N).
\end{equation}

To conclude the proof of (b') it is enough to observe that 
\begin{equation*}
\nabla s_k=r_{j_k}\nabla g_{j_k}=r_{j_k}\rho_{j_k}\nabla\tilde{g}_{j_k}+r_{j_k}\tilde{g}_{j_k}\otimes\nabla\rho_{j_k}
\end{equation*}
and use the fact that $(T_k(\nabla\zeta_{j_k}))$ is $p$-equiintegrable, together with (\ref{comprkvkwithTk}) and (\ref{rkgtildekconv}), to proceed as we did to show that $\nabla g_k$ is $2$-equiintegrable and construct that way a subsequence of $(s_k)$ such that $\nabla s_k$ is $p$-equiintegrable. 
\end{proof}

\subsection{Spatially-local minimizers: Zhang's Theorem}\label{SectionZhang}
In this section we establish a generalization of a theorem by K. Zhang in \cite{Zhang}. The original result states that smooth extremals are all spatially-local minimizers in a strict sense under Dirichlet boundary conditions. The generalization that we present here allows part of the boundary to take free values. We also remove the assumption on uniform continuity of the second derivative of the integrand. However, the essence of the proof remains the same as the the one in \cite{Zhang}. We state the result in this more general way aiming at using it for the new proof of the sufficiency result. 

Furthermore, we remark that the main idea behind Zhang's Theorem is that, if an extremal is smooth, in small subsets of its domain it is close enough to an affine function (in an uniform way). Therefore, we can exploit the strong quasiconvexity assumption on the integrand, according to which affine functions minimize the integrand under the corresponding affine boundary conditions, to obtain minimality in a local sense in space. 
\begin{theorem}\label{teo1Zhang}
Let $F\colon\R^{N\times n}\rightarrow \R$ satisfy $\mathrm{(H0)-(H2)}$ and $\mathrm{(H2')}$ for some $1<p<\infty$ and assume that $\Omega\subseteq \R^n$ is a $C^1$ bounded domain. If $u\in C^1(\overline{\Omega},\R^N)$ is an $F-extremal$,
then there exists $\overline{R}>0$ such that, for every $x_0\in\overline{\Omega}$,
\begin{equation}\label{ecteo1}
\frac{c_2}{2}\underset{\Omega(x_0,\overline{R})}{\int}\left|V(\nabla\varphi)\right|^2\dx\leq \underset{\Omega(x_0,\overline{R})}{\int}\left(F(\nabla u+\nabla\varphi)-F(\nabla u)\right)\dx
\end{equation}
whenever $\varphi\in \mathrm{Var}(\Omega(x_0,\overline{R}),\R^N)$.
%
%
%
\end{theorem}

\begin{remark}\label{Zhangwithcubes}
Let $\Omega_Q(x_0,r):=\Omega\cap Q(x_0,r)$, where $Q(x_0,r)$ is a cube with sides parallel to the coordinate axes. It is then easy to see that, if $\varphi\in\mathrm{Var}\left(\Omega_Q\left(x_0,\frac{\overline{R}}{2}\right)\right)$, then by assigning $\varphi$ the value of $0$ in $\Omega(x_0,\overline{R})\backslash\Omega_Q(x_0,\frac{\overline{R}}{2})$, we can assume that $\varphi\in \mathrm{Var}(\Omega(x_0,\overline{R}))$. Therefore, Theorem \ref{teo1Zhang} remains valid if we exchange $\Omega(x_0,\overline{R})$ by $\Omega_Q(x_0,\overline{R})$ in the statement.
%
\end{remark}

The proof of Theorem \ref{teo1Zhang} relies on the following growth and continuity estimates, for which the regularity assumed on $u$ plays a central role. 
\begin{lemma}\label{lemmaestimagesG}
Let $F\colon\R^{N\times n}\rightarrow \R$ satisfy  $\mathrm{(H0)-(H1)}$ as well as the Legendre-Hadamard condition.\footnote{$F\in C^2(\R^{N\times n})$ is said to satisfy the Legendre-Hadamard condition if for every  $\xi=(\xi_\alpha^i)\in\R^{N\times n}$, $\lambda\in\R^N$ and $\mu\in\R^n$, $\sum_{1=i,j}^N\sum_{1=\alpha,\beta}^  n \frac{\partial^2F(\xi)}{\partial\xi_\alpha^i\partial\xi_\beta^j}\lambda^i\lambda^j\mu_\alpha\mu_\beta\geq 0$. See also \cite[Theorem 5.3]{dacor}.} Let $u\in C^1(\overline{\Omega},\R^N)$ be an $F$-extremal and define the functional ${G\colon\overline{\Omega}\times\R^{N\times n}\rightarrow\R}$ by 
\begin{eqnarray*}
G(x,z)&:=&F(\nabla u(x)+z)-F(\nabla u(x))-\left\langle F'(\nabla u(x)),z\right\rangle\\
&=&\int_0^1(1-t)F''(\nabla u(x)+tz)[z,z]\dt.
\end{eqnarray*} 
Then, the following estimates remain true:
\begin{itemize}
\item[(a)] $|G(x,z)|\leq C_0|V(z)|^2$.
\item [(b)] For every $\varepsilon>0$  there is an $R=R_\varepsilon>0$ such that, for all $x\in\Omega(x_0,R)$, $J\in\R^{n\times n}$ and $z\in\R^{N\times n}$, if $w=z J$ and $|J-\mathrm{I}_n|<R$, then 
\begin{equation}\label{quaderror}
E:=|G(x_0,w)-G(x,z)|<\frac{\varepsilon}{4} |V(z)|^2.
\end{equation} 
\item[(c)] For any $\varepsilon>0$ there is a $\delta=\delta_\varepsilon\in (0,\min\{\frac{\varepsilon}{2},1\})$ such that, if $J\in\R^{n\times n}$ satisfies $|J-\mathrm{I}_n|<\delta$, then 
\begin{equation*}
c_2\left||V(z J)|^2-|V(z)|^2|\det J|\right|\leq  \,\frac{\varepsilon}{4} |V(z)|^2.
\end{equation*}
\end{itemize}
\begin{proof}
We first prove part (b) of the Lemma using the truncation strategy originated in \cite[Lemma II.3]{AcFus}. We observe that part (a) can be shown following the same ideas and it is in fact less technical, so we omit the proof. 

Observe first that, since $u\in C^1(\overline{\Omega},\R)$, $\nabla u$ is uniformly continuous and bounded in $\overline{\Omega}$. 

We establish (\ref{quaderror}) by considering the following two cases. 
\\
\textit{Case 1.} If $|z|\leq 1$ and $w=zJ$ with $|J-\mathrm{I}_n|<1$ then, by the local uniform continuity of $F''$, we can find a modulus of continuity, say $\omega\colon[0,\infty)\rightarrow [0,1]$, such that it is increasing, continuous, $\omega (0)=0$ and for which there is a constant $c>0$ with the property that
\begin{equation*}
|F''(\nabla u(x)+tz)-F''(\nabla u(x_0)+tw)|\leq c\,\omega(|x-x_0|+|z-w|)
\end{equation*}
for all $x,x_0\in \overline{\Omega}$, $t\in [0,1]$, $|z|\leq 1$ and $w=z J$ with $|J-\mathrm{I}_n|<1$.
We can further assume that $c$ is such that 
\begin{equation}\label{condoverc}
1+n+|F''(\nabla u(x)+tz)|\leq c
\end{equation} 
for all $x\in \overline{\Omega}$, $t\in [0,1]$ and $|z|\leq 1$. 
\\Using this,  we fix $0<\varepsilon<1$ and observe that
\begin{align*}
&|G(x_0,w)-G(x,z)|\leq  \,|G(x_0,w)-G(x_0,z)|+|G(x_0,z)-G(x,z)|\\
\leq & \,c\,\omega(|w-z|)|w|^2+\int_0^1|F''(\nabla u(x_0)+tz)||w-z|(|w|+|z|)\dt+c\,\omega(|x-x_0|)|z|^2\\
\leq & \,c\,\omega(|J-\mathrm{I}_n|)|J|^2|V(z)|^2+c\,|J-\mathrm{I}_n|(|J||V(z)|^2+|V(z)|^2)+c\,\omega(|x-x_0|)|V(z)|^2\\
\leq & \,\frac{\varepsilon}{4} |V(z)|^2, 
\end{align*}
where the last inequality is making use of (\ref{condoverc}) and it holds provided that $|J-\mathrm{I}_n|$, $\omega(|J-\mathrm{I}_n|)$ and $\omega(|x-x_0|)$ are small enough. Notice that, if that is the case, we can assume $|J|<c(n)$ for a constant $c(n)>0$.

Thus, for $|z|\leq 1$ we have that, if $R\in (0,1)$ is such that $c\,\omega(R)|J|^2<\frac{\varepsilon}{8}$, then for every  $x\in \Omega(x_0,R)$, $E\leq \frac{\varepsilon}{4} |V(z)|^2$, provided also that $|J-\mathrm{I}_n|<R$. 
\\
\textit{Case 2.} For $|z|>1$, we will need to make use of the Lipschitz bounds for $F$ that are derived from $\mathrm{(H0)-(H1)}$ and the Legendre-Hadamard condition (see \cite[Proposition 2.32]{dacor}). Following this, and the fact that $F'$ is also locally uniformly continuous, we have for $x\in\Omega(x_0,R)$ that 
\begin{align*}
&|G(x_0,w)-G(x,z)|\leq  \,|G(x_0,w)-G(x_0,z)|+|G(x_0,z)-G(x,z)|\\
\leq & |F(\nabla u (x_0)+w)-F(\nabla u(x_0)+z)|+|F'(\nabla u(x_0))||w-z|+|F(\nabla u(x_0)+z)-F(\nabla u(x)+z)|\\
&\,+|F(\nabla u(x_0))-F(\nabla u(x))|+|F'(\nabla u(x_0))-F'(\nabla u (x))||z|\\
\leq & \,c\, (1+|z|^p)|J-\mathrm{I}_n| + C (1+|z|+|z|^p)|\nabla u(x)-\nabla u(x_0)|\\
\leq & \,C\,(|J-\mathrm{I}_n|+(\text{osc}_{\Omega(x_0,R)}\nabla u))|V(z)|^2,
\end{align*}
where the second inequality follows after using the local Lipschitz continuity satisfied by $F$ and that $a^{p-1}b\leq \frac{a^p}{p}+\frac{b^p}{p}$ for $a,b>0$. The last inequality is a consequence from the fact that $|z|>1$.

Therefore, if for a given $\varepsilon>0$ we take $R>0$ such that
\begin{center}
$\begin{array}{lcr}
C \,(|J-\mathrm{I}_n|+(\text{osc}_{\Omega(x_0,R)}\nabla u))<\frac{\varepsilon}{4},\\
\end{array}$
\end{center}
our claim follows by choosing $R=R_\varepsilon>0$ suitable to make $E\leq \frac{\varepsilon}{4}|V(z)|^2$ for any $z\in\R^{N\times n}$. 
\\

For the proof of (c) observe that, by continuity of the determinant, 
for any given $C,\varepsilon>0$ there is a $\delta=\delta_\varepsilon\in (0,\min\{\frac{\varepsilon}{2},1\})$ such that, if $|J-\mathrm{I}_n|<\delta$ with $J\in\R^{n\times n}$, we can then ensure $|J-\mathrm{I}_n|+|1-|\det(J)||<\frac{\varepsilon}{4C}$. This technical observation enables us to estimate, for any $z\in\R^{N\times n}$ and any $J\in\R^{n\times n}$ with $|J-\mathrm{I}_n|<\delta$ as above, that
\begin{align}
&c_2\left||V(z J)|^2-|V(z)|^2|\det J|\right|\notag\\
\leq &\,C\left(|V(z J)|+|V(z)|\right)\left||V(z J|-|V(z)|\right|+C|V(z)|^2|1-|\det J||\notag\\
\leq &\,C|V(z)||z||J-\mathrm{I}_n|\left(1+|z J|^2+|z|^2\right)^{\frac{p-2}{4}}+C|V(z)|^2|1-|\det J||\label{prelimauxineq}\\
\leq &\, C\left(|J-\mathrm{I}_n||V(z)|^2+|V(z)|^2|1-|\det J||\right)\notag\\
\leq & \,\frac{\varepsilon}{4} |V(z)|^2.
\end{align}
We remark that inequality (\ref{prelimauxineq}) follows after applying Lemma \ref{lemmapropsV} (i)-(v), together with the fact that we can assume $|J|\leq C$, given that $|J-\mathrm{I}_n|<\delta<1$. 
\end{proof}

\end{lemma}

\begin{proof}[Proof of Theorem \ref{teo1Zhang}]
The main idea behind the proof will be an appropriate use of the quasiconvexity conditions  (H2) and (H2').

We first use Lemma \ref{lemmaestimagesG} to conclude that, for every $\varepsilon>0$ there are $0<\delta=\delta_\varepsilon<1$ and $\rho=\rho_\varepsilon>0$ such that, for every $J\in\R^{n\times n}$, $z\in\R^{N\times n}$ and  $x_0,x\in\R^n$, if $|J-\mathrm{I}_n|<\delta$ and $x\in\Omega(x_0,\rho)$, then 
\begin{align}
&|G(x_0,z\cdot J)-c_2|V(z\cdot J)|^2-(G(x,z)|\det J|-c_2|V(z)|^2|\det J|)|\notag\\
\leq&|G(x_0,z\cdot J)-G(x,z)|+|G(x,z)||1-|\det J||+c_2||V(z\cdot J)|^2-|V(z)|^2|\det J||\notag\\
\leq& \frac{\varepsilon}{2}|V(z)|^2\notag\\
\leq& \varepsilon |V(z)|^2|\det J|,\label{mainestimate}
\end{align}
where the last inequality follows from the local uniform continuity of the determinant and from the assumption that $|J-\mathrm{I}_n|<\delta$ for $0<\delta<\frac{\varepsilon}{2}$. 

Having obtained this preliminary estimate, we first show the theorem for the case in which $x_0\in \Gamma_N$. 

Then, because $\Omega$ is a set of class $C^1$, we can find an $R_0>0$, which does not depend on $x_0$, such that for every $0<r<R_0$ there is a diffeomorphism $\Phi_r\colon B_{\mathrm{n}(x_0)}^-(0,1)\rightarrow \left(\frac{\Omega-x_0}{r}\cap B(0,1)\right)$. Although $\Phi_r$ depends on $x_0$, the estimates we obtain from it are uniform on $x_0\in\Omega$: because $\partial\Omega$ is smooth and compact, we can construct the diffeomorphisms $\Phi_r$ so that, given a $\delta>0$, we can find an $R_1\in(0,R_0)$ such that for every $0<r\leq R_1$ and for every $x_0\in\partial\Omega$, 
\begin{equation}\label{eqhomeom}
\left\|\Phi_r-Id_{B_{\mathrm{n}(x_0)}^-(0,1)}\right\|_{\LL^\infty(B_{\mathrm{n}(x_0)}^-(0,1),\R^n)}+\left\|\nabla\Phi_r-\mathrm{I}_n\right\|_{\LL^\infty(B_{\mathrm{n}(x_0)}^-(0,1),\R^{n\times n})}<\delta,
\end{equation}
so that $\nabla\Phi_r$ converges to the identity matrix uniformly on $B_{\mathrm{n}(x_0)}^-(0,1)$ and uniformly for $x_0\in\partial\Omega$.\footnote{See Theorem C.1 in \cite{GyM} for a careful construction of the diffeomorphisms $\Phi_r$.}
\\
Having established the above estimates, after fixing $\varepsilon>0$ we obtain $\delta>0$ and $\rho>0$ such that (\ref{mainestimate}) is satisfied and, for such $\delta>0$, we take $R_1$ so that ($\ref{eqhomeom}$) holds. We further assume that $\left|\left|\Phi_{R_1}\right|\right|_{\LL^\infty(B_{\mathrm{n}(x_0)}^-(0,1))}\leq 2$. We now let $R:=\frac{1}{2}\min\{\rho,R_1\}$ and observe that, for any  $y\in B_{\mathrm{n}(x_0)}^-(0,1)$, we have $R\Phi_R(y)+x_0\in \Omega(x_0,R)$ and, therefore, (\ref{mainestimate}) holds with $x:=R\Phi_R(y)+x_0$, $z:=\nabla\varphi(R\Phi_R(y)+x_0)$ and $J:=\nabla\Phi_R(y)$, where $\varphi$ is any function in $\W^{1,p}(\Omega(x_0,R),\R^N)$ so that $\varphi=0$ on $\partial(B(x_0,R))\cap\Omega$.

After making this substitution in (\ref{mainestimate}), and since the inequality holds for every $y\in B_{\mathrm{n}(x_0)}^-(0,1)$, we denote $\tilde{y}_{R,x_0}:=R\Phi_R(y)+x_0$ and integrate over $B_{\mathrm{n}(x_0)}^-(0,1)$ to obtain that

\begin{align}
& \int_{B_{\mathrm{n}(x_0)}^-(0,1)}\left(F(\nabla u(x_0)+\nabla\varphi(\tilde{y}_{R,x_0})\cdot \nabla\Phi_R(y))-F(\nabla u(x_0))\right)\dy\notag\\
&-\int_{B_{\mathrm{n}(x_0)}^-(0,1)}\left\langle F'(\nabla u(x_0)),\nabla\varphi(\tilde{y}_{R,x_0})\cdot \nabla\Phi_R(y)\right\rangle\dy -c_2\int_{B_{\mathrm{n}(x_0)}^-(0,1)}|V(\nabla\varphi(\tilde{y}_{R,x_0})\cdot \nabla\Phi_R(y))|^2\dy\notag\\
&-\int_{B_{\mathrm{n}(x_0)}^-(0,1)}F(\nabla u(\tilde{y}_{R,x_0})+\nabla\varphi(\tilde{y}_{R,x_0}))|\det \nabla\Phi_R(y)|\dy+\int_{B_{\mathrm{n}(x_0)}^-(0,1)}F(\nabla u(\tilde{y}_{R,x_0}))|\det \nabla\Phi_R(y)|\dy\notag\\
&+\int_{B_{\mathrm{n}(x_0)}^-(0,1)}\left\langle F'(\nabla u (\tilde{y}_{R,x_0})),\nabla\varphi(\tilde{y}_{R,x_0})\right\rangle|\det \nabla\Phi_R(y)|\dy\notag\\
&+ c_2\int_{B_{\mathrm{n}(x_0)}^-(0,1)}|V\left(\nabla\varphi(\tilde{y}_{R,x_0})\right)|^2|\det \nabla\Phi_R(y)|\dy\notag\\
= &\int_{B_{\mathrm{n}(x_0)}^-(0,1)}G(x_0,\nabla\varphi(\tilde{y}_{R,x_0})\cdot \nabla\Phi_R(y))\dy -c_2\int_{B_{\mathrm{n}(x_0)}^-(0,1)}|V(\nabla\varphi(\tilde{y}_{R,x_0})\cdot \nabla\Phi_R(y))|^2\dy\notag\\
&-\int_{B_{\mathrm{n}(x_0)}^-(0,1)}G(\tilde{y}_{R,x_0},\nabla\varphi(\tilde{y}_{R,x_0}))|\det \nabla\Phi_R(y)|\dy+ c_2\int_{B_{\mathrm{n}(x_0)}^-(0,1)}\left|V\left(\nabla\varphi(\tilde{y}_{R,x_0})\right)\right|^2|\det \nabla\Phi_R(y)|\dy\notag\\
\leq & \,\varepsilon\int_{B_{\mathrm{n}(x_0)}^-(0,1)}|V(\nabla\varphi(\tilde{y}_{R,x_0}))|^2|\det \nabla\Phi_R(y)|\dy.\label{goodineq}
\end{align}
We are interested in using the quasiconvexity at the boundary condition in order to simplify the above expression. With this aim, we define $\tilde{\varphi}\colon B_{\mathrm{n}(x_0)}^-(0,1)\rightarrow\R^N$ as
\begin{equation*}
\tilde{\varphi}(y):=\frac{\varphi(R\Phi_R(y)+x_0)}{R}=\frac{\varphi(\tilde{y}_{R,x_0})}{R}.
\end{equation*}
Observe that, since $\varphi=0$ on $\partial(B(x_0,R))\cap\Omega$ and the diffeomorphism $\Phi_R^{-1}$ ``flattens'' the part of the boundary of $\frac{\Omega-x_0}{R}$ lying in $B(0,1)$, then 
\begin{equation}
\tilde{\varphi}=0 \mbox{\hspace{4mm} on \hspace{4mm}} \Phi_R^{-1}\left[\partial(B(0,1))\cap\frac{\Omega-x_0}{R}\right]=\partial(B(0,1))\cap B_{\mathrm{n}(x_0)}^-(0,1).
\end{equation}
Hence, by approximation, $\tilde{\varphi}$ is a suitable test function for the quasiconvexity at the free boundary condition. Since $\nabla\tilde{\varphi}(y)= \nabla\varphi(R\Phi_R(y)+x_0)\cdot \nabla\Phi_R(y)$, this means that 
\begin{align}
0 \leq &\int_{B_{\mathrm{n}(x_0)}^-(0,1)}\left(F(\nabla u(x_0)+\nabla\varphi(\tilde{y}_{R,x_0})\cdot \nabla\Phi_R(y))-F(\nabla u(x_0))\right)\dy\notag\\
& -c_2\int_{B_{\mathrm{n}(x_0)}^-(0,1)}|V(\nabla\varphi(\tilde{y}_{R,x_0})\cdot \nabla\Phi_R(y))|^2\dy.\label{QCatbdry}
\end{align}
Moreover, the weak Euler-Lagrange equation associated to the above minimality condition implies that 
\begin{equation}\label{ELQC}
\int_{B_{\mathrm{n}(x_0)}^-(0,1)}\left\langle F'(\nabla u(x_0)),\nabla\varphi(\tilde{y}_{R,x_0})\cdot \nabla\Phi_R(y)\right\rangle\dy=0.
\end{equation}
From  expressions (\ref{goodineq}), (\ref{QCatbdry}) and (\ref{ELQC}) we deduce that
\begin{align}
&-\int_{B_{\mathrm{n}(x_0)}^-(0,1)}F(\nabla u(\tilde{y}_{R,x_0})+\nabla\varphi(\tilde{y}_{R,x_0}))|\det \nabla\Phi_R(y)|\dy+\int_{B_{\mathrm{n}(x_0)}^-(0,1)}F(\nabla u(\tilde{y}_{R,x_0}))|\det \nabla\Phi_R(y)|\dy\notag\\
&+\int_{B_{\mathrm{n}(x_0)}^-(0,1)}\left\langle F'(\nabla u (\tilde{y}_{R,x_0})),\nabla\varphi(\tilde{y}_{R,x_0})\right\rangle|\det \nabla\Phi_R(y)|\dy\notag\\
&+ c_2\int_{B_{\mathrm{n}(x_0)}^-(0,1)}|V\left(\nabla\varphi(\tilde{y}_{R,x_0})\right)|^2|\det \nabla\Phi_R(y)|\dy\notag\\
\leq & \,\varepsilon\int_{B_{\mathrm{n}(x_0)}^-(0,1)}|V(\nabla\varphi(\tilde{y}_{R,x_0}))|^2|\det \nabla\Phi_R(y)|\dy.\notag
\end{align}
Applying the change of variables $x=R\Phi_R(y)+x_0=\tilde{y}_{R,x_0}$, this leads to
\begin{align}
&-\int_{\Omega(x_0,R)}\left(F(\nabla u(x)+\nabla\varphi(x))-F(\nabla u(x))+\left\langle F'(\nabla u(x)),\nabla\varphi(x)\right\rangle + c_2|V\left(\nabla\varphi(x)\right)|^2 \right) \dx\notag\\
\leq & \,\varepsilon\int_{\Omega(x_0,R)}|V(\nabla\varphi(x))|^2\dx.\label{bestineq}
\end{align}
Since $\varphi=0$ on $\partial(B(x_0,R))\cap\Omega$, in particular we have $\varphi=0$ on $\partial(B(x_0,R))\cap B_{\mathrm{n}(x_0)}^-(0,1)$. Therefore, because $u$ is an $F$-extremal, 
\begin{equation*}
\int_{\Omega(x_0,R)}\left\langle F'(\nabla u(x)),\nabla\varphi(x)\right\rangle\dx=0.
\end{equation*}
This, together with (\ref{bestineq}), imply for $\varepsilon=\frac{c_2}{2}$ that 
\begin{align}
&\int_{\Omega(x_0,R)}\left(F(\nabla u(x)+\nabla\varphi(x))-F(\nabla u(x))-c_2|V\left(\nabla\varphi(x)\right)|^2\right)\dx\notag\\
\geq & \,-\frac{c_2}{2}\int_{\Omega(x_0,R)}|V\left(\nabla\varphi(x)\right)|^2\dx,
\end{align}
which gives us the desired inequality after adding $c_2\int_{\Omega(x_0,R)}|V\left(\nabla\varphi(x)\right)|^2\dx$ to both sides of the above expression. 
This concludes the proof of for $x_0\in\Gamma_N$. 

On the other hand, for $R>0$ as defined above and $x_0\in\Omega\cup\Gamma_D$, if there exists $x_1\in\Omega(x_0,\frac{R}{2})\cap \Gamma_N$, then a variation $\varphi\in\mathrm{Var}\left(\Omega\left(x_0,\frac{R}{2}\right)\right)$ can be naturally extended to a variation in $\Omega(x_1,R)$, so that inequality (\ref{ecteo1}) also holds for such $\Omega\left(x_0,\frac{R}{2}\right)$ and $\varphi$.

Finally, if $x_0\in\overline{\Omega}$ and $\Omega(x_0,\frac{R}{2})\cap\Gamma_N=\emptyset$, then a simpler version of the above proof will work, since we can then use that the standard quasiconvexity holds in $\overline{\Omega}$. We  take $\overline{R}=\frac{R}{2}>0$, with $R$ as it was given for the case $x_0 \in\Gamma_N$, and replace $\Phi_R$ by the identity diffeomorphism in the argument above, given that there is no need, for this case, to flatten the boundary. All the other calculations follow in the exact same way. This concludes the proof of the theorem.
\end{proof}

\subsection{New proof of the sufficiency result}\label{SectNewProofSuff}

We now establish Theorem \ref{theoGM}, related to the sufficiency theorem proved by Grabovsky and Mengesha  in \cite{GyM}. The approach that we follow here  consists essentially in appropriately exploiting the result of K. Zhang \cite{Zhang}, that we generalized in the previous section, according to which smooth solutions of the weak Euler-Lagrange equation minimize the functional in \textit{small} subsets of the domain. The idea is then to partition the original domain into sufficiently small sets where we can apply Zhang's result and then add up the corresponding local estimates. This inevitably leads to obtaining an excess. The purpose is hence to prove that such excess converges to zero for a suitably normalized sequence of variations. The Decomposition Theorem and the fundamental theory of Young Measures play a crucial role in this argument, since they enable us to compare the strong positivity of the second variation with the behaviour of the shifted integrand $G$.  

We remark that, in the proof of the following result, the assumption that $u\in C^1(\overline{\Omega},\R^N)$ is mainly required while using the generalized version of Zhang's Theorem. 

\begin{proof}[Proof of Theorem \ref{theoGM}]
We will prove the result arguing by contradiction. Suppose that the theorem does not hold. Then, we can find a sequence $(\varphi_k)\subseteq \mathrm{Var}(\Omega,\R^N)$ such that $\|\varphi_k\|_{\LL^p(\Omega,\R^N)}\rightarrow 0$ and
\begin{equation}\label{forcontGM}
\underset{\Omega}{\int} F(\nabla u +\nabla\varphi_k)\dx<\underset{\Omega}{\int} F(\nabla u)\dx
\end{equation}
for all $k\in\N$. 

As in Lemma \ref{lemmaestimagesG}, we use Taylor's Approximation Theorem and define 
\begin{eqnarray*}
G(x,z)&:=&F(\nabla u(x)+z)-F(\nabla u(x))-\left\langle F'(\nabla u(x)),z\right\rangle\\
&=&\int_0^1(1-t)F''(\nabla u(x)+tz)[z,z]\dt.
\end{eqnarray*}

Note that, since $u$ is an $F$-extremal, for every $k\in\N$ it holds that
\begin{align}
\underset{\Omega}{\int} G(x,\nabla\varphi_k)\dx=&\underset{\Omega}{\int}\int_0^1(1-t)F''(\nabla u+t\nabla\varphi_k)[\nabla\varphi_k,\nabla\varphi_k]\dx\notag\\
=&\underset{\Omega}{\int} \bigl(F(\nabla u+\nabla\varphi_k)-F(\nabla u)-\left\langle F'(\nabla u),\nabla\varphi_k\right\rangle\bigr)\dx\notag\\
<&0.\label{eqGpsi}
\end{align} 

This inequality suggests the underlying idea behind the proof, which is to exploit the strong positivity of the second variation to obtain a contradiction. We split the remaining parts of the argument into the following steps.
\\
\\
\textit{Step 1. Derivation of a global estimate from Zhang's Theorem.} In this step we use Theorem \ref{teo1Zhang} and a covering argument to obtain, for $R=\overline{R}>0$ as given by Zhang's Theorem and for arbitrary $r\in (0,R)$,  a finite sequence of cubes satisfying 
\begin{equation}
\Omega \subseteq\bigcup_{j\in J}Q(x_j,r)
\end{equation}
such that, for every $\varphi\in \mathrm{Var}(\mathcal{A})$ and every $s\in (r,\min\{2r,R\})$, 

\begin{align}\label{eq8}
&\frac{c_2}{2}\underset{\Omega}{\int}|V(\nabla\varphi)|^2 \dx-c\sum_{j\in J}\underset{\Omega(x_j,s)-\Omega(x_j,r)}{\int}\left(|V(\nabla\varphi)|^2+\left|V\left(\frac{\varphi}{s-r}\right)\right|^2\right)\dx\notag\\
\leq& \underset{\Omega}{\int}(F(\nabla u+\nabla\varphi)-F(\nabla u))\dx=\underset{\Omega}{\int}G(x,\nabla\varphi)\dx.
\end{align}

Observe first that, taking $\Omega(x,R)=\Omega\cap Q(x,R)$, by Theorem \ref{teo1Zhang} and Remark \ref{Zhangwithcubes} we have
\begin{equation}\label{eqslm}
\frac{c_2}{2}\underset{\Omega(x,R)}{\int}|V(\nabla\varphi)|^2\dx\leq\underset{\Omega(x,R)}{\int}(F(\nabla u +\nabla\varphi)-F(\nabla u))\dx
\end{equation}
for all $\varphi\in \mathrm{Var}(\Omega(x,R),\R^N)$ and all $x\in\overline{\Omega}$.

Now, for a given $r\in(0,R)$, we consider a cover for $\Omega$ formed of a uniform grid consisting of non-overlapping cubes of side length $2r$, so that 
\begin{equation*}
\Omega\subseteq\bigcup_{j\in J} \overline{Q(x_j,r)}.
\end{equation*}

For each $j\in J$ and for $r<s<R$, consider cut-off functions $\rho_j\in C^1_c(Q(x_j,s))$ with the property that $\mathbbm{1}_{Q(x_j,r)}\leq \rho_j\leq\mathbbm{1}_{Q(x_j,s)}$ and $|\nabla\rho_j|\leq\frac{2}{s-r}$.

Note that the cubes $Q(x_j,s)$ have bounded overlap since, when $s<2r$, $Q(x_j,s)$ will intersect at most $3^n-1$ other such cubes. 

In addition, if $\varphi\in \mathrm{Var}(\Omega,\R^N)$, then $\rho_j\varphi\in \mathrm{Var}(\Omega(x_j,s),\R^N)$ and so, according to (\ref{eqslm}),

\begin{equation*}
\underset{\Omega(x_j,s)}{\int}\left(F(\nabla u)+\frac{c_2}{2}|V(\nabla(\rho_j\varphi))|^2\right)\dx\leq \underset{\Omega(x_j,s)}{\int}F(\nabla u+\nabla(\rho_j\varphi))\dx.
\end{equation*}

Since $u$ is an $F$-extremal, we also have 
\begin{equation*}
\frac{c_2}{2}\underset{\Omega(x_j,s)}{\int}|V(\nabla(\rho_j\varphi))|^2\dx\leq\underset{\Omega(x_j,s)}{\int}G(x,\nabla(\rho_j\varphi))\dx.
\end{equation*}

Then, since $\rho_j=1$ on $Q(x_j,r)$, we obtain 
\begin{align*}
&\frac{c_2}{2}\underset{\Omega(x_j,r)}{\int}|V(\nabla\varphi)|^2\dx+\frac{c_2}{2}\underset{\Omega(x_j,s)-\Omega(x_j,r)}{\int}|V(\nabla(\rho_j\varphi))|^2\dx\\
\leq&\underset{\Omega(x_j,r)}{\int}G(x,\nabla\varphi)\dx+\underset{\Omega(x_j,s)-\Omega(x_j,r)}{\int}G(x,\nabla(\rho_j\varphi))\dx.
\end{align*}

We use Lemma \ref{lemmaestimagesG} (a) to obtain, after adding up the previous inequalities over $j$, that
\begin{align*}
&\frac{c_2}{2}\underset{\Omega}{\int}|V(\nabla\varphi)|^2 \dx+\frac{c_2}{2}\sum_{j\in J}\underset{\Omega(x_j,s)-\Omega(x_j,r)}{\int}|V(\nabla(\rho_j\varphi))|^2\dx\\
\leq& \underset{\Omega}{\int}(F(\nabla u+\nabla\varphi)-F(\nabla u))\dx+c\sum_{j\in J}\underset{\Omega(x_j,s)-\Omega(x_j,r)}{\int}|V(\nabla(\rho_j\varphi))|^2\dx, 
\end{align*}
from where (\ref{eq8}) follows. 
\\
\\
\textit{Step 2. Reduction of the problem to $\W^{1,2}$-local minimizers.} In this step we will establish that $\varphi_k\rightarrow 0$ in $\W^{1,p}$ (and hence also in $\W^{1,2}$). We first show that $(\nabla\varphi_k)$ is bounded, for which we need to make use of the coercivity condition (H3) for the case $\Gamma_N\neq \emptyset$. However, if $\Gamma_N=\emptyset$, we can obtain this via a G{\aa}rding inequality. 
\\
\\
\textit{Case 1.} If $\Gamma_N=\emptyset$, assumptions $\mathrm{(H1)-(H2)}$ and the fact that $\varphi_k\in \W^{1,p}_0(\Omega,\R^N)$ imply that
\begin{align}
c_2\underset{\Omega}{\int}|\nabla\varphi_k|^p\dx&\leq\underset{\Omega}{\int}(F(\nabla\varphi_k)-F(0))\dx\nonumber\\
&\leq\underset{\Omega}{\int}(F(\nabla u+\nabla\varphi_k)+F(\nabla\varphi_k)-F(\nabla u+\nabla\varphi_k)-F(0))\dx\nonumber\\
&\leq \underset{\Omega}{\int}(F(\nabla u+\nabla\varphi_k)+\tilde{c}(1+|\nabla\varphi_k|^{p-1}+|\nabla u+\nabla\varphi_k|^{p-1})|\nabla u|-F(0))\dx.\label{firststepGarding}
\end{align}
Observe here that 
\begin{equation*}
c_2\underset{\Omega}{\int}\left(\frac{1}{2^{p-1}}|\nabla u+\nabla\varphi_k|^p-|\nabla u|^p\right)\dx\leq c_2\underset{\Omega}{\int}|\nabla\varphi_k|^p\dx
\end{equation*}
and, on the other hand, by Young's inequality applied to $c_2c_p|\nabla u+\nabla\varphi_k|^{p-1}c_p^{-1}|\nabla u|$ with an appropriate choice of the constant $c_p$, we have that
\begin{align*}
&\tilde{c}\underset{\Omega}{\int}(1+|\nabla\varphi_k|^{p-1}+|\nabla u+\nabla\varphi_k|^{p-1})|\nabla u|\dx\notag\\
\leq&c\underset{\Omega}{\int}(1+|\nabla u+\nabla\varphi_k|^{p-1}+|\nabla u|^{p-1})|\nabla u|\dx\\
\leq&\underset{\Omega}{\int}\left(\frac{c_2}{2^p}|\nabla u+\nabla\varphi_k|^p+c|\nabla u|^p+c|\nabla u|\right)\dx.
\end{align*}
Therefore, 
\begin{eqnarray*}
\frac{c_2}{2^{p-1}}\underset{\Omega}{\int}|\nabla u+\nabla\varphi_k|^p\dx\leq\underset{\Omega}{\int}\left(F(\nabla u+\nabla\varphi_k)-F(0)+c|\nabla u|+\left(c+c_2+\frac{2c^2}{c_2}\right)|\nabla u|^p\right)\dx
\end{eqnarray*}
or, equivalently, there are constants $\tilde{c}_3>0$ and $\tilde{c}_4>0$ such that 
\begin{equation*}
\tilde{c}_3\underset{\Omega}{\int}|\nabla u+\nabla\varphi_k|^p\dx\leq\underset{\Omega}{\int}F(\nabla u+\nabla\varphi_k)\dx+\tilde{c}_4\underset{\Omega}{\int}(1+|\nabla u|^p)\dx
\end{equation*} 
for all $k\in\N$.

This, together with assumption (\ref{forcontGM}) and Poincar\'{e} inequality, finally allows us to conclude that $(\varphi_k)$ is bounded in $\W^{1,p}$.
\\
\\
\textit{Case 2.} If $\Gamma_n\neq\emptyset$, then $(\varphi_k)$ is bounded in $\W^{1,p}(\Omega,\R^N)$ by assumptions (H3) and (\ref{forcontGM}). We remark that the reason why we cannot proceed, as in Case 1, to obtain a G{\aa}rding inequality without this assumption, is that $\varphi_k\notin\W^{1,p}_0(\Omega,\R^N)$ and, therefore, we cannot obtain (\ref{firststepGarding}) from the quasiconvexity condition.
\\

Having established that $(\varphi_k)$ is bounded in $\W^{1,p}$, we can further conclude that $\varphi_k\rightharpoonup 0$ in $\W^{1,p}(\Omega,\R^N)$.

%
%
%

Now, let $\gamma_k:=\|V(\nabla\varphi_k)\|_{\LL^2}$. Then, $\gamma_k>0$ for all $k\in\N$ and $(\gamma_k)$ is  bounded because $p\geq 2$. We will now show that  $\gamma_k\rightarrow 0$ as $k\rightarrow\infty$. Arguing by contradiction, we assume that there are $\gamma>0$ and  a subsequence, that we do not relabel, such that $\gamma_k\rightarrow\gamma$. Considering a further subsequence, we may also assume that $|V(\nabla\varphi_k)|^2\mathcal{L}^n\overset{*}{\rightharpoonup}\mu$ in $C^0\left(\overline{\Omega}\right)^*\cong\mathcal{M}\left(\overline{\Omega}\right)$.

We now take $r\in(0,R)$ and the grid so that $\mu\left(\bigcup_{j\in J} (\partial( Q(x_j,r))\cap\overline{\Omega})\right)=0$. This is possible because, \textit{for a given $x_0$}, only a countable amount of cubes can be such that $\mu(\partial Q(x_0,r))>0$. 

Now observe that, for $r<s<\min\{2r,R\}$, we get from inequality (\ref{eq8}) applied to $\varphi=\varphi_k$, that

\begin{align*}
&\frac{c_2}{2}\underset{\Omega}{\int}|V(\nabla\varphi_k)|^2 \dx-c\sum_{j\in J}\underset{\Omega(x_j,s)-\Omega(x_j,r)}{\int}\left(|V(\nabla\varphi_k)|^2+\left|V\left(\frac{\varphi_k}{s-r}\right)\right|^2\right)\dx\notag\\
\leq &\underset{\Omega}{\int}(F(\nabla u+\nabla\varphi_k)-F(\nabla u))\dx.
\end{align*}
Recall that, by assumption, $\varphi_k\rightarrow 0$ in $\LL^p(\Omega,\R^N)$ and, since $p\geq 2$, this implies that  $V(\varphi_k)\rightarrow 0$ in $\LL^2(\Omega,\R^N)$. Hence,

\begin{equation*}
\frac{c_2}{2}\gamma^2-c\mu\left( \overline{\Omega}\cap\bigcup_{j\in J}\left(\overline{Q(x_j,s)}-Q(x_j,r)\right) \right)\leq 0
\end{equation*}
and, letting $s\searrow r$ in the above expression, we get
\begin{equation*}
0<\frac{c_2}{2}\gamma^2=\frac{c_2}{2}\gamma^2-c\mu\left(\overline{\Omega}\cap\bigcup_{j\in J}\partial(Q(x_j,r)) \right)\leq 0, 
\end{equation*}
which is a contradiction. 

Consequently, $\gamma_k=\|V(\nabla\varphi_k)\|_{\LL^2}\rightarrow 0$.

Let $\alpha_k:= \|\nabla\varphi_k\|_{\LL^2}$ and $\beta_k:= (2|\Omega|)^{\frac{1}{2}-\frac{1}{p}}\|\nabla\varphi_k\|_{\LL^p}$. By Lemma \ref{lemmapropsV} we also have that $\alpha_k\rightarrow 0$ and $\beta_k\rightarrow 0$.  This way, we have reduced the problem to the case of $\W^{1,2}$-local minimizers. 
\\
\\
\textit{Step 3. Limit of the shifted functional at the normalized sequence.} We now define  $\psi_k:=\alpha_k^{-1}\varphi_k\in \mathrm{Var}(\Omega,\R^N)$. Hereby, $\underset{\Omega}{\int}|\nabla\psi_k|^2=1$ and hence we can assume, up to a subsequence, that $\psi_k\rightharpoonup \psi$ in $\W^{1,2}(\Omega,\R^N)$, $|\nabla\psi_k|^2\mathcal{L}^n\overset{*}{\rightharpoonup}\tilde{\mu}$ in $C^0_0(\overline{\Omega})^*$ and that $\nabla\psi_k\overset{Y}{\longrightarrow}(\nu_x)$ for some Young measure $(\nu_x)$. The purpose of this step is to show that
\begin{equation}\label{nonpositivesecondvar}
\frac{1}{2}\underset{\Omega}{\int}\int F''(\nabla u)[z,z]\dv{}\nu_x(z)\dx\leq0.
\end{equation}
The main idea will be to use the Decomposition Theorem (Theorem \ref{DecLemma}) and  estimate (\ref{eq8}) to prove that the concentrating part of $(\nabla\psi_k)$ does not change the sign of the second variation at the normalized sequence in the limit, while the oscillating part will be estimated by the left hand side of (\ref{nonpositivesecondvar}).

We recall that, by H\"{o}lder's inequality,
\begin{equation*}
\alpha_k= \|\nabla\varphi_k\|_{\LL^2}\leq (2|\Omega|)^{\frac{1}{2}-\frac{1}{p}}\|\nabla\varphi_k\|_{\LL^p}=\beta_k.
\end{equation*}
Therefore, $r_k:=\frac{\alpha_k}{\beta_k}\leq 1$ for every $k\in\N$.

We now claim that the given sequence of variations $(\varphi_k)$ is such that 
\begin{equation}\label{betapbyalpha2bded}
0\leq\underset{k\in\N}{\sup}\frac{\beta_k^p}{\alpha_k^2}=\Lambda<\infty
\end{equation}
for some real number $\Lambda>0$. Indeed, if $p=2$ this is trivially true and, if $p>2$, from the coercivity condition  (H3) applied to $\varphi_k$ it follows, after dividing by $\alpha_k^2$, that for every $k\in\N$, 
\begin{equation*}
\tilde{c}_4\frac{\beta_k^p}{\alpha_k^2}-c_5\leq\alpha_k^{-2}\underset{\Omega}{\int}\left(F(\nabla u+\nabla\varphi_k)-F(\nabla u)\right)\dx<0.
\end{equation*}
Whereby, the sequence $\left( \frac{\beta_k^p}{\alpha_k^2}\right)$ is bounded and the claim follows.

This, together with the the fact that $\underset{\Omega}{\int}|r_k\nabla\psi_k|^p=\beta_k^{-p}\underset{\Omega}{\int}|\nabla\varphi_k|^p=1$ and the Decomposition Theorem, implies that, for a subsequence of $(\psi_k)$ that we do not relabel, we can find sequences $(g_k)\subseteq \W^{1,2}_0(\Omega,\R^N)$ and $(b_k)\subseteq\mathrm{Var}(\Omega,\R^N)$ such that:
\begin{itemize}
\item $g_k\rightharpoonup 0$ and $b_k\rightharpoonup 0$ in $\W^{1,2}(\Omega,\R^N)$;
\item $r_k g_k\rightharpoonup 0$ and $r_k b_k\rightharpoonup 0$ in $\W^{1,p}(\Omega,\R^N)$;
\item $(|\nabla g_k|^2)$ and $(|r_k\nabla g_k|^p)$ are both equiintegrable;
\item $\nabla b_k\rightarrow 0$ in measure and
\item $\psi_k=\psi +g_k+b_k$.
\end{itemize}

Let us call $f_k:=\alpha_k^{-2}G(x,\alpha_k\nabla\psi_k)-\alpha_k^{-2}G(x,\alpha_k\nabla b_k)$. Then, by using Lemma \ref{lemmaestimagesG} and the Lipschitz property of $F$, we get that since $p\geq 2$ and $G(x,\cdot)$ is  quasiconvex  (and, therefore, also rank-one convex), there is a constant $c=c(p)>0$ such that, for every $z,w\in\R^{N\times n}$ and for every $x\in\overline{\Omega}$,
\begin{equation*}
|G(x,z)-G(x,w)|\leq c\left(|V_{p-1}(z)|+|V_{p-1}(w)|\right)|z-w|.
\end{equation*}
The proof of this inequality relies also on the fact that, for some constant $c>0$, 
\begin{equation*}
c^{-1}(|z|+|z|^{p-1})\leq|V_{p-1}(z)|\leq c(|z|+|z|^{p-1}).
\end{equation*}
This implies that, for any $\varepsilon>0$, there exists a constant $c_\varepsilon$ such that 
\begin{align*}
|f_k|&\leq c\alpha_k^{-1}(|V_{p-1}(\alpha_k\nabla\psi_k)|+|V_{p-1}(\alpha_k\nabla b_k)|)|\nabla\psi+\nabla g_k|\\
&\leq c\left(|\nabla\psi_k|+|\nabla b_k|+\alpha_k^{p-2}(|\nabla\psi_k|^{p-1}+|\nabla b_k|^{p-1})\right)|\nabla\psi+\nabla g_k|\\
&\leq\varepsilon \left(|\nabla\psi_k|^2+|\nabla b_k|^2+ \alpha_k^{p-2}(|\nabla\psi_k|^p+|\nabla b_k|^p )\right)+c_\varepsilon\left(|\nabla\psi+\nabla g_k|^2+\alpha_k^{p-2}|\nabla\psi+\nabla g_k|^p\right).
\end{align*}
Consequently, we can observe that for any set $A\subseteq\R^n$, 
\begin{equation}
\underset{A}{\int}|f_k|\dx\leq\varepsilon c_1+\tilde{c}_\varepsilon\underset{A}{\int}\left(|\nabla\psi+\nabla g_k|^2+\alpha_k^{p-2}|\nabla\psi+\nabla g_k|^p\right)\dx.
\end{equation}
Taking into account that $\left(\frac{\beta_k^p}{\alpha_k^2}\right)$ is  bounded and that
\begin{equation*}
\alpha_k^{p-2}|\nabla g_k|^p= \frac{\beta_k^p}{\alpha_k^2}r_k^p|\nabla g_k|^p, 
\end{equation*}
we deduce that $(\alpha_k^{p-2}|\nabla g_k|^p)$ is equiintegrable and, hence, so is $(f_k)$. 

Now, let $\varepsilon>0$. Since $(\nabla\psi_k)$ is measure-tight\footnote{A sequence $f_j\colon\Omega\rightarrow \R^d$ is measure-tight if $\underset{t\rightarrow\infty}{\lim}\underset{j\in\N}{\sup}\mathcal{L}^n(\{x\in\Omega\st|f_j(x)|>t\})=0$, which always holds for sequences that generate Young measures.} and $\nabla b_k\rightarrow 0$ in measure, we can take $m_\varepsilon>0$ large enough so that, for every $m\geq m_\varepsilon$, 
\begin{equation*}
\underset{\{|\nabla\psi_k|\geq m\}\cup\{|\nabla b_k|\geq m\}}{\int}|f_k|\dx<\varepsilon
\end{equation*}
for all $k\in\N$.\\
Then, for all $m\geq m_\varepsilon$, 
\begin{equation*} 
\underset{\{|\nabla\psi_k|< m\}\cap\{|\nabla b_k|< m\}}{\int}f_k\dx-\varepsilon<\underset{\Omega}{\int} f_k\dx.
\end{equation*}
We will now use the Fundamental Theorem of Young measures to take the limit inferior at both sides of the above expression and obtain that 
\begin{equation}\label{preeqtripstar}
\frac{1}{2}\underset{\Omega}{\int}\int F''(\nabla u)[z,z] \mathbbm{1}_{B(0,m)}(z)\dv{}\nu_x(z)\dx-\varepsilon\leq\liminf_{k\rightarrow\infty}\underset{\Omega}{\int} f_k\dx
\end{equation}
for all $m\geq m_\varepsilon$.
In order to prove this claim, consider the integrand $H:\overline{\Omega}\times\R^{N\times n}\rightarrow\R$ given by
\begin{equation*}
H(x,z):=F''(\nabla u(x))[z,z]\mathbbm{1}_{B(0,m)}(z).
\end{equation*}
Notice that $H(x,\cdot)$ is lower semicontinuous for every $x\in\overline{\Omega}$. By the Fundamental Theorem of Young measures, this implies that, since $\nabla\psi_k\overset{Y}{\longrightarrow}\nu_x$,
\begin{equation}\label{D2FbyFTYM}
\underset{\Omega}{\int}\int F''(\nabla u)[z,z]\mathbbm{1}_{B(0,m)}(z)\dv{}\nu_x(z)\dx\leq \liminf_{k\rightarrow\infty}\underset{|\nabla\psi_k|<m}{\int}F''(\nabla u)[\nabla\psi_k,\nabla\psi_k]\dx.
\end{equation}
On the other hand, the sequence of functions
\begin{equation*}
F''(\nabla u+t\alpha_k\nabla b_k)[\nabla b_k,\nabla b_k]\mathbbm{1}_{\{|\nabla b_k|<m \}\cap\{ |\nabla \psi_k|<m \}}
\end{equation*}
is bounded in $\LL^\infty(\Omega)$ for all $t\in[0,1]$ and, therefore, it is equiintegrable. In addition, this sequence converges to $0$ in measure because $\nabla b_k\rightarrow 0$ in measure and $F''$ is continuous. These two facts imply, by Vitali's Convergence Theorem, that
\begin{equation}\label{D2FbyVCT}
F''(\nabla u+t\alpha_k\nabla b_k)[\nabla b_k,\nabla b_k]\mathbbm{1}_{\{|\nabla b_k|<m \}\cap\{ |\nabla \psi_k|<m \}}\rightarrow 0
\end{equation}
in $\LL^1(\Omega)$ when $k\rightarrow\infty$ and for all $t\in[0,1]$. 

It is also clear, by the Dominated Convergence Theorem, that since $\alpha_k\rightarrow 0$, 
\begin{equation}\label{withandwithoutalpa}
\left|\underset{\Omega}{\int}\int_0^1(1-t)\left(F''(\nabla u+t\alpha_k\nabla\psi_k)-F''(\nabla u)\right)[\nabla\psi_k,\nabla\psi_k]\mathbbm{1}_{\{|\nabla b_k|<m \}\cap\{ |\nabla \psi_k|<m \}}\dt\dx\right|\rightarrow 0.
\end{equation}
Furthermore, given that $\nabla b_k\rightarrow 0$ in measure, we have that 
\begin{align}\label{withandwithouhtb}
&\left| \underset{\Omega}{\int}F''(\nabla u)[\nabla\psi_k,\nabla\psi_k]\left(\mathbbm{1}_{\{|\nabla \psi_k|<m \}}-    \mathbbm{1}_{\{|\nabla b_k|<m \}\cap\{ |\nabla \psi_k|<m \}}\right)\dx \right|\notag\\
\leq &cm^2\underset{\Omega}{\int}\mathbbm{1}_{\{|\nabla \psi_k|<m \}}\left(1- \mathbbm{1}_{\{|\nabla b_k|<m \}}   \right)\dx\rightarrow 0. 
\end{align}
By combining (\ref{D2FbyFTYM})-(\ref{withandwithouhtb}), we obtain that (\ref{preeqtripstar}) holds for all $m\geq m_\varepsilon$.

We now claim that
\begin{equation}\label{preeqtripstarbis}
\frac{1}{2}\underset{\Omega}{\int}\int F''(\nabla u)[z,z]\dv{}\nu_x(z)\dx-2\varepsilon\leq\underset{k\rightarrow \infty}{\liminf}\underset{\Omega}{\int} f_k\dx.
\end{equation}
Indeed, since by the Fundamental Theorem for Young Measures $(\nu_x)$ has a finite second moment, meaning that $\underset{\Omega}{\int}\underset{\hspace{1mm}\R^{N\times n}}{\int}|z|^2\dv\nu_x(z)\dx<\infty$, we can find $m\geq m_\varepsilon$ large enough so that
\begin{equation*}
\left|\underset{\Omega}{\int}\underset{\R^{N\times n}}{\int}F''(\nabla u(x))[z,z]\mathbbm{1}_{\R^{N\times n}\backslash\overline{B(0,m)}}(z)\dv\nu_x(z)\dx\right|<\varepsilon.
\end{equation*}
By letting $\varepsilon\rightarrow 0$ in (\ref{preeqtripstarbis}), we conclude that
\begin{equation}\label{eqtripstar}
\frac{1}{2}\underset{\Omega}{\int}\int F''(\nabla u)[z,z]\dv{}\nu_x(z)\dx\leq\liminf_{k\rightarrow\infty}\underset{\Omega}{\int} f_k\dx.
\end{equation}
We now show that
\begin{align}\label{0leqGbk}
0&\leq\underset{k\rightarrow\infty}{\liminf}\underset{\Omega}{\int}\alpha_k^{-2}G(x,\alpha_k\nabla b_k)\dx.
\end{align}
We take $\varphi=\alpha_kb_k$ in inequality (\ref{eq8}) and recall that $c_p^{-1}(|\xi|^2+|\xi|^p)\leq |V(\xi)|^2\leq c_p(|\xi|^2+|\xi|^p)$. Using that $u$ is an $F$-extremal and $b_k\in \mathrm{Var}(\Omega,\R^N)$, after dividing by $\alpha_k^2$ we get, for some constant $c_p>0$,
\begin{align*}
&\frac{c_2c_p}{2}\underset{\Omega}{\int}\left(|\nabla b_k|^2 +\alpha_k^{p-2}|\nabla b_k|^p   \right)\dx\notag\\
&-c\sum_{j\in J}\underset{\Omega(x_j,s)-\Omega(x_j,r)}{\int}\left(|\nabla b_k|^2+\alpha_k^{p-2}|\nabla b_k|^p +\frac{| b_k|^2}{(s-r)^2}+\alpha_k^{p-2}\frac{| b_k|^p}{(s-r)^p}  \right)\dx\notag\\
\leq &\alpha_k^{-2}\underset{\Omega}{\int}(G(x,\alpha_k\nabla b_k)\dx
\end{align*}
for every $s,r$ such that $\frac{R}{2}<r<s<R$.

Notice that 
\begin{equation*}
\alpha_k^{p-2}\left(|b_k|^p+|\nabla b_k|^p\right)=\frac{\beta_k^p}{\alpha_k^2}r_k^p\left(|b_k|^p+|\nabla b_k|^p\right).
\end{equation*}
Since $r_kb_k\rightharpoonup 0$ in $\W^{1,p}(\Omega,\R^N)$, we use again that $\left( \frac{\beta_k^p}{\alpha_k^2}\right) $ is bounded and deduce that, for a subsequence that we do not relabel, it also holds that 

\begin{equation*}
\alpha_k^{\frac{p-2}{p}}b_k ={\beta_k}{\alpha_k^{-\frac{2}{p}}}r_kb_k   \rightharpoonup 0 \hspace{7mm}\mathrm{ in } \hspace{7mm} \W^{1,p}(\Omega,\R^N).
\end{equation*}
We can now use this, and the fact that $b_k\rightharpoonup 0$ in $\W^{1,2}(\Omega,\R^N)$, to proceed exactly as we did to prove that $\alpha_k\rightarrow 0$ and whereby conclude that
\begin{align}
0&\leq \frac{c_2c_p}{2}\underset{k\rightarrow\infty}{\liminf}\underset{\Omega}{\int} |\nabla b_k|^2+\alpha_k^{p-2}|\nabla b_k|^p\dx\notag\\
= \,&\underset{k\rightarrow\infty}{\liminf}\left[\frac{c_2c_p}{2}\underset{\Omega}{\int}\left(|\nabla b_k|^2 +\alpha_k^{p-2}|\nabla b_k|^p   \right)\dx\right.\notag\\
&\left.-c\sum_{j\in J}\underset{\Omega(x_j,s)-\Omega(x_j,r)}{\int}\left(|\nabla b_k|^2+\alpha_k^{p-2}|\nabla b_k|^p +\frac{| b_k|^2}{(s-r)^2}+\alpha_k^{p-2}\frac{| b_k|^p}{(s-r)^p}  \right)\dx\right]\notag\\
 \leq\,&\underset{k\rightarrow\infty}{\liminf}\underset{\Omega}{\int}\alpha_k^{-2}G(x,\alpha_k\nabla b_k)\dx\notag.
\end{align}
Hence, (\ref{0leqGbk}) is settled.

Using this, together with (\ref{eqGpsi}) and (\ref{eqtripstar}), we get 
\begin{align}
\frac{1}{2}\underset{\Omega}{\int}\int F''(\nabla u)[z,z]\dv{}\nu_x(z)\dx&\leq\underset{k\rightarrow\infty}{\liminf}\underset{\Omega}{\int}f_k\dx+\underset{k\rightarrow\infty}{\liminf}\underset{\Omega}{\int}\alpha_k^{-2}G(x,\alpha_k\nabla b_k)\dx\notag\\
&\leq\underset{k\rightarrow\infty}{\liminf}\left(\underset{\Omega}{\int}f_k\dx+\underset{\Omega}{\int}\alpha_k^{-2}G(x,\alpha_k\nabla b_k)\right)\dx\notag\\
&=\underset{k\rightarrow\infty}{\liminf}\underset{\Omega}{\int}\alpha_k^{-2}G(x,\alpha_k \psi_k)\dx\notag\\
&\leq0\notag,
\end{align}
from which (\ref{nonpositivesecondvar}) follows.
\\
\\
\textit{Step 4. Final contradiction from the strong positivity of the second variation applied to the normalized sequence.} In this step we will show that $\psi_k \rightharpoonup 0$ in $\W^{1,p}$, from which a contradiction will follow after using inequality (\ref{eq8}). The main idea will be to use the complementing behaviour coming from (\ref{nonpositivesecondvar}) and the strong positivity of the second variation. 

We first claim that
\begin{align}\label{QCforquadwithYM}
\frac{1}{2}\underset{\Omega}{\int}F''(\nabla u(x))[\overline{\nu}_x,\overline{\nu}_x]\dx+{c_2}\underset{\Omega}{\int}\int|z-\overline{\nu}_x|^2\dv{}\nu_x(z)\dx&\leq\frac{1}{2}\underset{\Omega}{\int}\int F''(\nabla u(x))[z,z]\dv{}\nu_x(z)\dx.
\end{align}
Indeed, since  $F$ is strongly quasiconvex, for every $x\in\Omega$ the quadratic function
\begin{equation*}
\eta\mapsto F''(\nabla u(x))[\eta,\eta]-2c_2|\eta-\overline{\nu}_x|^2
\end{equation*}
is quasiconvex. To verify this, it is enough to compute the second variation of the functional $\mathcal{G}(v):=\int_\Omega F(\nabla v)-F(\nabla u(x))-c_2|V(\nabla v-\nabla u(x))|^2\dy$ at the affine function $v(y):= \nabla u(x)y$ and use the minimality property associated to quasiconvexity. 

Hence, by Jensen's inequality (see \cite{BallZhang}), we obtain (\ref{QCforquadwithYM}) after integrating over $\Omega$.

In addition, since $\overline{\nu}_x=\nabla\psi(x)$ and the second variation is strictly positive, we obtain from (\ref{nonpositivesecondvar}) and (\ref{QCforquadwithYM}) that
\begin{align}
\frac{c_3}{2}\underset{\Omega}{\int}|\nabla\psi|^2\dx+c_2\underset{\Omega}{\int}\int|z-\nabla\psi|^2\dv{}\nu_x(z)\dx&\leq 0
\end{align}
and thus, using Poincar\'e inequality, we can conclude that $\psi=0$ and $\nu_x=\delta_0$. This implies that $\nabla\psi_k\rightharpoonup 0$ in $\LL^2(\Omega,\R^N)$ and, furthermore, that $\nabla\psi_k\rightarrow 0$ in measure.

On the other hand, taking $\varphi=\alpha_k\psi_k$ in inequality (\ref{eq8}) we obtain, after dividing by $\alpha_k^2$, that

\begin{align*}
&\frac{c_2\tilde{c_p}}{2}\underset{\Omega}{\int}|\nabla\psi_k|^2\dx\notag\\
-&c\sum_{j\in J}\underset{\Omega(x_j,s)-\Omega(x_j,r)}{\int}\left(|\nabla\psi_k|^2+\alpha_k^{p-2}|\nabla\psi_k|^p +\frac{|\psi_k|^2}{(s-r)^2}+\alpha_k^{p-2}\frac{|\psi_k|^p}{(s-r)^p}\right)\dx\\
\leq&\frac{c_2\tilde{c_p}}{2}\underset{\Omega}{\int}\left(|\nabla\psi_k|^2+\alpha_k^{p-2}|\nabla\psi_k|^p\right)\dx\notag\\
-&c\sum_{j\in J}\underset{\Omega(x_j,s)-\Omega(x_j,r)}{\int}\left(|\nabla\psi_k|^2+\alpha_k^{p-2}|\nabla\psi_k|^p +\frac{|\psi_k|^2}{(s-r)^2}+\alpha_k^{p-2}\frac{|\psi_k|^p}{(s-r)^p}\right)\dx<0
\end{align*}
for all $k\in\N$ and for all $r,s$ such that $\frac{R}{2}<r<s<R$.
Observe that
\begin{equation*}
\alpha_k^{\frac{p-2}{p}}\psi_k= \beta_k\alpha_k^{-\frac{2}{p}}r_k\psi_k.
\end{equation*}
Hence, for a subsequence that we do not relabel, we use again (\ref{betapbyalpha2bded}) to further conclude that $\alpha_k^{\frac{p-2}{p}}\psi_k \rightharpoonup 0$ in $\W^{1,p}(\Omega,\R^N)$. Arguing exactly as we did to prove that $\gamma_k\rightarrow 0$ and inequality (\ref{0leqGbk}), we can now take the limit when $k\rightarrow\infty$ and use the property $\underset{\Omega}{\int}|\nabla\psi_k|^2\dx=1$,  to obtain that
\begin{equation*}
0<\frac{c_2\tilde{c_p}}{2}\leq 0,
\end{equation*}
which is a contradiction. This concludes the proof of the theorem. 
\end{proof}
We recall that, as an application of this theorem, in \cite[Section 6]{GyM} Grabovsky \& Mengesha constructed an interesting class of strong local minimizers that are not global minimizers. Further examples of this situation can be found in \cite{KohnStern,Taheritopology}. We highlight the relevance that the following sections acquire in this context, since they are concerned with regularity results for strong local minimizers. 

\section{Lipschitz extremals and $\mathrm{BMO}$-local minimizers}\label{SectLipExtBMO}

It is well known that, by using Taylor Approximation Theorem,  Lipschitz extremals at which the second variation is positive can be shown to be weak local minimizers of the functional $\mathcal{F}$.

In this section we will establish that Lipschitz solutions to the weak Euler-Lagrange equation at which the second variation is positive, are slightly more than merely weak local minimizers. This result is inspired by \cite[Theorem 6.1]{Kristensen}, where the same is proved for the case of homogeneous integrands and by assuming that the variations are uniformly bounded. See also \cite{Firoozye}.

We also remark that we work under the assumption of strong quasiconvexity ($\mathrm{\overline{H}}2$), which is not necessary to show that Lipschitz extremals with strictly positive second variation are weak local minimizers. More precisely, rank one convexity and $p$-growth are enough in that case.


We first state the following definitions.
\begin{definition}
Let $\phi\in \LL^1(\Omega,\R^{N\times n})$. We say that $\phi$ is of \textbf{bounded mean oscillation} if and only if 
\begin{equation*}
\underset{B(x,r)\subseteq\Omega}{\sup}\underset{B(x,r)}{\mint}|\phi-(\phi)_{x,r}|\dy<\infty.
\end{equation*}
In this case, we define the semi-norm
\begin{equation*}
[\phi]_{\mathrm{BMO}(\Omega,\R^{N\times n})}:=\underset{B(x,r)\subseteq\Omega}{\sup}\underset{B(x,r)}{\mint}|\phi-(\phi)_{x,r}|\dx<\infty
\end{equation*}
and we set
\begin{equation*}
\mathrm{BMO}(\Omega,\R^{N\times n}):=\{\phi\in \LL^1(\Omega,\R^{N\times n})\st [\phi]_{\mathrm{BMO}(\Omega,\R^{N\times n})}<\infty\}.
\end{equation*}
\end{definition}

We recall the following definition regarding the space that consists of the closure of $C_0^0(\Omega,\R^{N\times n})$ in $\mathrm{BMO}(\Omega,\R^{N\times n})$.
\begin{definition}\label{defVMO}
Let $\phi\in \mathrm{BMO}(\Omega,\R^{N\times n})$. We say that $\phi$ is of \textbf{vanishing mean oscillation} if and only if
\begin{equation*}
\underset{\rho\rightarrow 0}{\lim}\,\underset{r\leq\rho}{\sup}\,\underset{B(x,r)\subseteq\Omega}{\sup}\underset{B(x,r)}{\mint}|\phi-(\phi)_{x,r}|\dy=0
\end{equation*}
and we set
\begin{equation*}
\mathrm{VMO}(\Omega,\R^{N\times n}):=\{\phi\in \mathrm{BMO}(\Omega,\R^{N\times n})\st\ \phi \text{ is of vanishing mean oscillation}\}.
\end{equation*}
Furthermore, we define 
\begin{equation*}
\mathrm{VMO}_0(\Omega,\R^{N\times n}):=\left\{\phi\in \mathrm{VMO}(\Omega,\R^{N\times n})\st\ \tilde{\phi}\in\mathrm{VMO}(\R^n,\R^{N\times n}),  \,\,\, \tilde{\varphi}=
0 \mbox{ in }\R^n-\Omega,\,\,\,\tilde{\varphi}\hspace{-1mm}\restriction_\Omega=\varphi 
\right\}.
\end{equation*}
and, for $h\in \mathrm{VMO}(\tilde{\Omega},\R^{N\times n})$, $\tilde{\Omega}$ a neighbourhood of $\overline{\Omega}$,\footnote{Notice that, by  \cite[Lemma 5]{BrezisNirenbergBMOII}, if $\Omega$ is a smooth bounded domain, a function $h\in\mathrm{VMO}(\partial\Omega,\R^{N\times n})$ can be extended to a function $\tilde{h}\in \mathrm{VMO}(\tilde{\Omega},\R^{N\times n})$ as used above. The definition of BMO and VMO for maps between manifolds can be found in \cite[SI.1]{BrezisNirenbergBMOI}.} we let
\begin{equation*}
\mathrm{VMO}_h(\Omega,\R^{N\times n}):=\left\{ \phi\in \mathrm{VMO}(\Omega,\R^{N\times n})\st\ \phi-h\in \mathrm{VMO}_0(\Omega,\R^{N\times n})    \right\}.
\end{equation*}
\end{definition}
\begin{remark}
$\mathrm{VMO}(\Omega,\R^{N\times n})$ is, indeed, the closure of $C_0^0(\Omega,\R^{N\times n})$ in $\mathrm{BMO}(\Omega,\R^{N\times n})$. See \cite{BrezisNirenbergBMOI,BrezisNirenbergBMOII} for a comprehensive discussion on this and other properties of the VMO space, including the notion of trace in this context.
\end{remark}

The main result in this section can now be stated as follows.
\begin{theorem}\label{morethanWLMVMO}
Let $F\colon\overline{\Omega}\times\R^{N\times n}\rightarrow\R$ be a function satisfying $\mathrm{(\overline{H}0)-(\overline{H}2)}$ for some $1<p<\infty$ and such that $F(x_0,\cdot)$ is quasiconvex for every $x_0\in\overline{\Omega}$. Let $\overline{u}\in \W^{1,\infty}(\Omega,\R^N)$ be an extremal with strictly positive second variation, i.e., for some $c_3>0$ and all $\varphi\in C_0^\infty(\Omega,\R^N)$, 
\begin{equation}
\label{uextremalVMO}
\underset{\Omega}{\int}\left\langle F_z(\cdot,\nabla\overline{u}),\nabla\varphi\right\rangle\dx=0
\end{equation}
and
\begin{equation}
\label{secvarposVMO}
c_3\underset{\Omega}{\int}|\nabla\varphi|^2\dx\leq \underset{\Omega}{\int}F_{zz}(\cdot,\nabla\overline{u})[\nabla\varphi,\nabla\varphi]\dx. 
\end{equation}
Then, there is a $\delta>0$ such that
\begin{equation*}
\underset{\Omega}{\int}F(\cdot,\nabla\overline{u})\dx\leq \underset{\Omega}{\int}F(\cdot,\nabla\overline{u}+\nabla\varphi)\dx
\end{equation*}
for every $\varphi\in C_0^\infty(\Omega,\R^N)$ with $[\nabla\varphi]_{\mathrm{BMO}}\leq\delta$.
\end{theorem}
In addition to the $x$-dependence, we remark that our statement differs from Theorem 6.1 in \cite{Kristensen} in that, in their result, the parameter $\delta$ that gives the local minimality, depends on a given constant $M>0$ for which the variations $\varphi$ are required to satisfy $\|\nabla\varphi\|_{\LL^\infty}\leq M$. By adapting the truncation technique from Acerbi {\&} Fusco \cite{AcFus}, we have been able to remove this additional restriction.  

For the proof of Theorem \ref{morethanWLMVMO} we will require the following definition and the subsequent lemmata, that generalize the Hardy-Littlewood-Fefferman-Stein maximal inequality to Orlicz spaces.
\begin{definition}
Let $f\colon\R^n\rightarrow\R^{N\times n}$ be an integrable map. We define the \textbf{Hardy-Littlewood maximal function} by
\begin{equation*}
f^\star(x):=\underset{B(y,r)\ni x}{\sup}\,\underset{B(y,r)}{\mint}|f(y)|\dy,
\end{equation*}
where the supremum is taken over all balls $B(y,r)\subseteq\R^n$ containing $x$. Similarly, the \textbf{Fefferman-Stein maximal function} is given by
\begin{equation*}
f^{\#}(x):=\underset{B(y,r)\ni x}{\sup}\,\underset{B(y,r)}{\mint}|f(y)-(f)_{y,r}|\dy.
\end{equation*}
\end{definition}
A useful generalization of the Hardy-Littlewood maximal inequality is the following. 
\begin{lemma}\label{LemmaHLMaxVMO}
Let $\Phi\colon[0,\infty)\rightarrow[0,\infty)$ be a continuously increasing function with $\Phi(0)=0$. Assume, in addition, that $\Phi(t)=t^p A(t)$ for some $p>1$ and some increasing function $A\colon[0,\infty)\rightarrow[0,\infty)$.  Then, there exists a constant $\gamma=\gamma(n,p)$ such that
\begin{equation}
\underset{\R^n}{\int}\Phi(|f|)\dx\leq \underset{\R^n}{\int}\Phi(f^{\star})\dx \leq\gamma\underset{\R^n}{\int}\Phi(2|f|)\dx
\end{equation}
for all $f\in \LL^1(\R^n,\R^{N\times n})$.
\end{lemma}
The proof of the first inequality in this lemma follows from the fact that $\Phi$ is increasing and from Lebesgue Differentiation Theorem, which implies that $|f(x)|\leq f^\star(x)$ for almost every $x\in\R^n$. For a proof of the second inequality we refer the reader to \cite[Lemma 5.1]{GrIwaMos}.
We can relate both notions of maximal functions in the following way.
\begin{lemma}\label{lemmaFSMaxVMO}
Let $\Phi\colon[0,\infty)\rightarrow[0,\infty)$ be a continuously increasing function with $\Phi(0)=0$. Let $\varepsilon>0$ and $f\in  \LL^1(\R^n,\R^{N\times n})$. Then, 
\begin{equation}
\label{eqlemmaFSMaxVMO}
\underset{\R^n}{\int}\Phi(f^\star)\dx\leq \frac{5^n}{\varepsilon}\underset{\R^n}{\int}\Phi\left(\frac{f^\#}{\varepsilon}\right)\dx+2\cdot 5^{3n}\varepsilon\underset{\R^n}{\int}\Phi(5^n2^{n+1}f^\star)\dx.
\end{equation}
If, in addition, we have that
\begin{equation*}
\underset{t>0}{\sup}\frac{\Phi(2t)}{\Phi(t)}<\infty,
\end{equation*}
we can further conclude that there is a constant $\gamma_1=\gamma_1(n)$ such that
\begin{equation}
\label{eqlemmaFSMax2VMO}
\underset{\R^n}{\int}\Phi(f^\star)\dx\leq \gamma_1\underset{\R^n}{\int}\Phi(f^\#) \dx
\end{equation}
whenever $f\in \LL^1(\R^n,\R^{N\times n})$ is such that $\underset{\R^n}{\int}\Phi(f^\star)\dx<\infty$. 
\end{lemma}
The proof of (\ref{eqlemmaFSMaxVMO}) can be found, for example, in \cite{Kristensen}. Inequality (\ref{eqlemmaFSMax2VMO}) follows easily from (\ref{eqlemmaFSMaxVMO}) under the given extra assumptions.

We can now proceed with the proof of the theorem.
\\
\begin{proof}[Proof of Theorem \ref{morethanWLMVMO}]
Let $\omega$ be a modulus of continuity for $F_{zz}$ on the set
\begin{equation*}
\overline{\Omega}\times\{\xi\in\R^{N\times n}\st |\xi|\leq 1+\|\nabla\overline{u}\|_\infty\}.
\end{equation*}
We extend $\omega$ to cover all $\overline{\Omega}\times\R^{N\times n}$, so that it has the following properties:
\begin{itemize}
\item $\omega:[0,\infty)\rightarrow[0,\infty)$;
\item $\omega$ is continuous and increasing;
\item $\omega(0)=0$ and $\omega(t)=1$ for every $t\geq 1$;
\item $\underset{t>0}{\sup}\frac{\omega(2t)}{\omega(t)}<\infty$ and
\item $|F_{zz}(x,\xi)-F_{zz}(x,\eta)|\leq c\omega(|\xi-\eta|)$ for some constant $c=c(\|\nabla u\|_{\LL^\infty})>0$,  every  $|\xi|,|\eta|\leq 1+\|\nabla\overline{u}\|_\infty$, and every $x\in\overline{\Omega}$.
\end{itemize}
Now, if $\varphi\in C_0^\infty(\Omega,\R^N)$, then by Taylor's formula, (\ref{uextremalVMO}), and (\ref{secvarposVMO}), we have
\begin{align}
&\underset{\Omega}{\int}\left(F(\cdot,\nabla\overline{u}+\nabla\varphi)-F(\cdot,\nabla\overline{u})\right)\dx\notag\\
=&\underset{\Omega}{\int}\left(F(\cdot,\nabla\overline{u}+\nabla\varphi)-F(\cdot,\nabla\overline{u})-\left\langle F_z(\cdot,\nabla\overline{u}),\nabla\varphi\right\rangle\right)\dx\notag\\
=&\underset{\Omega}{\int}\left(F(\cdot,\nabla\overline{u}+\nabla\varphi)-F(\cdot,\nabla\overline{u})-\left\langle F_z(\cdot,\nabla\overline{u}),\nabla\varphi\right\rangle-\frac{1}{2}F_{zz}(\cdot,\nabla\overline{u})[\nabla\varphi,\nabla\varphi]\right)\mathbbm{1}_{\{\|\nabla\varphi|> 1\}}\dx\notag\\
&+\underset{\Omega}{\int}\int_0^1(1-t)\left(F_{zz}(\cdot,\nabla\overline{u}+t\nabla\varphi)-F_{zz}(\cdot,\nabla\overline{u})\right)[\nabla\varphi,\nabla\varphi]\mathbbm{1}_{\{\|\nabla\varphi|\leq 1\}}\dt\dx\notag\\
&+\frac{1}{2}\underset{\Omega}{\int}F_{zz}(\cdot,\nabla\overline{u})[\nabla\varphi,\nabla\varphi]\dx\notag\\
\geq &\frac{c_3}{2} \underset{\Omega}{\int}|\nabla\varphi|^2\dx-c\underset{\Omega}{\int}\left(\left(1+|\nabla\overline{u}|^{p-1}+|\nabla\varphi|^{p-1}\right)|\nabla\varphi|+|\nabla\varphi|+|\nabla\varphi|^2\right)\mathbbm{1}_{\{\|\nabla\varphi|> 1\}}\dx\notag\\
&-c\underset{\Omega}{\int}\omega(|\nabla\varphi|)|\nabla\varphi|^2\mathbbm{1}_{\{\|\nabla\varphi|\leq 1\}}\dx\notag\\
\geq &\frac{c_3}{2} \underset{\Omega}{\int}|\nabla\varphi|^2\dx-c\underset{\Omega}{\int}\left(|\nabla\varphi|^2+|\nabla\varphi|^p\right)\mathbbm{1}_{\{\|\nabla\varphi|> 1\}}\dx-c\underset{\Omega}{\int}\omega(|\nabla\varphi|)|\nabla\varphi|^2\mathbbm{1}_{\{\|\nabla\varphi|\leq 1\}}\dx,\label{wlmtheoVMO}
\end{align}
where $c=c(\|\nabla\overline{u}\|_\infty,n,p)$ and the last inequality follows from the fact that $ab^{p-1}\leq a^p+b^p$ for $a,b>0$. 
\\
We now consider two different cases. 
\\
\textit{Case 1.} If $1<p\leq 2$, it follows from the above chain of inequalities that 
\begin{align}
&\underset{\Omega}{\int}\left(F(\cdot,\nabla\overline{u}+\nabla\varphi)-F(\cdot,\nabla\overline{u})\right)\dx\notag\\
\geq & \frac{c_3}{2} \underset{\Omega}{\int}|\nabla\varphi|^2\dx-c\underset{\Omega}{\int}|\nabla\varphi|^2\mathbbm{1}_{\{\|\nabla\varphi|> 1\}}\dx-c\underset{\Omega}{\int}\omega(|\nabla\varphi|)|\nabla\varphi|^2\mathbbm{1}_{\{\|\nabla\varphi|\leq 1\}}\dx\notag\\
\geq &  \frac{c_3}{2} \underset{\Omega}{\int}|\nabla\varphi|^2\dx-c\underset{\Omega}{\int}\omega(|\nabla\varphi|)|\nabla\varphi|^2\dx\label{mintheosubquadVMO}.
\end{align}
Extending $\varphi$ by $0$ outside of $\Omega$, we see $\nabla\varphi$ as a map defined on $\R^n$. 

Then, applying Lemmata \ref{LemmaHLMaxVMO} and \ref{lemmaFSMaxVMO}  with $\Phi(t)=t^2\omega(t)$, we find, for a new constant $c>0$, that
\begin{align*}
&\underset{\Omega}{\int}\left(F(\cdot,\nabla\overline{u}+\nabla\varphi)-F(\cdot,\nabla\overline{u})\right)\dx\notag\\
\geq & \frac{c_3}{2} \underset{\Omega}{\int}|\nabla\varphi|^2\dx-c\underset{\Omega}{\int}\omega(c|(\nabla\varphi)^\#|)|(\nabla\varphi)^\#|^2\dx.\label{almostminthVMO}
\end{align*}
It is easy to observe that
\begin{equation}\label{phicatleqphistar}
(\nabla\varphi)^\#\leq 2(\nabla\varphi)^\star. 
\end{equation}
This, and the Hardy-Littlewood-Wiener maximal inequality, imply that 
\begin{equation*}
\underset{\R^n}{\int}|(\nabla\varphi)^\#|^2\dx\leq 4\underset{\R^n}{\int}|(\nabla\varphi)^\star|^2\dx\leq c\underset{\R^n}{\int}|\nabla\varphi|^2\dx.
\end{equation*}
Putting this and (\ref{mintheosubquadVMO}) together, we conclude that 
\begin{equation}
\underset{\Omega}{\int}\left(F(\cdot,\nabla\overline{u}+\nabla\varphi)-F(\cdot,\nabla\overline{u})\right)\dx
\geq   \underset{\Omega}{\int}\left(\frac{c_3\,c}{2}-c\,\omega(c|(\nabla\varphi)^\#|)\right)|(\nabla\varphi)^\#|^2\dx.
\end{equation}
By taking $\delta>0$ small enough, the right hand side of the above expression will be non-negative when $[\nabla\varphi]_{\mathrm{BMO}}=\|(\nabla\varphi)^\#\|_\infty\leq \delta$. This concludes the proof for the case $1<p\leq 2$.
\\
\\
\textit{Case 2.} If, on the other hand, $2<p<\infty$, from (\ref{wlmtheoVMO}) we infer that
\begin{align}
&\underset{\Omega}{\int}\left(F(\cdot,\nabla\overline{u}+\nabla\varphi)-F(\cdot,\nabla\overline{u})\right)\dx\notag\\
\geq& \frac{c_3}{2} \underset{\Omega}{\int}|\nabla\varphi|^2\dx-c\underset{\Omega}{\int}\left(|\nabla\varphi|^2+|\nabla\varphi|^p\right)\dx-c\underset{\Omega}{\int}\omega(|\nabla\varphi|)|\nabla\varphi|^2\dx\notag\\
\geq& \frac{c_3}{2} \underset{\Omega}{\int}|\nabla\varphi|^2\dx-c\underset{\Omega}{\int}\omega(|\nabla\varphi|)\left(|\nabla\varphi|^2+|\nabla\varphi|^p\right)\dx.\label{locminp2VMO}
\end{align}
We extend $\varphi$ by defining it like $0$ outside of $\Omega$ and, for a $\delta>0$ still to be specified, we assume that $\|(\nabla\varphi)^\#\|_\infty<\delta<1$. Then, by Lemmata \ref{LemmaHLMaxVMO} and \ref{lemmaFSMaxVMO} applied with $\Phi(t)=t^p\omega(\tilde{c}t)$ and $\Phi(t)=t^2\omega(\tilde{c}t)$ in both directions, we use again (\ref{phicatleqphistar}) to obtain that, for different constants $\tilde{c}=\tilde{c}(n)$,
\begin{align}
\underset{\R^n}{\int}|\nabla\varphi|^p\omega(\nabla\varphi) \dx\leq & \,c\underset{\R^n}{\int}|(\nabla\varphi)^\star|^p\omega((\nabla\varphi)^\star)\dx\leq \, c \underset{\R^n}{\int}|(\nabla\varphi)^\#|^p\omega(\tilde{c}(\nabla\varphi)^\#)\dx\notag\\
\leq &  \,c\underset{\R^n}{\int}|(\nabla\varphi)^\#|^2\omega(\tilde{c}(\nabla\varphi)^\#)\dx\leq \,  c\underset{\R^n}{\int}|(\nabla\varphi)^\star|^2\omega(\tilde{c}(\nabla\varphi)^\star)\dx\notag\\
\leq & \, c\underset{\R^n}{\int}|\nabla\varphi|^2\omega(\tilde{c}(\nabla\varphi))\dx.\label{usingmaxVMO}
\end{align}
It is for the third inequality above that we are using the assumption $\|(\nabla\varphi)^\#\|_\infty<\delta<1$.
The estimates obtained in (\ref{locminp2VMO}) and (\ref{usingmaxVMO}) lead to 
\begin{equation*}
\underset{\Omega}{\int}\left(F(\cdot,\nabla\overline{u}+\nabla\varphi)-F(\cdot,\nabla\overline{u})\right)\dx\geq  \frac{c_3}{2} \underset{\Omega}{\int}|\nabla\varphi|^2\dx-c\underset{\Omega}{\int}\omega(\tilde{c}|\nabla\varphi|)|\nabla\varphi|^2\dx.
\end{equation*}
We are now in the same situation as in (\ref{mintheosubquadVMO}) and we can conclude the proof in the same way that we did for $1<p\leq 2$. 
\end{proof}

\section{Regularity up to the boundary for a class of local minimizers}\label{SectBdryReg}

The main result in this section concerns regularity up to the boundary for two certain classes of local minimizers. On the one hand, it provides a partial boundary regularity result for strong local minimizers, extending the work of Kristensen-Taheri in \cite{Kristensen}, where interior partial regularity for strong local minimizers was shown. Their technique, that we adapt here four our purposes, relies on the blow-up method and on a remarkable modification to Evans' proof on interior partial regularity for the case of absolute minimizers  \cite{Evanspr}. 

On the other hand, the regularity up to the boundary that we prove here also holds for weak local minimizers with VMO derivative. In this case, full (and not partial) regularity is achieved. This immediately enables us to establish a connection with Weierstrass' sufficiency problem in the vectorial case. We discuss this in detail in Section \ref{SectGralizedSuff}.

In addition to $(\overline{{\mathrm{H}}}0)$-$(\overline{{\mathrm{H}}}2)$, from now on we also consider the following H\"older continuity assumption, which is standard in the treatment of interior and boundary regularity problems (see \cite{KristMingBdryreg}).

\begin{itemize}
\item[(${\mathrm{H}}$C)] $F$ and $F_z$ satisfy the following H\"older continuity properties: for some fixed $\alpha \in (0,1)$ and for every $x,y\in\overline{\Omega}$, $z\in\R^{N\times n}$, 
\begin{equation*}
\left\{
\begin{array}{l}
|F(x,z)-F(y,z)|\leq c_1|x-y|^\alpha(1+|z|^p);\\
|F_z(x,z)-F_z(y,z)|\leq c_1|x-y|^\alpha(1+|z|^{p-1}).
\end{array}
\right.
\end{equation*}
\end{itemize}

\begin{theorem}\label{BdryRegularityTheorem}
Let $\Omega$ be a bounded domain of class $C^{1,\alpha}$ and $g\in C^{1,\alpha}(\overline{\Omega},\R^N)$ for some $\alpha\in(0,1)$. Assume that  $F:\overline{\Omega}\times\R^{N\times n}\rightarrow\R$ is an integrand satisfying $(\overline{\mathrm{H}}0$)-($\overline{\mathrm{H}}2)$  and $(\mathrm{{H}C})$ for some $p\in [2,\infty)$. Suppose either one of the following:
\begin{itemize}
\item[(a)] there is a $q\in [1,\infty)$ such that $\overline{u}\in \W^{1,p}_g(\Omega,\R^N)\cap\W^{1,q}(\Omega,\R^{N\times n})$ is a $\W^{1,q}$-local minimizer or
\item[(b)] $\overline{u}\in \W^{1,\infty}_g(\Omega,\R^{N\times n})$ is a $\mathrm{BMO}$-local minimizer and $\nabla\overline{u}\in\mathrm{VMO}_{\nabla g}(\Omega,\R^{N\times n})$.
\end{itemize}
Then, $\overline{u}\in C^{1,\beta}(\Omega_0\cup \Sigma_0,\R^{N\times n})$ for every $\beta\in[0,\alpha)$, 
where 
\begin{equation*}\Omega_0:=\left\{x_0\in\Omega\st \underset{r\rightarrow 0}{\limsup} |(\nabla\overline{u})_{x_0,r}| < \infty \mbox {\,\,\, and \,\,\,} \underset{r\rightarrow 0}{\liminf} \underset{\Omega(x_0,r)}{\mint}|\nabla\overline{u}-(\nabla\overline{u})_{x_0,r}|^p\dx=0\right\},
\end{equation*}
\begin{equation*}\Sigma_0:=\left\{x_0\in\partial\Omega\st \underset{r\rightarrow 0}{\limsup} |(\nabla_\mathrm{n}\overline{u})_{x_0,r}| < \infty \mbox { \,\,\,and\,\,\, } \underset{r\rightarrow 0}{\liminf} \underset{\Omega(x_0,r)}{\mint}|\nabla\overline{u}-(\nabla\overline{u})_{x_0,r}|^p\dx=0\right\}.
\end{equation*}
In particular, under the assumptions in (b), $\overline{u}\in C^{1,\beta}(\overline{\Omega},\R^{N\times n})$ for  $\beta \in [0,\alpha)$.
\end{theorem}

\begin{remark}
We recall that not every VMO function is continuous. One of the multiple examples that show this and appear in \cite[Section I.2]{BrezisNirenbergBMOI} is the (bounded) function $f(x)= \sin(\log | \log |x||)$ for $x\in B$, which is of vanishing mean oscillation. 
\end{remark}

We note here that in \cite{Moser} the author states  that, if $u$ is a Lipschitz extremal of a rank one convex integrand of class $C^2$ and such that $\nabla u\in \mathrm{VMO}(\Omega,\R^N)$, then $u\in C^{1,\alpha}(\Omega,\R^N)$ for some $\alpha>0$. However, there is an inconsistency in the proof of \cite[Lemma 3.2]{Moser}, that does not appear to be easy to fix with the strategy suggested there. Nevertheless, the result holds true for quasiconvex integrands and can be obtained as a consequence of Theorems \ref{morethanWLMVMO} and \ref{BdryRegularityTheorem}. The use of the BMO-local minimality property for this class of extremals seems to provide the necessary estimates that lead to higher regularity via a Caccioppoli inequality.

Concerning boundary regularity, it is often convenient to work in the model situation in which the domain is assumed to be a half ball. We introduce the following notation for that purpose. 

\begin{definition}
Given $x_0\in \R^{n-1}\times \{0\}$ and $R>0$, we denote $ \Gamma:=B(x_0,R)\cap \R^{n-1}\times \{0\}$ and  define the Sobolev space
\begin{equation*}
\W^{1,p}_{\Gamma}(B^+(x_0,R),\R^N):=\left\{  v\in \W^{1,p}((B^+(x_0,R),\R^N)\st v=0\mbox{ on }\Gamma      \right\}.
\end{equation*}
\end{definition}

The following lemmata are well known in the boundary regularity theory and will be often be used in the subsequent sections. 

\begin{proposition}\label{Poincnormalderiv}
Let $p\in [1, \infty)$ and $v\in \W^{1,p}_\Gamma(B^+(x_0,R),\R^N)$ for some $x_0\in \R^{n-1}\times \{0\}$ and $R>0$. Then, 
\begin{equation*}
\underset{B^+(x_0,R)}{\mint}\left|\frac{v}{R}\right|^{p}\dx\leq\,\frac{1}{p}\underset{B^+(x_0,R)}{\mint}|\nabla_\mathrm{n} v|^{p}\dx.
\end{equation*}
\end{proposition}
The origins of this result can be traced back to \cite[Lemma 5.IV]{Campanatoeqelip}. We refer the reader to \cite[Lemma 3.4]{Beckregweaksols} for a nice exposition of the proof. 
\\

\begin{lemma}\label{quasiminimality}
Let $f\colon\Omega\rightarrow\R^{N\times n}$ be such that $f\in \LL^p(\Omega,\R^{N\times n})$ for some $p\in[1,\infty)$. Then, there is a constant $c>0$, depending only on $p$, such that for any $\omega\subseteq\Omega$ and every $\xi \in\R^{N\times n}$,
\begin{equation*}
\underset{\omega}{\int}\left|f-(f)_\omega\right|^p\dx\leq\,c\underset{\omega}{\int}\left|f-\xi\right|^p\dx.
\end{equation*}
In addition, for every $\eta\in\R^N$ and every $\nu_0\in\R^n$ such that $|\nu_0|=1$,
\begin{equation*}
\underset{\omega}{\int}\left|f-(f\cdot\nu_0)_\omega\otimes\nu_0\right|^p\dx\leq\,c\underset{\omega}{\int}\left|f-\eta\otimes\nu_0\right|^p\dx.
\end{equation*}
\end{lemma}
\begin{proof}
The proof of these estimates relies mainly on triangle inequality. We adapt the ideas from \cite{SchmidtregW1p} to prove the second part of the Lemma, given that the first part is analogous and is found more often in the literature. See also (16)-(18) in \cite{Kronzbdryreg}.
\\
Let $\nu_0\in \mathbb{S}^{n-1}$ and $\eta\in\R^N$ arbitrary. Then, 
\begin{align*}
\underset{\omega}{\mint}\left|f-(f\cdot\nu_0)_\omega\otimes\nu_0\right|^p\dx&\leq c\,\underset{\omega}{\mint}\left|f-\eta\otimes\nu_0\right|^p\dx+c\,\underset{\omega}{\mint}\left|((f\cdot\nu_0)_\omega-\eta)\otimes\nu_0\right|^p\dx.
\end{align*}
Observe that the second term in the right hand side of this inequality is the mean integral of a constant vector. Therefore, recalling that $|a\otimes b|=|a||b|$ and applying Jensen's inequality, we have that 
\begin{align*}
\left|((f\cdot\nu_0)_\omega-\eta)\otimes\nu_0\right|^p= &\left|(f\cdot\nu_0-\eta)_\omega\right|^p\notag\\
\leq&\,\underset{\omega}{\mint}|f\cdot\nu_0-\eta|^p\dx.
\end{align*} 
On the other hand, since $|\nu_0|=1$, we have $\eta=(\nu_0\cdot\nu_0)\eta=(\eta\otimes\nu_0)\cdot\nu_0$. Therefore, by Cauchy-Schwarz inequality,
\begin{align*}
\underset{\omega}{\mint}|f\cdot\nu_0-\eta|^p\dx=&\underset{\omega}{\mint}|f\cdot\nu_0-(\eta\otimes\nu_0)\cdot\nu_0|^p\dx\\
\leq &\underset{\omega}{\mint}|f-(\eta\otimes\nu_0)|^p\dx.
\end{align*}
By bringing together the three chains of inequalities above, we obtain the desired result.
\end{proof}

\subsection{Characterization of regular boundary points}\label{SectRegularSet}

The results in this section establish that the average value at boundary points of the gradient of Sobolev functions vanishing on the (smooth) boundary of their domain is essentially given by the average of the the normal derivative,  i.e., the averages of the tangential derivatives do not play a significant role. Despite their elementary nature, to the best of the author knowledge, they were not previously used systematically in this context. 

For smooth functions defined on the half unit ball $B^+(x_0,r)$, with $x_0\in\R^{n-1}\times\{0\}$, it is clear that the partial derivatives  with respect to the first $n-1$ variables vanish at at $x_0$. This is precisely what motivates the results below. 

The strategy that we follow is also intuitively clear:  due to the Divergence Theorem, for $1\leq i\leq n-1$ we can express $(\nabla_\mathrm{i}u(y))_{B^+(x_0,r)}$ in terms of a surface integral of the function $u$ and, since $u$ vanishes on $\Gamma$, by the Fundamental Theorem of Calculus we know that  the pointwise values of $u$ on $\partial B^+(x_0,r)$ depend exclusively on its derivative with respect to the $n$-th variable, from where the result will follow.

We remark that, although we only state the following lemma and corollary for $p>\frac{3}{2}$, this restriction comes exclusively from the curvature of the unit ball and the way in which the Jacobian of the corresponding transformation of variables blows up as we approach $\Gamma$. Hence, considering slightly modified domains that are still flat in the portion of their boundary normal to $\mathrm{e_n}$ could lead to improving this exponent. 

\begin{lemma}\label{Lemmitaaveragesgoto0}
Let $p>\frac{3}{2}$. There exists a constant $c=c(n,p)>0$ such that, for every  $u\in\W^{1,p}_\Gamma(B^+(x_0,r),\R^N)$ with $x_0\in\R^{n-1}\times\{0\}$ and $r>0$, the following inequality remains true for  $1\leq i\leq n-1$
\begin{equation*}
\left|\underset{B^+(x_0,r)}{\mint}\nabla_\mathrm{i}u\dx\right|\leq \,c\left(\underset{B^+(x_0,r)}{\mint}|\nabla_\mathrm{n}u|^p\dx\right)^{\frac{1}{p}}.
\end{equation*}
\end{lemma}
\begin{proof}
Without loss of generality we assume that $x_0=0$, the general case following then by a translation argument. 

By the Divergence Theorem,
\begin{align*}
\left|\underset{B^+(0,r)}{\int}\nabla_\mathrm{i}u(y)\dy\right|=\left| \underset{\partial B^+(0,r)}{\int}u(x)\frac{x_i}{r}\dv\mathcal{H}^{n-1}(x) \right|.
\end{align*}
We parametrize the curve $\partial B^+(0,r)\cap\partial B(0,r)$ using the usual Cartesian coordinates and the parametrization $\gamma:B_{\mathrm{n-1}}(0,r)\rightarrow \partial B^+(0,r)\cap\partial B(0,r) $ given by
\begin{equation*}
\gamma(x')=\gamma(x_1,...,x_{n-1})=\left(x',(r^2-|x'|^2)^{\frac{1}{2}}\right).
\end{equation*}
Using the Binet-Cauchy formula (see \cite[Section 3.2]{EvGar}), it is easy to see that the corresponding Jacobian is expressed by
\begin{equation*}
J\gamma(x')=\frac{r}{(r^2-|x'|^2)^\frac{1}{2}}.
\end{equation*}
For a function $u\in\W^{1,p}_\Gamma(B^+(0,r),\R^N)\cap C^1(B^+(0,r),\R^N)$, we evaluate the parametrization into the stated hypersurface integral and use the fact that $u(x',0)=0$ for $x'\in B_{\mathrm{n}-1}(0,r)$ to deduce that, for $1\leq i\leq n-1$,
\begin{align*}
\left|\underset{\partial B^+(0,r)}{\int}u(x)\frac{x_i}{r}\dv\mathcal{H}^{n-1}(x) \right|=&\left| \underset{ B_{\mathrm{n}-1}(0,r)}{\int}u\left(x', (r^2-|x'|^2)^{\frac{1}{2}}  \right)\frac{x_i}{(r^2-|x'|^2)^{\frac{1}{2}} }\dv x' \right|\notag\\
\leq &\underset{ B_{\mathrm{n}-1}(0,r)}{\int}\left| u\left(x', (r^2-|x'|^2)^{\frac{1}{2}}  \right)-u(x',0)\right|\frac{|x_i|}{(r^2-|x'|^2)^{\frac{1}{2}} }\dv x'\notag\\
=&\underset{ B_{\mathrm{n}-1}(0,r)}{\int}\left| \int_0^{(r^2-|x'|^2)^\frac{1}{2}} \nabla_\mathrm{n}u(x',t)\dt \right|\frac{|x_i|}{(r^2-|x'|^2)^{\frac{1}{2}} }\dv x'\notag\\
\leq& \underset{B^+(0,r)}{\int}|\nabla_\mathrm{n}u(x)|  \frac{|x_i|}{(r^2-|x'|^2)^{\frac{1}{2}} }\dx.
\end{align*}
We use Jensen and H\"{o}lder inequalities to conclude from above that, for $p'=\frac{p}{p-1}$,
\begin{align}\label{AfterHolder}
\left|\underset{B^+(0,r)}{\int}\nabla_\mathrm{i}u(y)\dy\right|& \leq \left(\underset{B^+(0,r)}{\int}\left|\nabla_\mathrm{n}u(x)\right|^p\dx  \right)^\frac{1}{p}\left(\underset{B^+(0,r)}{\int}\frac{|x_i|^{p'}}{(r^2-|x'|^2)^{\frac{p'}{2}} }\dx  \right)^\frac{1}{p'}.
\end{align}
We are only left with estimating the second factor above. Since we are assuming that $1\leq i\leq n-1$, such a factor does not depend on the variable $x_n$. Therefore, 
\begin{align}\label{intinn-1}
\underset{B^+(0,r)}{\int}\frac{|x_i|^{p'}}{(r^2-|x'|^{2})^{\frac{p'}{2}} }\dx =& \underset{ B_{\mathrm{n}-1}(0,r)}{\int}\int_0^{(r^2-|x'|^2)^\frac{1}{2}}\frac{|x_i|^{p'}}{(r^2-|x'|^2)^{\frac{p'}{2}}}\dv x_n\dx'\notag\\
=&\underset{ B_{\mathrm{n}-1}(0,r)}{\int}\frac{|x_i|^{p'}}{(r^2-|x'|^2)^\frac{p'-1}{2}}\dx'\notag\\
\leq &\underset{ B_{\mathrm{n}-1}(0,r)}{\int}\frac{|x'|^{p'}}{(r^2-|x'|^2)^\frac{{p'}-1}{2}}\dx'.
\end{align}
This is the integral of a radial function in $\R^{n-1}$ and can therefore be computed, for $p'< 3$, as

\begin{align}\label{IntofRadF}
\underset{ B_{\mathrm{n}-1}(0,r)}{\int}\frac{|x'|^{p'}}{(r^2-|x'|^2)^\frac{{p'}-1}{2}}\dx'=n\,\omega_\mathrm{n}\int_0^r\frac{\rho^{p'}}{(r^2-\rho^2)^\frac{{p'}-1}{2}}\rho^{n-2}\dv\rho=c(n,p')r^{n}.
\end{align}
The desired inequality is obtained, for $p>\frac{3}{2}$, from (\ref{AfterHolder})-(\ref{IntofRadF}) after taking averages. 

The general case of a function in $\W^{1,p}_\Gamma(B^+(0,r),\R^N)$ follows by approximation. 
\end{proof}

\begin{corollary}\label{corollequivnormaldec}
Let $p>\frac{3}{2}$. There exists a constant $c_*=c_*(n,p)>0$ such that, for every  $u\in\W^{1,p}_\Gamma(B^+(x,r),\R^N)$ with $x\in\R^{n-1}\times\{0\}$,
\begin{equation*}
\underset{B^+(x,r)}{\mint}\left|\nabla u - (\nabla_\mathrm{n} u)_{x,r}\otimes \mathrm{e}_\mathrm{n}\right|^p\dy\leq c_*\underset{B^+(x,r)}{\mint}\left|\nabla u - (\nabla u)_{x,r}\right|^p\dy.
\end{equation*}

\end{corollary}
\begin{proof}
Let $x\in \R^{n-1}\times\{0\}$ and $r>0$. Define the function $\tilde{u}\colon B^+(x,r)\rightarrow\R^N$ by $\tilde{u}(y):= u(y)-(\nabla_\mathrm{n}u)_{x,r}y_n$.
Observe that $\tilde{u}=0$ on $\Gamma$ and, for $1\leq i \leq n-1$, $\nabla_\mathrm{i}\tilde{u}=\nabla_\mathrm{i}{u}$. Then, 
\begin{align*}
&\underset{B^+(x,r)}{\mint}\left|\nabla u - (\nabla_\mathrm{n} u)_{x,r}\otimes \mathrm{e}_\mathrm{n}\right|^p\dy\notag\\
\leq&\,c(p) \underset{B^+(x,r)}{\mint}\left|\nabla u - (\nabla_1 u)_{x,r}\otimes \mathrm{e}_1 -(\nabla_2 u)_{x,r}\otimes \mathrm{e}_2 -...- (\nabla_\mathrm{n} u)_{x,r}\otimes \mathrm{e}_\mathrm{n}\right|^p\dy\notag\\
&+\,c(p)\left(\left| (\nabla_1 u)_{x,r}     \right|^p+\left| (\nabla_2 u)_{x,r}     \right|^p+...+\left| (\nabla_\mathrm{n-1} u)_{x,r}     \right|^p\right) \notag\\
=&\,c(p)\underset{B^+(x,r)}{\mint}\left|\nabla u - (\nabla u)_{x,r}\right|^p\dy + \,c(p)\left(\left| (\nabla_1 \tilde{u})_{x,r}     \right|^p+\left| (\nabla_2 \tilde{u})_{x,r}     \right|^p+...+\left| (\nabla_\mathrm{n-1} \tilde{u})_{x,r}     \right|^p\right).
\end{align*}
We now apply Lemma  \ref{Lemmitaaveragesgoto0} to the function $\tilde{u}$ and conclude from above that
\begin{align*}
&\underset{B^+(x,r)}{\mint}\left|\nabla u - (\nabla_\mathrm{n} u)_{x,r}\otimes \mathrm{e}_\mathrm{n}\right|^p\dy\notag\\
\leq&\,c(p)\underset{B^+(x,r)}{\mint}\left|\nabla u - (\nabla u)_{x,r}\right|^p\dy + \,c(n,p)\underset{B^+(x,r)}{\mint}|\nabla_\mathrm{n} u -(\nabla_\mathrm{n}u)_{x,r}|^p\dy\notag\\
\leq&\, c(n,p)\underset{B^+(x,r)}{\mint}\left|\nabla u - (\nabla u)_{x,r}\right|^p\dy.
\end{align*}
This proves the claim. 
\end{proof}
\begin{remark}
This inequality has as an immediate consequence the possibility of restating the definition of the regular set  for boundary regularity results already present in the literature. In what follows we refer to two interesting examples in which Corollary \ref{corollequivnormaldec} enables us to reformulate such definition of the regular set.

\begin{itemize}
\item[I.] In \cite[Theorem 3.1]{Grotowskinonlin}, Grotowski shows that for $\Omega$ and $g$ as in Theorem  \ref{BdryRegularityTheorem} and for coefficients $A:\Omega\times\R^N\times \R^{N\times n}\rightarrow\R^{N\times n}$ which are smooth, with bounded derivative $|A_p(x,\xi,p)|\leq L$,  uniformly strongly elliptic and satisfying a H\"older continuity assumption similar to (HC), weak solutions  $u\in \W^{1,2}_g(\Omega,\R^N)$ to the homogeneous system
\begin{equation*}
\mathrm{div}\,A(\cdot,u,\nabla u)=0
\end{equation*}
are H\"older continuous on a neighbourhood in $\overline{\Omega}$ of all points $y\in\partial \Omega$ satisfying 
\begin{equation}\label{normalcondRegBdryPoints}
\underset{r\rightarrow 0}{\liminf}\underset{\Omega(y,r)}{\mint}|\nabla u-\nabla g-(\nabla_{\nu(y)}(u-g))_{y,r}\otimes\nu(y)|^2\dx=0,
\end{equation}
where $\nu(y)$ is the outward-pointing unit normal vector to $\partial\Omega$ at the point $y$. 

Using Corollary \ref{corollequivnormaldec}, condition (\ref{normalcondRegBdryPoints}) can be relaxed, after transforming $\partial\Omega$ back to its original shape, to requiring 

\begin{equation}\label{improvedcondRegBdryPoints}
\underset{r\rightarrow 0}{\liminf}\underset{\Omega(y,r)}{\mint}|\nabla u-\nabla g-(\nabla u-\nabla g)_{y,r}|^2\dx=0.
\end{equation}
\item[II.] An equivalent argument can be applied for the characterization of the regular set of boundary points in \cite[Theorem 1.1]{BeckVarInt}. We recall that Beck's result concerns $\W^{1,p}$-local minimizers for functionals of the type
\begin{equation*}
u \mapsto \underset{\Omega}{\int}\left(F(\cdot,u,\nabla u)+h(x,u)\right)\dx
\end{equation*}
over the class $\W^{1,p}_g(\Omega,\R^N)$ under essentially the same assumptions over $\Omega$ and $g$ that we have mentioned before. In addition, $F$ satisfies essentially the corresponding versions of hypotheses $(\mathrm{\overline{H}}0)$-$(\mathrm{\overline{H}}2)$ and $(\mathrm{HC})$ that are valid for integrands of the type $F(x,u,z)$, while $h(x,u)$ is assumed to be a Carath\'eodory function.
\end{itemize}
Finally, we remark that characterizing the regular points as in (\ref{improvedcondRegBdryPoints}) results interesting also in view of the available strategies of dimension reduction (see \cite[Ch. VIII]{Giaq}, \cite[Ch. 9]{Giusti}, \cite{MingioneBoundsSingSet}), to establish the existence of at least one regular boundary point. To the best of the author's knowledge, such existence is only known for integrands satisfying a strong convexity condition, but the required estimates seem to be much harder to obtain in the quaxiconvex setting. See \cite{JostMeierBdryReg,DKM,KristMingBdryreg}.
\end{remark}

\subsection{Proof of Theorem \ref{BdryRegularityTheorem}}

This section is devoted to the proof of the regularity up to the boundary result, Theorem \ref{BdryRegularityTheorem}. 

We begin by establishing the following technical lemma regarding a shifted version of the integrand $F$ defining the functional. The strategy that we follow to prove it is fully inspired in the truncation technique from \cite[Lemma 2.3]{AcFus}, which is also used for the shifted functionals in the partial regularity result in \cite{Kristensen}. However, in order to tackle the problem of boundary regularity for $\W^{1,q}$-local minimizers, we need to adapt Kristensen-Taheri's strategy in such a way that the blown-up sequence near the boundary can still be shown to satisfy a suitable Caccioppoli inequality. 

Since each element of a blown-up sequence on the boundary will depend on the normal vector at each point, the estimates for the functional have to be treated as if there was certain $x$-dependence, even if the original integrand were homogeneous. A natural strategy to simplify the problem is then to flatten the boundary before constructing the  blown-up sequence. Having done this, we need to define the shifted integrands given in Lemma \ref{LemmaGrowths0} in order to exploit all our assumptions. The Lemma is stated for the model situation $\Omega=B ^+$. The integrand $F_j$ in the statement corresponds to the standard shifted integrand that enables us to reduce the problem to integrands with growth of the type $|F(x,z)|\leq c\left( |z|^2+|z|^p  \right)$. On the other hand, $\tilde{F}_j$ enables us to use the minimality of $\overline{u}$, essential to obtain the desired Caccioppoli inequality. This is so, thanks to the lack of $x$-dependence in the term $F_z$ that appears in the definition of $\tilde{F}_j$. Finally, the auxiliary integrand $G_j$ enables us to apply the strong quasiconvexity condition as defined in $(\mathrm{\overline{H}}2)$, given that it does not depend on $x$ in any way. We will be able to see all these properties in action at Step 2.1 of the proof of Proposition \ref{PropBoundaryDecay} below.

\begin{lemma}\label{LemmaGrowths0}

Let $(\lambda_j)$ be a sequence in $(0,\infty)$ converging  to $0$, $(x_j)\subseteq  \Gamma$ and $(r_j)$ a sequence of radii such that $B^+(x_j,r_j)\subseteq B^+$  and satisfying $\frac{r_j^\alpha}{\lambda_j}\leq 1$. In addition, let $(\xi_j)\subseteq\R^{N\times n}$ be a sequence with $|\xi_j|\leq m$ for some given $m>0$. Let $F\colon \overline{B^+}\rightarrow\R$ satisfy $\mathrm{(\overline{H}0)}-\mathrm{(\overline{H}2)}$ and $\mathrm{(HC)}$ adapted to the model situation $\Omega=B^+$. Consider the functionals $F_j\colon B^+\times \R^{N\times n}\rightarrow\R^{N\times n}$, $\tilde{F}_j\colon B^+\times \R^{N\times n}\rightarrow\R$ and $G_j\colon \R^{N\times  n}\rightarrow\R$ given by
\begin{align}
F_j(x,z):=&\frac{F(x_j+r_jx,\xi_j+\lambda_jz)-F(x_j+r_jx,\xi_j)-\lambda_jF_{z}(x_j+r_jx,\xi_j)[z]}{\lambda_j^2}\notag\\
=&\int_0^1(1-t)F_{zz}(x_j+r_jx,\xi_j+t\lambda_jz)[z,z];\notag
\end{align}

\begin{align}
\tilde{F}_j(x,z):=&\frac{F(x_j+r_jx,\xi_j+\lambda_jz)-F(x_j+r_jx,\xi_j)-\lambda_j F_{z}(x_j,\xi_j)[z]}{\lambda_j^2};\notag
\end{align}
and
\begin{align}
G_j(z):=&\frac{F(x_j,\xi_j+\lambda_jz)-F(x_j,\xi_j)-\lambda_jF_{z}(x_j,\xi_j)[z]}{\lambda_j^2}\notag\\
=&\int_0^1(1-t)F_{zz}(x_j,\xi_j+t\lambda_jz)[z,z].\notag
\end{align}

Then, the following estimates remain true for every $x\in B^+$ and every $z,w\in\R^{N\times n}$:
\begin{itemize}
\item[$\mathrm{(i)}$] $|F_j(x,z)-\tilde{F}_j(x,z)|\leq c(m)\frac{r_j^{\alpha}}{\lambda_j}(1+|z|^2) $;
\item[$\mathrm{(ii)}$] $|\tilde{F}_j(x,z)-\tilde{F}_j(x,w)|\leq c(m)(\frac{r_j^{\alpha}}{\lambda_j}+|w|+|z|+\lambda_j^{p-2}(|w|^{p-1}+|z|^{p-1}))|z-w|$;
\item[$\mathrm{(iii)}$] $|F_j(x,z)-G_j(z)|\leq\,c(m)\frac{r_j^{\alpha}}{\lambda_j}(1+|z|^2+\lambda_j^{p-2}|z|^{p})$;
\item[$\mathrm{(iv)}$] $|{F}_j(x,z)-{F}_j(x,w)|\leq c(m)\left(|z|+|w|+\lambda_j^{p-2}(|z|^{p-1}+|w|^{p-1})\right)|z-w|$.
\end{itemize}
\end{lemma}

\begin{proof}

For $x\in B^+$ and $z\in\R^{N\times n}$, 
\begin{itemize}
\item[($\mathrm{i}$)] $|F_j(x,z)-\tilde{F}_j(x,z)|=\frac{1}{\lambda_j}|F_z (x_j+r_jx,\xi_j)-F_z (x_j,\xi_j)||z|\leq c(m,\Omega)\frac{r_j^\alpha}{\lambda_j}|z|\leq c(m)\frac{r_j^{\alpha}}{\lambda_j}(1+|z|^2) $.
\item[($\mathrm{ii}$)] By the Fundamental Theorem of Calculus, we have
\begin{align*}
&|\tilde{F}_j(x,z)-\tilde{F}_j(x,w)|\notag\\
=&\frac{1}{\lambda_j}\left|\int_0^1 \left(F_z(x_j+r_jx,\xi_j+ \lambda_jw+t\lambda_j(z-w))[z-w]- {F_z(x_j,\xi_j)}[z-w]\right)\dt\right|\notag\\
\leq& \frac{1}{\lambda_j}\int_0^1 \left|\left(F_z(x_j+r_jx,\xi_j+ \lambda_jw+t\lambda_j(z-w))-F_z(x_j,\xi_j+ \lambda_jw+t\lambda_j(z-w))\right)[z-w]\right|\dt\notag\\
&+\frac{1}{\lambda_j}\int_0^1 \left|\left(F_z(x_j,\xi_j+ \lambda_jw+t\lambda_j(z-w))-F_z(x_j,\xi_j) \right)[z-w]\right|\dt\notag\\
=:&\,\mathrm{I+II}.
\end{align*}
Then, (HC)$_2$ implies 
\begin{equation}\label{growthslemmaI}
\mathrm{I}\leq\, c(m)\,\frac{r_j^\alpha}{\lambda_j}(1+\lambda_j^{p-1}|w|^{p-1}+\lambda_j^{p-1}|z-w|^{p-1})|z-w|.
\end{equation}
Regarding II, for each $t\in (0,1)$ we consider the two following cases 
\\
\textit{Case 1.} $|\lambda_jw|+|\lambda_jt(z-w)|\leq 1$. Then, the identity
\begin{align*}
&\left(F_z(x_j,\xi_j+ \lambda_jw+t\lambda_j(z-w))-F_z(x_j,\xi_j) \right)[z-w]\notag\\
=&\int_0^1F_{zz}(x_j,\xi_j+\lambda_js(w+t(z-w)))[z-w,\lambda_j(w+t(z-w))]\dv s
\end{align*}
implies, in this case, 
\begin{equation*}
\mathrm{II}\leq c(m)\left(|w|+|z-w|\right)|z-w|.
\end{equation*}
\textit{Case 2.} $|\lambda_jw|+|\lambda_jt(z-w)|> 1$. Then, we use ($\mathrm{\overline{H}}$1)-($\mathrm{\overline{H}}$2) to derive that
\begin{align*}
\mathrm{II}\leq& \frac{1}{\lambda_j}\int_0^1 \left( |F_z(x_j,\xi_j+ \lambda_jw+t\lambda_j(z-w))| +|F_z(x_j,\xi_j) | \right) |z-w|\dt\\
\leq& \frac{c(m)}{\lambda_j}(1+|\lambda_jw|^{p-1}+|\lambda_j(z-w)|^{p-1})|z-w|\\
\leq &\frac{c(m)}{\lambda_j}\left(|\lambda_jw|+|\lambda_jt(z-w)|+|\lambda_jw|^{p-1}+|\lambda_j(z-w)|^{p-1}\right)|z-w|.
\end{align*}
The previous inequalities lead to conclude
\begin{equation}\label{growthslemmaII}
\mathrm{II}\leq c(m)\left(|w|+|z-w|+\lambda_j^{p-2}\left(|w|^{p-1}+|z-w|^{p-1}\right) \right)|z-w|.
\end{equation}
Given that $\frac{r_j^\alpha}{\lambda_j}\leq 1$, equations (\ref{growthslemmaI}) and (\ref{growthslemmaII}) imply the claim, after using triangle inequality.
\item[($\mathrm{iii}$)] We have
\begin{align}
&|F_j(x,z)-G_j(z)|\notag\\
\leq&\frac{1}{\lambda_j^2}\left| F(x_j+r_jx,\xi_j+\lambda_jz)- F(x_j+r_jx,\xi_j)+ F(x_j,\xi_j+\lambda_jz) - F(x_j,\xi_j)  \right|\notag\\
&+\frac{1}{\lambda_j}|F_z(x_j+r_jx,\xi_j)-F_z(x_j,\xi_j)||z|\notag\\
\leq& \frac{1}{\lambda_j}\int_0^1 | F_z (x_j+r_jx,\xi_j+t\lambda_jz)-F_z(x_j,\xi_j+t\lambda_jz) ||z|\dt\notag\\
&+\frac{1}{\lambda_j}|F_z(x_j+r_jx,\xi_j)-F_z(x_j,\xi_j)||z|\notag\\
\leq&\,c(m) \frac{r_j^\alpha}{\lambda_j}(1+|\lambda_jz|^{p-1})|z|\notag\\
\leq&\,c(m)\frac{r_j^\alpha}{\lambda_j}(1+|z|^2+\lambda_j^{p-2}|z|^{p}).\notag
\end{align}

\item[($\mathrm{iv}$)] This part of the lemma can be established exactly as in \cite[Lemma II.3]{AcFus}. 

\end{itemize}
This concludes the proof.
\end{proof}

The following lemma will be used to construct suitable test functions to exploit the BMO-minimality of $\overline{u}$ in part (b) of Theorem \ref{BdryRegularityTheorem}. The idea is to obtain control on the mean oscillations of the proposed test functions in terms of those of $\nabla \overline{u}$, whose mean oscillations will be under control, in that particular case, because it is a VMO function. 

\begin{lemma}\label{LemmaBMOforPhi}
Let $q>n$, $B(x_0,R)\subseteq \R^n$ and, for $\varsigma\in (0,1)$, let $\rho\colon\R^n\rightarrow[0,1]$ be a cut-off function such that $\mathbbm{1}_{B(x_0,\varsigma R)}\leq \rho\leq \mathbbm{1}_{B(x_0,R)}$ and $\|\nabla\rho\|_{\LL^\infty}\leq \frac{1}{(1-\varsigma)R}$. Suppose either one of the following situations:
\begin{itemize}
\item[(i)] \textit{($\mathrm{BMO}$ estimates for cut-offs near the boundary).} The function ${v}\in\W^{1,q}_\Gamma(B^+(x_0,R),\R^N)$ takes the form
\begin{equation*}
v(y):= w(y)-y_n(\nabla_\mathrm{n}w)_{x_0,R}
\end{equation*}
for some $w\in\W^{1,q}_\Gamma(B^+(x_0,R),\R^N)$, $x_0\in \R^{n-1}\times\{0\}$. Furthermore, we assume  $\tilde{v}$ to be the extension of $v$ to the whole ball $B(x_0,R)$ that assigns the value of $0$ outside $B^+(x_0,R)$. 

\item[(ii)] \textit{($\mathrm{BMO}$ estimates for cut-offs in the interior).} The function $\tilde{v}\in\W^{1,q}(B(x_0,R),\R^N)$ takes the form
\begin{equation*}
\tilde{v}(y):= w(y)-b(y)
\end{equation*}
for some $w\in\W^{1,q}(B(x_0,R),\R^N)$ and  an affine function $b$ satisfying $\nabla b=(\nabla w)_{x_0,R}$ and $(w-b)_{x_0,R}=0$.
\end{itemize} 

Define $\varphi\colon B(x_0,R)\rightarrow\R^N$ by $\varphi:= \rho \tilde{v}$. Then, there exists a constant $c=c(n,q,\varsigma)>0$ such that, for any $B(x,s)\subseteq B(x_0,R)$,
\begin{align*}
\underset{B(x,s)}{\mint}|\nabla \varphi-(\nabla\varphi)_{x,s}|\dy\leq c\underset{B(x,s)}{\mint}|\nabla \tilde{v}-(\nabla \tilde{v})_{x,s}|\dy+c\left(\underset{B(x_0,R)}{\mint}|\nabla\tilde{v}-(\nabla\tilde{v})_{x_0,R}|^q\dy\right)^\frac{1}{q}.
\end{align*}
\end{lemma}
\begin{proof}
We will first prove that, for both cases (i) and (ii),  there exists a constant $c=c(n,q,\varsigma)>0$ such that, for any $B(x,s)\subseteq B(x_0,R)$,
\begin{align}\label{Claim1forBMOPhi}
\underset{B(x,s)}{\mint}|\nabla \varphi-(\nabla\varphi)_{x,s}|\dy\leq c\underset{B(x,s)}{\mint}|\nabla \tilde{v}-(\nabla \tilde{v})_{x,s}|\dy+c\left(\underset{B(x_0,R)}{\mint}|\nabla\tilde{v}|^q\dy\right)^\frac{1}{q}.
\end{align}
We remark that, for the proof of (\ref{Claim1forBMOPhi}), we only use (by an extension argument) that $\tilde{v}$ can be assumed to be in $\W^{1,q}(\R^n,\R^N)$ and that there must exist $x\in B(x_0,R)$ such that $\tilde{v}(x)=0$. This holds for both cases (i) and (ii). Hence, at this stage, we do not use the definition of $\tilde{v}$ as an affine function (which changes depending on whether we are in case (i) or (ii)). 

Motivated by the ideas of Stegenga \cite{Stegenga}, we first note that
\begin{align*}
&\left|\underset{B(x,s)}{\mint}|\rho\nabla \tilde{v}-(\rho\nabla \tilde{v})_{x,s}|\dy-\left|(\nabla \tilde{v})_{x,s}\right|\underset{B(x,s)}{\mint}|\rho-(\rho)_{x,s}|\dy\right| \nonumber\\
\leq&\underset{B(x,s)}{\mint}|(\rho\nabla \tilde{v}-(\rho\nabla \tilde{v})_{x,s})-(\nabla \tilde{v})_{x,s}(\rho-(\rho)_{x,s})|\dy\nonumber\\
\leq&\underset{B(x,s)}{\mint}|\rho(\nabla \tilde{v}-(\nabla \tilde{v})_{x,s})|\dy+(\rho)_{x,s}(\nabla \tilde{v})_{x,s}-(\rho\nabla \tilde{v})_{x,s}|\nonumber\\
\leq&\|\rho\|_{\LL^\infty}\underset{B(x,s)}{\mint}|\nabla \tilde{v}-(\nabla \tilde{v})_{x,s}|\dy+|(\rho)_{x,s}(\nabla \tilde{v})_{x,s}-(\rho\nabla \tilde{v})_{x,s}|.
\end{align*}
We estimate the second term above by
\begin{align*}
\left|(\rho)_{x,s}(\nabla \tilde{v})_{x,s}-(\rho\nabla \tilde{v})_{x,s}\right|=&\left|\underset{B(x,s)}{\mint}\rho(\nabla \tilde{v})_{x,s}\dy-\underset{B(x,s)}{\mint}\rho\nabla \tilde{v}\dy\right|\nonumber\\
\leq &\underset{B(x,s)}{\mint}|\rho||\nabla \tilde{v}-(\nabla \tilde{v})_{x,s}|\dy\nonumber\\
\leq&\|\rho\|_{\LL^\infty}\underset{B(x,s)}{\mint}|\nabla \tilde{v}-(\nabla \tilde{v})_{x,s}|\dy,
\end{align*}
from where we deduce that
\begin{align}
\left|\underset{B(x,s)}{\mint}|\rho\nabla \tilde{v}-(\rho\nabla \tilde{v})_{x,s}|\dy-\left|(\nabla \tilde{v})_{x,s}\right|\underset{B(x,s)}{\mint}|\rho-(\rho)_{x,s}|\dy\right|&\leq 2\underset{B(x,s)}{\mint}|\nabla \tilde{v}-(\nabla \tilde{v})_{x,s}|\dy.\nonumber
\end{align}
In addition, 
\begin{align*}
\left|(\nabla \tilde{v})_{x,s}\underset{B(x,s)}{\mint}|\rho-(\rho)_{x,s}|\dy\right|=&\left|(\nabla \tilde{v})_{x,s}\right|\left|\underset{B(x,s)}{\mint}\left|\rho-\underset{B(x,s)}{\mint}\rho(z)\dv{z}\right|\dy\right|\nonumber\\
\leq&\left|(\nabla \tilde{v})_{x,s}\right|\underset{B(x,s)}{\mint}\underset{B(x,s)}{\mint}|\rho(y)-\rho(z)|\dv{z}\dy\nonumber\\
\leq&\left|(\nabla \tilde{v})_{x,s}\right|\,\|\nabla\rho\|_{\LL^\infty}2s.
\end{align*}
The Divergence Theorem implies that
\begin{align*}
\left|(\nabla \tilde{v})_{x,s}\right|\,\|\nabla\rho\|_{\LL^\infty}2s
\leq&\frac{2}{(1-\varsigma)R}\left|\underset{\partial B(x,s)}{\mint}\tilde{v}\otimes\frac{x-y}{s}\right|\dv{}\mathcal{H}^{n-1}(y)\nonumber\\
\leq&\frac{2n}{(1-\varsigma)R}\underset{\partial B(x,s)}{\mint}\left|\tilde{v}\right|\dv{}\mathcal{H}^{n-1}(y)\nonumber\\
\leq& \frac{2n}{(1-\varsigma)R}\|\tilde{v}\|_{\LL^\infty(B(x_0,R),\R^N)}.
\end{align*}
Furthermore, by Morrey's Embedding Theorem\footnote{See \cite[Theorem 9.12]{Brezisbook}.} we infer, using the property $\tilde{v}(x)=0$ for some $x\in B(x_0,R)$, that
\begin{equation}\label{MorreyforBMO}
\|\tilde{v}\|_{\LL^\infty(B(x_0,R),\R^N)}\leq cR^{1-\frac{n}{q}}\|\nabla\tilde{v}\|_{\LL^q(B(x_0,R),\R^N)}=cR\left(\underset{B(x_0,R)}{\mint}|\nabla\tilde{v}|^q\dy\right)^\frac{1}{q}
\end{equation}
for some $c=c(n,q,\varsigma)>0$.

All the previous inequalities brought together imply that
\begin{equation}\label{rhonablav}
\left|\underset{B(x,s)}{\mint}|\rho\nabla \tilde{v}-(\rho\nabla \tilde{v})_{x,s}|\dy\right|\leq c\underset{B(x,s)}{\mint}|\nabla \tilde{v}-(\nabla \tilde{v})_{x,s}|\dy+c\left(\underset{B(x_0,R)}{\mint}|\nabla\tilde{v}|^q\dy\right)^\frac{1}{q}.
\end{equation}

On the other hand, following an analogous argument to the one above, we deduce
\begin{align}\label{nablarhov}
&\left|\underset{B(x,s)}{\mint}|\nabla\rho\otimes\tilde{v}-(\nabla\rho\otimes \tilde{v})_{x,s}|\dy-|(\tilde{v})_{x,s}| \underset{B(x,s)}{\mint}|\nabla\rho-(\nabla\rho)_{x,s}|\dy\right|\nonumber\\
\leq& \,2\|\nabla\rho\|_{\LL^\infty}\underset{B(x,s)}{\mint}|\tilde{v}-(\tilde{v})_{x,s}|\dy\leq \,\frac{2}{(1-\varsigma)R}\underset{B(x,s)}{\mint}|\tilde{v}-(\tilde{v})_{x,s}|\dy\nonumber\\
\leq&\,\frac{4}{(1-\varsigma)R}\|\tilde{v}\|_{\LL^\infty(B(x,s),\R^N)}.
\end{align}
It is also easily seen that
\begin{equation}\label{afternablarhov}
|(\tilde{v})_{x,s}| \underset{B(x,s)}{\mint}|\nabla\rho-(\nabla\rho)_{x,s}|\dy\leq \frac{2}{(1-\varsigma)R}\|\tilde{v}\|_{\LL^\infty(B(x,s),\R^N)}\leq c\left(\underset{B(x_0,R)}{\mint}|\nabla\tilde{v}|^q\dy\right)^\frac{1}{q},
\end{equation}
where the last inequality is also a consequence of (\ref{MorreyforBMO}).

Inequality (\ref{Claim1forBMOPhi}) now follows directly from inequalities (\ref{rhonablav})-(\ref{afternablarhov}). 

To conclude the proof of the Lemma, we treat separately the two possible scenarios under consideration. 

\begin{itemize}
\item[(i)] Assume first that, for some $w\in\W^{1,q}_\Gamma(B^+(x_0,R),\R^N)$, $v\colon B^+(x_0,R)\rightarrow\R^N$ takes the form
\begin{equation*}
v(y):= w(y)-y_n(\nabla_\mathrm{n}w)_{x_0,R}.
\end{equation*}

Then, we apply the already proved inequality (\ref{Claim1forBMOPhi}) to the function $\tilde{v}$, to conclude that
\begin{align*}
\underset{B(x,s)}{\mint}|\nabla \varphi-(\nabla\varphi)_{x,s}|\dy\leq & c\underset{B(x,s)}{\mint}|\nabla \tilde{v}-(\nabla \tilde{v})_{x,s}|\dy+c\left(\underset{B(x_0,R)}{\mint}|\nabla\tilde{v}|^q\dy\right)^\frac{1}{q}.\nonumber
\end{align*}
Observe further that, since $w=0$ on $\Gamma$, we can also consider its corresponding extension $\tilde{w}$ defined on $B(x_0,R)$.  We now use triangle inequality followed by applying Lemma \ref{Lemmitaaveragesgoto0} to $\tilde{w}$ and with $p=q$ to estimate
\begin{align*}
\left(\underset{B(x_0,R)}{\mint}|\nabla\tilde{v}|^q\dy\right)^\frac{1}{q}\leq& c\left(\underset{B(x_0,R)}{\mint}|\nabla\tilde{v}-(\nabla\tilde{v})_{x_0,R}|^q\dy\right)^\frac{1}{q}+c\sum_{i=0}^{n}|(\nabla_\mathrm{i}\tilde{v})_{x_0,R}|\notag\\
\leq& c\left(\underset{B(x_0,R)}{\mint}|\nabla\tilde{v}-(\nabla\tilde{v})_{x_0,R}|^q\dy\right)^\frac{1}{q}+c{n}\left(\underset{B(x_0,R)}{\mint}|\nabla_\mathrm{n}\tilde{v}|^q\dy\right)^\frac{1}{q}\notag\\
=&c\left(\underset{B(x_0,R)}{\mint}|\nabla\tilde{v}-(\nabla\tilde{v})_{x_0,R}|^q\dy\right)^\frac{1}{q}+c{n}\left(\underset{B(x_0,R)}{\mint}|\nabla_\mathrm{n}\tilde{w}-(\nabla_\mathrm{n}\tilde{w})_{x_0,R}|^q\dy\right)^\frac{1}{q}\notag\\
=&c\left(\underset{B(x_0,R)}{\mint}|\nabla\tilde{v}-(\nabla\tilde{v})_{x_0,R}|^q\dy\right)^\frac{1}{q}+c{n}\left(\underset{B(x_0,R)}{\mint}|\nabla_\mathrm{n}\tilde{v}-(\nabla_\mathrm{n}\tilde{v})_{x_0,R}|^q\dy\right)^\frac{1}{q}\notag\\
 \leq &c\left(\underset{B(x_0,R)}{\mint}|\nabla\tilde{v}-(\nabla\tilde{v})_{x_0,R}|^q\dy\right)^\frac{1}{q}.
\end{align*}
This concludes the proof of part (i) of the Lemma. 
\item[(ii)] If we now assume that $\tilde{v}(y):= w(y)-b(y)$ with $w$ and $b$ as in the statement of the Lemma, then clearly $(\tilde{v})_{x_0,R}=0$ and, therefore, the conclusion of the result follows immediately from inequality (\ref{Claim1forBMOPhi}).
\end{itemize}
\end{proof}

We also remark that these estimates will enable us to  make use of the corresponding $\mathrm{BMO}$-local minimality to prove the following decay rates, from which Theorem \ref{BdryRegularityTheorem} (b) will be obtained.

The following propositions are well known to provide the precise decay rate needed to obtain regularity on the set $\Omega_0\cup \Sigma_0$, once we have applied Corollary \ref{corollequivnormaldec}. The strategy that we follow is largely inspired by the regularity proof in \cite{Kristensen}.

\newpage
\begin{proposition}[Decay estimate for boundary points]\label{PropBoundaryDecay}
Let $\Omega= B^+$ and $g\in C^{1,\alpha}(\overline{ \Omega},\R^N)$. Assume that $F\colon\overline{\Omega}\times\R^{N\times n}\rightarrow\R$ and $\overline{u}\in\W^{1,p}_g(\Omega,\R^N)$ are as in Theorem \ref{BdryRegularityTheorem}. For $\delta\in (0,1)$, $x\in  \Gamma$ and $r<1-|x|$, denote
\begin{equation}\label{excessboundary}
E(x,r):=\underset{B^+(x,r)}{\mint}\left|V\left(\nabla \overline{u}-(\nabla_\mathrm{n}\overline{u})_{x,r}\otimes \mathrm{e_n}\right) \right|^2 \dy+r^{2\delta\alpha}.
\end{equation} 
Then, for every $m > 0$ there exists $C = C(m) > 0$ with the property
that, for each $\tau\in  (0, \frac{1}{2} )$ and $\delta\in (0,1)$, there exists $\varepsilon = \varepsilon(m, \tau, \delta ) > 0$ such that, if $x\in \Gamma$, $|(\nabla_\mathrm{n} \overline{u})_{x,r}|\leq m$ and $E(x,r)<\varepsilon$, then 
\begin{equation*}
E(x,\tau r)<C\tau^{2\delta\alpha}E(x,r). 
\end{equation*}

\end{proposition}

\begin{proposition} [Decay estimate for interior points]\label{PropInteriorDecay}
Let $\Omega$, $g\in C^{1,\alpha}(\overline{ \Omega},\R^N)$,  $\overline{u}\in\W^{1,p}_g(\Omega,\R^N)$ and $F\colon\overline{\Omega}\times\R^{N\times n}\rightarrow\R$ as in Theorem \ref{BdryRegularityTheorem}. For $\delta\in (0,1)$ and $B(x,r)\subseteq \Omega$  denote
\begin{equation}
E(x,r):=\underset{B(x,r)}{\mint}\left|V \left(     \nabla\overline{u}-(\nabla\overline{u})_{x,r}      \right)\right|^2\dy+r^{2\delta\alpha} .
\end{equation} 
Then, for every $m > 0$ there exists $C = C(m) > 0$ with the property
that, for each $\tau\in  (0, \frac{1}{2} )$ and $\delta\in (0,1)$, there exists $\varepsilon = \varepsilon(m, \tau,\delta ) > 0$ such that, if $B(x,r)\subseteq\Omega$, $|(\nabla \overline{u})_{x,r}|\leq m$ and $E(x,r)<\varepsilon$, then 
\begin{equation*}
E(x,\tau r)<C\tau^{2\delta\alpha}E(x,r). 
\end{equation*}
\end{proposition}

\begin{proof}[Proof of Proposition \ref{PropBoundaryDecay}]

By considering
\begin{equation*}
\tilde{u}:=\overline{u}-g
\end{equation*}
and $\tilde{F}\colon B^+ \times\R^{N\times n}\rightarrow\R$ given by
\begin{equation*}
\tilde{F}(x,z):=F\left(x,z+\nabla g\right)
\end{equation*} 
we can assume, without loss of generality, that $g=0$, so that $\overline{u}:=0$ on $\Gamma$. Furthermore, under assumptions (b) of Theorem \ref{BdryRegularityTheorem}, we can suppose $\nabla u\in \mathrm{VMO}_0$. 
\\

Assuming that the proposition is false, we can find  $m>0$ such that, for every $C>0$, there are  a corresponding $\tau\in (0,\frac{1}{2})$, a $\delta\in (0,1)$ and a sequence of half balls $B^+(x_j,r_j)\subseteq B^+$ with $x_j\in \Gamma$, such that $|(\nabla\overline{u})_{x_j,r_j}|\leq m$ and $E(x_j,r_j)\rightarrow 0$ but $E(x_j,\tau r_j)>C\tau^2 E(x_j,rj)$. 

We split the proof into several steps in order to obtain a contradiction for suitably large values of $C$. 
\\
\\
\textit{Step 1. The blow up}. We denote
\begin{equation*}
\zeta_j:=(\nabla_\mathrm{n}\overline{u})_{x_j,r_j}=\underset{B^+(x_j,r_j)}{\mint}|\nabla_\mathrm{n}\overline{u}\dx|,\,\,\,\mbox{ }\xi_j:=\zeta_j\otimes\mathrm{e_n},\,\,\,\mathrm{   } \lambda_j:=\sqrt{E(x_j,r_j)}
\end{equation*}
and define the function $u_j\colon B^+\rightarrow\R^{N \times n}$ by
\begin{equation*}
u_j(y):= \frac{\overline{u}(x_j+r_jy)-r_j\zeta_jy_\mathrm{n}}{r_j\lambda_j}.
\end{equation*}
Then, $u_j=0$ on $\Gamma$ and
\begin{equation*}
\nabla u_j(y):= \frac{\nabla\overline{u}(x_j+r_jy)-\xi_j}{\lambda_j}.
\end{equation*}
Additionally, it is clear that
\begin{equation}\label{ujbded}
\underset{B^+}{\mint}\left(|\nabla u_j|^2+\lambda_j^{p-2}|\nabla u_j|^p\right)\dx\leq 1
\end{equation}
and 
\begin{align}\label{eccontr}
&\underset{B^+(0,\tau)}{\mint}\left(|\nabla u_j-(\nabla_\mathrm{n} u_j)_{0,\tau}\otimes\mathrm{e_n} |^2+\lambda_j^{p-2} |\nabla u_j-(\nabla_\mathrm{n} u_j)_{0,\tau}\otimes\mathrm{e_n} |^p \right)\dx+\frac{(\tau r_j)^{2\delta\alpha}}{\lambda_j^2}>C\tau^{2\delta\alpha}.
\end{align}

The Euler-Lagrange equation satisfied by $\overline{u}$ implies that
\begin{equation}\label{ELforAharm}
\frac{1}{\lambda_j}\underset{B^+}{\int}\left(F_{z}(x_j+r_jx,\xi_j+\lambda_j\nabla u_j)- F_{z}(x_j,\xi_j)       \right)[\nabla\varphi]\dx=0
\end{equation}
for every $\varphi\in \W^{1,p}_0(B^+,\R^N)$.

In addition, assumption $(\mathrm{{H}C})_2$ implies that 
\begin{align}\label{convforAharm}
&\frac{1}{\lambda_j}\underset{B^+}{\int}\left(F_{z}(x_j,\xi_j+\lambda_j\nabla u_j)-F_{z}(x_j+r_jx,\xi_j+\lambda_j\nabla u_j)\right)[\nabla\varphi]\dx\notag\\
\leq&c\frac{r_j^\alpha}{\lambda_j}\underset{B^+}{\int}\left(1+|\lambda_j \nabla u_j|^{p-1} \right)|\nabla\varphi|\dx\leq cr_j^{(1-\delta)\alpha} \left(\underset{B^+}{\int}|\nabla\varphi|^p\dx\right)^{\frac{1}{p}} \longrightarrow 0
\end{align}
as $j\rightarrow\infty$. Here we have applied H\"older inequality for the last estimate and we note that $c=c(m,\|\nabla\overline{u}\|_{\LL^p})$.

Furthermore, we let $B_j^{+,\infty}:=\{x\in B^+\st \lambda_j|\nabla u_j|>1 \}$ and $B_j^{+,1}:=B^+-B_j^{+,\infty} $. Then, (\ref{ujbded}) implies that $\mathcal{L}^n(B_j^+)\leq \lambda_j^2\mathcal{L}^n(B^+)$. We now use the growth estimate for $F_z$ that follows from $(\overline{\mathrm{H}}1)-(\overline{\mathrm{H}}2)$, as well as $|\xi_j|\leq m$ and (\ref{ujbded}) together with H\"{o}lder inequality, to infer that, for $\varphi\in C_0^\infty(B^+,\R^N)$,
\begin{align}\label{FTCforAharm>1}
&\frac{1}{\lambda_j}\underset{B_j^{+,\infty}}{\int}\left(F_{z}(x_j,\xi_j+\lambda_j\nabla u_j)-F_{z}(x_j,\xi_j)\right)[\nabla\varphi]\dx\notag\\
\leq & c(p)\|\nabla\varphi\|_{\LL^\infty}\left( (1+m^{p-1})\mathcal{L}^n(B^+)+ \mathcal{L}^n(B^+) \right)\lambda_j.
\end{align}
On the other hand, 
\begin{align}\label{FTCforAharm<1}
&\frac{1}{\lambda_j}\underset{B_j^{+,1}}{\int}\left(F_{z}(x_j,\xi_j+\lambda_j\nabla u_j)-F_{z}(x_j,\xi_j)\right)[\nabla\varphi]\dx\notag\\
=&\underset{B_j^{+,1}}{\int}\int_0^1\left(F_{zz}(x_j,\xi_j+t\lambda_j\nabla u_j)-F_{zz}(x_j,\xi_j)\right)[\nabla u_j,\nabla\varphi]\dt\dx +\underset{B_j^{+,1}}{\int}F_{zz}(x_j,\xi_j)[\nabla u_j,\nabla\varphi]\dx\notag\\
=:&\,\mathrm{I+II}.
\end{align}
Observe that $\mathbbm{1}_{B_j^{+,1}}\rightarrow \mathbbm{1}_{B^+}$ in $L^1$. In addition, for subsequences that we do not relabel, we can assume that $u_j\rightharpoonup u$ in $\W^{1,2}(B^+,\R^N)$, $x_j\rightarrow x_0$ and $\xi_j\rightarrow \xi$ for some $u\in \W^{1,2}(B^+,\R^N)$ $x_0\in\Gamma$ and $\xi\in\R^{N\times n}$. Whereby,
\begin{equation}\label{IIto0forAharm}
\mathrm{II}\longrightarrow \underset{B^+}{\int}F_{zz}(x_0,\xi)[\nabla u,\nabla \varphi]\dx
\end{equation}
 as $j\rightarrow\infty$. 
On the other hand, observe that $\lambda_j^2\underset{B^+}{\int}|\nabla u_j|^2 \dx\rightarrow 0$. Hence, $\lambda_j\nabla u_j\rightarrow 0$ in measure. Given that $F_{zz}$ is continuous, we infer that $\mathbbm{1}_{ B_j^{+,1}}\left(F_{zz}(x_j,\xi_j+t\lambda_j\nabla u_j)-F_{zz}(x_j,\xi_j)\right)\rightarrow 0$ in measure for every $t\in[0,1]$. Additionally, this sequence is uniformly bounded in $B^+$. Therefore, Vitali's Convergence Theorem implies, for every $\varphi\in C^\infty_0(B^+,\R^N)$, that
\begin{align}\label{Ito0forAharm}
\mathrm{|I|}=&\left|\underset{B_j^{+,1}}{\int}\int_0^1\left(F_{zz}(x_j,\xi_j+t\lambda_j\nabla u_j)-F_{zz}(x_j,\xi_j)\right)[\nabla u_j,\nabla\varphi]\dt\dx\right|\notag\\
\leq& \underset{B_j^{+,1}}{\int}\int_0^1\left|\left(F_{zz}(x_j,\xi_j+t\lambda_j\nabla u_j)-F_{zz}(x_j,\xi_j)\right)[\nabla u_j,\nabla\varphi]\right|\dt\dx\notag\\
\leq& \left(\underset{B_j^{+,1}}{\int}\int_0^1\left|F_{zz}(x_j,\xi_j+t\lambda_j\nabla u_j)-F_{zz}(x_j,\xi_j)\right|^2\dt\dx \right)^{\frac{1}{2}}\|\nabla u_j\|_{\LL^2}\|\nabla\varphi\|_{\LL^\infty}\longrightarrow 0.
\end{align}
Inequalities (\ref{ELforAharm})-(\ref{Ito0forAharm}) enable us to conclude, after an approximation argument, that for every $\varphi\in\W^{1,2}_0(B^+,\R^N)$,
\begin{equation*}
\underset{B^+}{\int}F_{zz}(x_0,\xi)[\nabla u,\nabla \varphi]\dx=0.
\end{equation*}
This implies, by a classical regularity result due to Campanato,\footnote{See \cite[Teorema 9.2]{Campanatoeqelip}, \cite[Theorem 2.4]{Grotowskinonlin}, \cite[Proposition 4.2]{Kristensen}.} that $u$ is $C^1$ and that there exists a constant $\gamma_0>0$ such that, for every $0<r\leq\frac{1}{2}$, 
\begin{equation}\label{ec2forcont}
\underset{B^+(0,r)}{\mint}|\nabla u - (\nabla u)_{0,r}|^2\dx\leq \gamma_0 r^2. 
\end{equation}
\\
\\
\textit{Step 2.1. Caccioppoli inequality for boundary points.}
\\
\\
For given $x_0\in\Gamma$ and $r>0$ satisfying $B^+(x_0,r)\subseteq  B^+$ we define the affine function
\begin{equation*}
a_j(y):=y_n (\nabla_\mathrm{n}u_j)_{x_0,r}.
\end{equation*}

In this step we will show that there exists $\theta\in (0,1)$ independent of $B^+(x_0,r) $ such that, for every $\varsigma\in (0,1)$, 
\begin{align}\label{ecthetabdry}
&\underset{B^+(x_0,\varsigma r)}{\int} \left(|\nabla u_j - (\nabla_\mathrm{n}u_j)_{x_0,\varsigma r}\otimes \mathrm{e}_\mathrm{n}|^2+\lambda_j^{p-2} |\nabla u_j - (\nabla_\mathrm{n}u_j)_{x_0,\varsigma r}\otimes \mathrm{e}_\mathrm{n}|^p \right)\dx\notag\\
\leq& \theta\underset{B^+(x_0,r)}{\int}\left(|\nabla u_j - (\nabla_\mathrm{n}u_j)_{x_0, r}\otimes \mathrm{e}_\mathrm{n}|^2+\lambda_j^{p-2} |\nabla u_j - (\nabla_\mathrm{n}u_j)_{x_0, r}\otimes \mathrm{e}_\mathrm{n}|^p\right)\dx\notag\\
&+\theta\underset{B^+(x_0,r)}{\int}\left(\frac{|u_j-a_j|^2}{(1-\varsigma)^2r^2}+\lambda_j^{p-2}\frac{|u_j-a_j|^p}{(1-\varsigma)^pr^p}\right)\dx\notag\\
&+\theta r^n(1-\varsigma^n)\left(\left|(\nabla_\mathrm{n}u_j)_{x_0,r}\otimes \mathrm{e_n} \right|^2+ \lambda_j^{p-2} \left|(\nabla_\mathrm{n}u_j)_{x_0,r}\otimes \mathrm{e_n} \right|^p   \right)+\frac{r_j^{(1-\delta)\alpha}}{1-\varsigma}
\end{align}
holds for every $j\geq J=J(\varsigma,r)$, where $J(\varsigma,r)$ is allowed to depend on $\varsigma$ and $r$. 

Consider the functionals $F_j$, $\tilde{F}_j$ and $G_j$ as defined in Lemma \ref{LemmaGrowths0} and assume that $\rho$ is a cut-off function satisfying $\mathbbm{1}_{B(x_0,\varsigma r)}\leq\rho\leq \mathbbm{1}_{B(x_0, r)}$ and $|\nabla\rho|\leq \frac{1}{(1-\varsigma)r}$.

Define 
\begin{equation}
\varphi_j:=\rho(u_j-a_j);\mbox{\hspace{7mm}} \psi_j:= (1-\rho)(u_j-a_j).
\end{equation}

Then, the strong quasiconvexity condition ($\mathrm{\overline{H}2}$) implies for $G_j$ that
\begin{align*}
&\underset{B^+(x_0,\varsigma r)}{\int}\left(|\nabla u_j-\nabla a_j|^2+\lambda_j^{p-2}|\nabla u_j-\nabla a_j|^p\right)\dx\notag\\
\leq&\underset{B^+(x_0,r)}{\int}\left(|\nabla\varphi_j|^2+\lambda_j^{p-2}|\nabla\varphi_j|^p\right)\dx\notag\\
\leq&\,c\underset{B^+(x_0,r)}{\int}\left(G_j(\nabla a_j+\nabla \varphi_j)-G_j(\nabla a_j)\right)\dx\notag\\
=&\,c\underset{B^+(x_0,r)}{\int}\left(G_j(\nabla a_j+\nabla \varphi_j)-F_j(x,\nabla a_j+\nabla\varphi_j)     \right)\dx\notag\\
&+c\underset{B^+(x_0,r)}{\int}\left(F_j(x,\nabla u_j- \nabla\psi_j) -F_j(x,\nabla u_j)\right)\dx+c \underset{B^+(x_0,r)}{\int}\hspace{-2mm}\left(   F_j(x,\nabla u_j)-\tilde{F}_j(\nabla x,\nabla u_j)  \right)\dx\notag\\
&+c\underset{B^+(x_0,r)}{\int}\hspace{-2mm}\left(   \tilde{F}_j(x,\nabla u_j)- \tilde{F}_j(x,\nabla a_j)  \right)\dx +c\underset{B^+(x_0,r)}{\int}\hspace{-2mm}\left( \tilde{F}_j(x,\nabla a_j) -F_j(x,\nabla a_j) \right)\dx   \notag\\
&+c\underset{B^+(x_0,r)}{\int}\hspace{-2mm}\left( {F}_j(x,\nabla a_j) -G_j(\nabla a_j)\right)\dx\notag\\
=:&\,\hspace{3mm} \mathrm{I+II+III+IV+V+VI}.
\end{align*}
We estimate each term separately.
Since $u_j\rightharpoonup u$ in $\W^{1,2}$, $\nabla a_j$ is a convergent sequence and, therefore, bounded. Hence, by Lemma \ref{LemmaGrowths0} (iii), 
\begin{align*}
\mathrm{I+VI}\leq &cr_j^{\alpha(1-\delta)}   \underset{B^+(x_0,r)}{\int}\hspace{-2mm}\left(1+|\nabla u_j|^2 +|\nabla\psi_j|^2 +\lambda_j^{p-2}|\nabla u_j|^p+ \lambda_j^{p-2}|\nabla \psi_j|^p   \right)\dx\notag\\
\leq& cr_j^{\alpha(1-\delta)} \left(1+  \underset{B^+(x_0,r)}{\int}\hspace{-2mm}\left( |\nabla\psi_j|^2 +\lambda_j^{p-2}|\nabla \psi_j|^p   \right)\dx\right),
\end{align*}
where the second inequality follows from (\ref{ujbded}).

Similarly, we have
\begin{align*}
\mathrm{III+V}\leq cr_j^{\alpha(1-\delta)}\underset{B^+(x_0,r)}{\int}\hspace{-2mm}\left(1+|\nabla u_j|^2 +|\nabla a_j|^2  \right)\dx\leq cr_j^{\alpha(1-\delta)}.
\end{align*}
We now estimate IV, for which the minimality of $\overline{u}$ is crucial; it implies that $u_j$ is either
\begin{itemize}
\item[(a)] a $\W^{1,q}$-local minimizer of $I_j(u):= \underset{B^+}{\int}\tilde{F_j}(\nabla u)\dx$, under assumptions (a) of the theorem or
\item[(b)]  a $\mathrm{BMO}$-local minimizer of $I_j$.
\end{itemize} 
Hence, Lemma \ref{LemmaGrowths0} (ii) enables us to conclude that 
\begin{align}\label{estimateIV}
\mathrm{IV}\leq& \underset{B^+(x_0,r)}{\int}\hspace{-2mm}\left(   \tilde{F}_j(x,\nabla\psi_j+\nabla a_j)- \tilde{F}_j(x,\nabla a_j)  \right)\dx\notag\\
\leq&  c\underset{B^+(x_0,r)}{\int}\hspace{-2mm}\left( r_j^{\alpha(1-\delta)}  +  |\nabla\psi_j|+|\nabla a_j|+\lambda_j^{p-2}(|\nabla\psi_j|^ {p-1}+|\nabla a_j|^{p-1})  \right)|\nabla\psi_j|\dx,
\end{align}
provided 
\begin{itemize}
\item[(a)] $\|\nabla\varphi_j\|_{\LL^q(B(x_0,r))}<\frac{\delta}{\lambda_jr_j^{\frac{n}{q}}}$ for the case of $\W^{1,q}$-local minimizers or, respectively,
\item[(b)] $[\nabla\varphi_j]_{\mathrm{BMO}(B(x_0,r),\R^N)}<\frac{\delta}{\lambda_j}$ for the case of $\mathrm{BMO}$-local minimizers.
\end{itemize}
Here, $\varphi_j$ is assumed to take the value of $0$ in $B-B^+(x_0,r)$. We verify this smallness property in each case separately:
\begin{itemize}
\item[(a)] It follows by triangle and Poincar\'e inequalities that
\begin{equation*}
\|\nabla \varphi_j\|_{{\LL^q(B(x_0,r),\R^{N\times n})} } \leq \left( 1+\frac{c}{1-\varsigma} \right)\|\nabla u_j-\nabla a_j\|_{{\LL^q(B^+(x_0,r),\R^{N\times n})} }
\end{equation*}
for some constant $c>0$ depending on $n,N$ and $q$. 
We can then change coordinates and use Lemma \ref{quasiminimality}, together with the inclusion $B(x_j+r_jx_0,rr_j)\subseteq B(x_j,r_j)$ and Cauchy-Schwarz inequality to obtain that 
\begin{equation*}
\|\nabla u_j-\nabla a_j\|_{\LL^q(B^+(x_0,r),\R^{N\times n})}^q\leq \frac{2^q}{\lambda_j^qr_j^n}\underset{B^+(x_j,r_j)} {\int}|\nabla \overline{u}|^q\dx.
\end{equation*}
Since $\nabla\overline{u}\in\LL^q$, the previous inequalities imply that there is some $J=J(\varsigma,r)$ such that, for $j\geq J$, 
\begin{equation*}
\|\nabla \varphi_j\|_{{\LL^q(B(x_0,r),\R^{N\times n})} }\leq \frac{c}{\lambda_j r_j^\frac{n}{q}(1-\varsigma)}\left(\underset{B(x_j,r_j)}{\int}|\nabla\overline{u}|^q\right)^\frac{1}{q}<\frac{\delta}{\lambda_j r_j^\frac{n}{q}},
\end{equation*}
as we wished to obtain. 
\item[(b)] Regarding the case of $\mathrm{BMO}$-local minimizers, extend first the function $u_j-a_j$ by assigning the value of $0$ in $B-B^+$. We keep the notation $u_j-a_j$ to denote such extension. Then, the definition of $\varphi_j$ enables us to apply Lemma \ref{LemmaBMOforPhi}(i) and conclude that, for any $B(x,s)\subseteq B(x_0,r)$,
\begin{align}\label{ec0forphidelta}
&\underset{B(x,s)}{\mint}|\nabla \varphi_j-(\nabla\varphi_j)_{x,s}|\dy\notag\\
\leq&c\underset{B(x,s)}{\mint}|\nabla u_j-(\nabla u_j)_{x,s}|\dy+c\left(\underset{B(x_0,r)}{\mint}|\nabla u_j-(\nabla u_j)_{x_0,r}|^{n+1}\dy\right)^\frac{1}{n+1}
\end{align}
for some constant $c=c(n,\varsigma)>0$.    
We now observe that, since $|x|\leq 1-s$ because $B(x,s)\subseteq B(0,1)$, then $B(x_j+r_jx,sr_j)\subseteq B(x_j,r_j)$. With this, we infer after a change of variables that, if $\overline{u}$ still denotes the extension of $\overline{u}$ from $B^+$ to $B$ that takes the value of $0$ off $B^+$, then
\begin{align}\label{ec1forphidelta}
\underset{B(x,s)}{\mint}|\nabla u_j-(\nabla u_j)_{x,s}|\dy\leq \frac{1}{\lambda_j}\underset{B(x_j,r_j)}{\mint}|\nabla \overline{u}-(\nabla \overline{u})_{x_j,r_j}|\dy.
\end{align}
Similarly, $B(x_j+r_jx_0,rr_j)\subseteq B(x_j,r_j)$ and, hence,
\begin{align}\label{ec2forphidelta}
\left(\underset{B(x_0,r)}{\mint}|\nabla u_j-(\nabla u_j)_{x_0,r}|^{n+1}\dy\right)^\frac{1}{n+1}\leq \frac{2^\frac{n}{n+1}\|\nabla\overline{u}\|_{\LL^\infty}^{\frac{n}{n+1}}}{\lambda_j}\left(\underset{B(x_j,r_j)}{\mint}|\nabla \overline{u}-(\nabla \overline{u})_{x_j,r_j}|\dy\right)^\frac{1}{n+1}.
\end{align}
Since $\nabla\overline{u}\in \mathrm{VMO}_{0}(B^+,\R^N)$ and $g\in C^{1,\alpha}$, inequalities (\ref{ec0forphidelta})-(\ref{ec2forphidelta}) imply that
\begin{equation}
[\nabla\varphi_j]_{\mathrm{BMO}}\leq \frac{\delta}{\lambda_j}
\end{equation}
for every $j\geq J(n,\varsigma,\|\nabla\overline{u}\|_{\LL^\infty})$ and, hence, estimate (\ref{estimateIV}) remains valid for such large values of $j$. 

\end{itemize}

We further estimate $\mathrm{IV}$ by simplifying the term $r_j^{\alpha(1-\delta)}\underset{B^{+}}{\int}|\nabla\psi_j|\dx$, for which we use the definition of $\psi_j$, Poincar\'e inequality and (\ref{ujbded}). We conclude that

\begin{align*}\label{estimateIVb}
\mathrm{IV}\leq&  c\left( \frac{r_j^{\alpha(1-\delta)}}{1-\varsigma}  + \underset{B^+(x_0,r)}{\int}\hspace{-2mm}\left( |\nabla\psi_j|+|\nabla a_j|+\lambda_j^{p-2}(|\nabla\psi_j|^ {p-1}+|\nabla a_j|^{p-1}) \right)|\nabla\psi_j| \dx\right).
\end{align*}

Finally, we use part (iv) of Lemma \ref{LemmaGrowths0} to estimate
\begin{align*}
\mathrm{II}\leq &\, c \underset{B^+(x_0,r)}{\int}\hspace{-2mm}\left(|\nabla u_j|+|\nabla\psi_j|+\lambda_j^{p-2}\left(|\nabla u_j|^{p-1}+ |\nabla \psi_j|^{p-1}\right)  \right)|\nabla\psi_j|\dx.
\end{align*}

Bringing together the previous estimates and using Lemma \ref{quasiminimality}, that $ab^{p-1}\leq a^p+b^p$ as well as triangle inequality, we obtain 
\begin{align*}
&\underset{B^+(x_0,\varsigma r)}{\int}\left(|\nabla u_j-(\nabla_\mathrm{n} u_j)_{x_0,\varsigma r}\otimes \mathrm{e_n}|^2+\lambda_j^{p-2}(|\nabla u_j-(\nabla_\mathrm{n} u_j)_{x_0,\varsigma r}\otimes \mathrm{e_n}|^p \right)\dx\\
\leq & \,c \underset{B(x_0, r)\backslash B(x_0,\varsigma r)}{\int}\left(|\nabla u_j-(\nabla_\mathrm{n} u_j)_{x_0, r}\otimes \mathrm{e_n}|^2 + \lambda_j^{p-2}|\nabla u_j-(\nabla_\mathrm{n} u_j)_{x_0, r}\otimes \mathrm{e_n}|^p \right) \dx\nonumber\\
&+c\underset{B(x_0, r)}{\int}\left(\frac{1}{(1-\varsigma)^2r^2}|u_j-a_j|^2+\frac{\lambda_j^{p-2}}{(1-\varsigma)^pr^p}|u_j-a_j|^p  \right)\dx\nonumber\\
&+c r^n(1-\varsigma^n)\left(|(\nabla_\mathrm{n} u_j)_{x_0, r}\otimes \mathrm{e_n}|^2+\lambda_j^{p-2}|(\nabla_\mathrm{n} u_j)_{x_0, r}\otimes \mathrm{e_n}|^p\right) + c\frac{r_j^{\alpha(1-\delta)}}{1-\varsigma}.
\end{align*}
Then, the Caccioppoli inequality of the first kind (\ref{ecthetabdry}) follows after applying Widman's hole-filling trick and with $\theta:=c/(1+c)$.
\\
\\
\textit{Step 2.2. Caccioppoli inequality for interior points.}
The derivation of a similar Caccioppoli inequality for interior points follows exactly the same strategy as before, with the main difference being that, if $B(x_0,r)\subseteq B^+$, we can define 
\begin{equation}\label{definPhijinterior}
\varphi_j:=\rho(u_j-b_j);\mbox{\hspace{7mm}} \psi_j:= (1-\rho)(u_j-b_j)
\end{equation}
with $b_j$ an arbitrary affine function. Hence, we fix $B(x_0,r)\subseteq B^+$ and let $b_j(y):= (u_j)_{x_0,r}+(\nabla u_j)_{x_0,r}\cdot(y-x_0)$. By repeating all the steps performed to obtain (\ref{ecthetabdry}), but using part (ii) (instead of (i)) of Lemma \ref{LemmaBMOforPhi} for the case  of $\mathrm{BMO}$-local minimizers, we also establish that, for some $\theta\in (0,1)$ independent of $j, \varsigma$ and $B(x_0,r)$,
\begin{align}\label{ecthetaint}
&\underset{B(x_0,\varsigma r)}{\int}\left(|\nabla u_j-(\nabla u_j)_{x_0,\varsigma r}|^2+\lambda_j^{p-2}(|\nabla u_j-(\nabla u_j)_{x_0,\varsigma r}|^p \right)\dx\notag\\
\leq & \,\theta \underset{B(x_0, r)}{\int}\left(|\nabla u_j-(\nabla u_j)_{x_0, r}|^2 + \lambda_j^{p-2}|\nabla u_j-(\nabla u_j)_{x_0, r}|^p \right) \dx\nonumber\\
&+\theta\underset{B(x_0, r)}{\int}\left(\frac{1}{(1-\varsigma)^2r^2}|u_j-b_j|^2+\frac{\lambda_j^{p-2}}{(1-\varsigma)^pr^p}|u_j-b_j|^p  \right)\dx\nonumber\\
&+\theta r^n(1-\varsigma^n)\left(|(\nabla u_j)_{x_0, r}|^2+\lambda_j^{p-2}|(\nabla u_j)_{x_0, r}|^p\right) + \theta \frac{r_j^{\alpha(1-\delta)}}{1-\varsigma}.
\end{align}
\\
\\
\textit{Step 3. Strong convergence.}

This stage of the proof consists in establishing the strong convergence (up to a subsequence) of $\nabla u_j$ to $\nabla u$ in $B^+(0,\sigma)$. The idea is to consider the extensions of $u_j$ and $u$ to the unit ball $B$ by assigning the value of $0$ in $B-B^+$. Such extensions belong to $\W^{1,2}(B,\R^N)$. We shall then see $|\nabla u_j|^2\mathcal{L}^n$ as a convergent sequence of measures on $B$ and prove that the limit measure coincides with $|\nabla u|^2\mathcal{L}^n$. This strategy was developed by Kristensen and Taheri \cite{Kristensen} to overcome the fact that one cannot obtain a Caccioppoli inequality of the second kind, as in the case of global minimizers, for local minimizers. We reproduce Kristensen-Taheri's proof here while highlighting the  appropriate fine modifications in the argument to treat the boundary points. 

The goal is then to show that for every $\sigma<1$, 
\begin{equation*}
\underset{B^+(0,\sigma)}{\int}\left( |\nabla (u_j-u)|^2+\lambda_j^{p-2}|\nabla (u_j-u)|^p\right)\dx\rightarrow 0 \text{\hspace{3mm} as \hspace{3mm} $j\rightarrow\infty$}.
\end{equation*}

From now on, $u_j$ and $u$ denote the corresponding extensions in $\W^{1,2}(B,\R^N)$ of these functions. Since $(u_j)$ is bounded in $\W^{1,2}$,  there is a further subsequence, that for convenience we do not relabel, and a positive finite Borel measure $\mu$, such that
\begin{equation*}
\left(|\nabla u_j|^2+\lambda_j^{p-2}|\nabla u_j|^p\right)\mathcal{L}^n\overset{*}{\rightharpoonup}\mu
\end{equation*}
in $C^0(\overline{B})^*$.
We note that $u_j\rightharpoonup u$ in $\W^{1,2}(B,\R^N)$ and, for a set $A\subseteq \overline{B}$, let
\begin{equation*}
\nu(A):=\left(\mu-|\nabla u|^2\mathcal{L}^n\right)(A).
\end{equation*} 
It is then clear that $\nu$  is a positive, finite Borel measure on $\overline{B}$. 
The main objective is now to show that 
\begin{equation}\label{nulargerLebmeas}
\underset{r\rightarrow 0^+}{\liminf}\frac{\nu({B[x_0,r]})}{r^n}=0
\end{equation}
for every $x_0\in B$. 

It is to prove this statement that we will make use of the Caccioppoli inequalities obtained. 
We first focus on establishing this limit for the case $x_0\in\Gamma$, given that if $x_0\in B^+$, the proof follows the same argument after some simple modifications. The case $x_0 \in B^-$ is trivial because then $\nu (B[x_0,r])=0$ for $r$ small enough. 

We define the map 
\begin{equation*}
a(y):=\left\{
\begin{array}{ll}
y_n (\nabla_\mathrm{n} u)_{B^+(x_0,r)} &\mbox{ if }y\in B^+\cup \Gamma\\
0 &\mbox{ if }y\in B-B^+.
\end{array}
\right.
\end{equation*}
By weak convergence, it is clear that $a_j\rightarrow a$ in $\LL^2$. Furthermore, since $\lambda_j\rightarrow 0$ as $j\rightarrow\infty$, Rellich-Kondrachov Embedding theorem implies that
\begin{equation*}
\underset{B(x_0,r)}{\int}\left(\frac{|u_j-a_j|^2}{(1-\varsigma)^2r^2}+\lambda_j^{p-2}\frac{|u_j-a_j|^p}{(1-\varsigma)^pr^p}\right)\dx\rightarrow \frac{1}{(1-\varsigma)^2r^2}\underset{B(x_0,r)}{\int}|u-a|^2\dx.
\end{equation*}
On the other hand, observe that for every $\varepsilon>0$ there is a constant $c_\varepsilon>0$ such that
\begin{equation}\label{Auxineqp}
(1-\varepsilon)|\nabla u_j|^p-c_\varepsilon|\nabla a_j|^p\leq|\nabla u_j-\nabla a_j|^p\leq(1+\varepsilon)|\nabla u_j|^p+c_\varepsilon|\nabla a_j|^p.
\end{equation}
In addition, by elementary properties of the convergence of Radon measures (see, for example, \cite[Section 1.9]{EvGar}), for any Borel set $A\subseteq \overline{B}$ we have that
\begin{align*}
\mu(\mathrm{int}\, A)\leq &\underset{j\rightarrow\infty}{\liminf}\underset{A}{\int}\left(|\nabla u_j|^2+\lambda_j^{p-2}|\nabla u_j|^p\right)\dx\notag\\
\leq& \underset{j\rightarrow\infty}{\limsup}\underset{A}{\int}\left(|\nabla u_j|^2+\lambda_j^{p-2}|\nabla u_j|^p\right)\dx\notag\\
\leq& \mu(\overline{A}). 
\end{align*}
It then follows from (\ref{ecthetabdry}) and the auxiliary inequality (\ref{Auxineqp}) that
\begin{align*}
&(1-\varepsilon)\mu(B(x_0,\varsigma r))+\underset{B(x_0,\varsigma r)}{\int}\left( |\nabla a|^2-2\langle\nabla a,\nabla u\rangle \right)\dx\notag\\
\leq& (1-\varepsilon)\underset{j\rightarrow\infty}{\liminf}\underset{B(x_0,\varsigma r)}{\int}\left(|\nabla u_j|^2+\lambda_j^{p-2}|\nabla u_j|^p\right)\dx+\underset{j\rightarrow\infty}{\lim}\underset{B(x_0,\varsigma r)}{\int}\left(|\nabla a_j|^2-2\langle\nabla a_j,\nabla u_j\rangle\right)\dx\notag\\
\leq &\underset{j\rightarrow\infty}{\liminf}\underset{B(x_0,\varsigma r)}{\int}\left(|\nabla u_j-\nabla a_j|^2+\lambda_j^{p-2}|\nabla u_j-\nabla a_j|^p\right)\dx\notag\\
\leq &\underset{j\rightarrow\infty}{\limsup}\underset{B(x_0,\varsigma r)}{\int}\left(|\nabla u_j-\nabla a_j|^2+\lambda_j^{p-2}|\nabla u_j-\nabla a_j|^p\right)\dx\notag\\
\leq& \theta\left( (1+\varepsilon)\mu({B[x_0,r]})+\underset{B(x_0,r)}{\int}\left(|\nabla a|^2-2\langle\nabla a,\nabla u\rangle\right)\dx \right)\notag\\
&+\frac{\theta}{(1-\varsigma)^2r^2}\underset{B(x_0,r)}{\int}|u-a|^2\dx+\theta|(\nabla_\mathrm{n}u)_{B^+(x_0,r)}|^2r^n(1-\varsigma^n).
\end{align*}
Since this holds for any $\varepsilon>0$, it readily implies that
\begin{align}\label{postCacciopVMOmin}
&\mu(B(x_0,\varsigma r))-\underset{B(x_0,\varsigma r)}{\int}|\nabla u|^2\dx\notag\\
\leq& \theta\left( \mu({B[x_0,r]})-\underset{B(x_0,r)}{\int}|\nabla u|^2\dx \right)+\frac{\theta}{(1-\varsigma)^2r^2}\underset{B(x_0,r)}{\int}|u-a|^2\dx\notag\\
&+\theta|(\nabla_\mathrm{n}u)_{B^+(x_0,r)}|^2r^n(1-\varsigma^n)+\underset{B(x_0,r)}{\int}|\nabla u-\nabla a|^2\dx.
\end{align}

In addition, for all but a countable number of $r\in (0,1-|x_0|)$ we have that $\nu(\partial B(x_0,\varsigma r) )=0$. Then, from (\ref{postCacciopVMOmin}) we deduce that, for such $r\in (0,1-|x_0|)$,
\begin{equation}\label{CacciopforNu}
\nu({B[x_0,\varsigma r]})\leq \theta\nu({B[x_0,r]})+\left(\frac{\varepsilon_1(r)}{(1-\varsigma)^2}+\varepsilon_2(r)(1-\varsigma^n)+\varepsilon_3(r)\right)r^n
\end{equation}
for every $\varsigma\in (0,1)$, where
\begin{equation*}
\varepsilon_1(r):=\frac{\theta}{r^{n+2}}\underset{ B^+(x_0,r)}{\int}|u-a|^2\dx,\mbox{\hspace*{7mm}}\varepsilon_2(r):=\theta|(\nabla_\mathrm{n}u)_{B^+(x_0,r)}|^2
\end{equation*}
and
\begin{equation*}
\varepsilon_3(r):=\frac{1}{r^n}\underset{ B^+(x_0,r)}{\int}|\nabla u-\nabla a|^2\dx.
\end{equation*}
Since $u$ is of class $C^1$ in $B^+\cup\Gamma$,  $\varepsilon_2(r)\rightarrow\theta|\nabla_\mathrm{n} u(x_0)\otimes \mathrm{e_n}|^2$ as $r\rightarrow 0^+$. Furthermore, Corollary \ref{corollequivnormaldec} implies that $\varepsilon_3(r)\rightarrow 0$. Using this, the Poincar\'{e}-type estimate from Proposition \ref{Poincnormalderiv} and Cauchy-Schwarz inequality, we also obtain that $\varepsilon_1(r)\rightarrow 0$.

In order to prove (\ref{nulargerLebmeas}), observe first that, if $\limsup_{r\rightarrow 0^+}\frac{\nu({B[x_0,r]})}{r^n}=0$, the claim trivially follows. We therefore assume that
\begin{equation}\label{limsuppositive}
\underset{r\rightarrow 0^+}{\limsup}\frac{\nu({B[x_0,r]})}{r^n}>0. 
\end{equation}
This implies, in particular, that $\nu{(B[x_0,r])}>0$ for all $r>0$ and, therefore, we can rewrite (\ref{CacciopforNu}) for all but a countable number of $r\in (0,1-|x_0|)$,  so that 
\begin{equation*}
\frac{\nu({B[x_0,\varsigma r]})}{\nu({B[x_0,r]})}\leq\theta+\left( \frac{\varepsilon_1(r)}{(1-\varsigma)^2}+\varepsilon_2(r)(1-\varsigma^n)+\varepsilon_3(r)\right)\frac{r^n}{\nu(B[x_0,r])}.
\end{equation*}
 We can now take the limit superior as $r\rightarrow 0^+$ and deduce that, for $\varsigma\in (0,1)$, 
\begin{align}\label{afterlimr}
&\underset{r\rightarrow 0^+}{\limsup}\frac{\nu(B[x_0,\varsigma r])}{\nu(B[x_0,r])}\notag\\
\leq &\, \theta+\underset{r\rightarrow 0^+}{\limsup}\left(\left( \frac{\varepsilon_1(r)}{(1-\varsigma)^2}+\varepsilon_2(r)(1-\varsigma^n)+\varepsilon_3(r) \right)\frac{r^n}{\nu(B[x_0,r])}\right)\notag\\
\leq&\, \theta+\theta|\nabla_\mathrm{n} u(x_0)|^2(1-\varsigma^n) \underset{r\rightarrow 0^+}{\limsup}\frac{r^n}{\nu(B[x_0,r])}.
\end{align}
We will now show that, if (\ref{limsuppositive}) holds, then
\begin{equation}\label{expressiongeqvarsigma}
\underset{r\rightarrow 0^+}{\limsup}\frac{\nu(B[x_0,\varsigma r])}{\nu(B[x_0,r])}\geq \varsigma^n.
\end{equation}
We proceed by a contradiction argument and observe that, if $\underset{r\rightarrow 0^+}{\limsup}\frac{\nu(B[x_0,\varsigma r])}{\nu(B[x_0,r])}<\varsigma^n$, then we can take $\gamma<1$ and $r_\gamma>0$ such that
\begin{equation*}
\nu(B[x_0,\varsigma r])<\gamma\varsigma^n\nu(B[x_0,r])
\end{equation*}
for every $r\leq r_\gamma$. We can then iterate this and obtain that, for every $k\in\N^+$, 
\begin{equation*}
\nu(B[x_0,\varsigma^k r])<(\gamma\varsigma^n)^k\nu(B[x_0,r]).
\end{equation*}
We now fix a sequence $(r_j)$ such that $r_j\rightarrow 0^+$ and
\begin{equation*}
\underset{j\rightarrow\infty}{\lim}\frac{\nu(B[x_0,r_j])}{r_j^n}=\underset{r\rightarrow 0^+}{\limsup}\frac{\nu(B[x_0,r])}{r^n}. 
\end{equation*}
Since $\varsigma\in (0,1)$, we have that $(0,r_\gamma]=\bigcup_{k=0}^\infty[\varsigma^{k+1}r_\gamma,\varsigma^k r_\gamma]$. Hence, as $j\rightarrow\infty$ we have that
\begin{equation*}
k_j:=\sup\left\{\ k\st r_j\in [\varsigma^{k+1}r_\gamma,\varsigma^k r_\gamma] \right\}\rightarrow\infty.
\end{equation*}
On the other hand, since $r_j\leq \varsigma^{k_j}r_\gamma$ and $\varsigma^{k_j+1}r_\gamma\leq r_j$, 
\begin{equation*}
\frac{\nu(B[x_0,r_j])}{r_j^n}\leq\frac{\nu(B[x_0,\varsigma^{k_j}r_\gamma])}{(\varsigma^{k_j+1}r_\gamma)^n}<\frac{1}{\varsigma^n}\gamma^{k_j}\frac{\nu(B[x_0,r_\gamma])}{r_\gamma^n}\rightarrow 0
\end{equation*}
as $j\rightarrow\infty$, which contradicts (\ref{limsuppositive}).

Thus, from (\ref{afterlimr}) and (\ref{expressiongeqvarsigma}) we conclude that  
\begin{equation*}
\varsigma^n\leq\theta+\theta|\nabla_\mathrm{n} u(x_0)|^2(1-\varsigma^n)\underset{r\rightarrow 0^+}{\limsup}\frac{r^n}{\nu(B[x_0,r])}
\end{equation*}
for every $\varsigma\in (0,1)$. We now take the limit as $\varsigma\rightarrow 1^-$ and deduce that, since $0<\theta<1$, it must hold that
\begin{equation*}
\underset{r\rightarrow 0^+}{\limsup}\frac{r^n}{\nu(B[x_0,r])}=\infty.
\end{equation*}
This means that (\ref{nulargerLebmeas}) is valid for every $x_0\in \Gamma$. 

The proof for the interior case, $x_0\in B^+$, can be performed in the exact same way, with the main difference being that we consider whole instead of only half balls in the derivation of (\ref{postCacciopVMOmin}), and taking only $r>0$ small enough that $B(x_0,r)\subseteq B^+$. Additionally, we consider the affine function $b(y):= (u)_{x_0,r}+(\nabla u)_{x_0,r}\cdot(x-x_0)$ instead of $a(y)$ and we use the Caccioppoli inequality (\ref{ecthetaint}) instead of (\ref{ecthetabdry}). 

Furthermore, in order to prove that 
\begin{equation*}
\tilde{\varepsilon}_3(r):=\frac{1}{r^n}\underset{B(x_0,r)}{\int}|\nabla u-\nabla b|^2\dx\rightarrow 0
\end{equation*}
as $r\rightarrow 0^+$, we only need to use that $u$ is of class $C^1$, and not Corollary \ref{corollequivnormaldec}. That 
\begin{equation*}
\tilde{\varepsilon}_1(r):=\frac{\theta}{r^{n+2}}\underset{B(x_0,r)}{\int}| u- b|^2\dx\rightarrow 0
\end{equation*}
follows then by the standard Poincar\'{e} inequality. The rest of the argument can be carried out in a completely analogous way than for boundary points. 

Having shown (\ref{nulargerLebmeas}) for every $x_0\in B$, we now fix  $\sigma\in (0,1)$ and use Vitali's covering theorem (see \cite[p.35]{EvGar}), as well as (\ref{nulargerLebmeas}), to obtain for a given $\varepsilon>0$ a countable family of disjoint open balls $\{B_i\}_{i\in I}$, such that $B_i \subseteq B$,
\begin{equation*}
\nu\left(B[0,\sigma]\backslash \bigcup_{i\in I}B_i\right)=0
\end{equation*} 
and 
\begin{equation*}
\nu(B_i)<\varepsilon\mathcal{L}^n(B_i)
\end{equation*}
for every $i\in I$. Then, $\nu(B[0,\sigma])\leq\varepsilon\mathcal{L}^n(B)$ and, therefore, $\nu\measurerestr B=0$. In other words, $\mu\measurerestr B=|\nabla u|^2\mathcal{L}^n\measurerestr B$, which means that
\begin{equation*}
\underset{B^+(0,\sigma)}{\int}\left(|\nabla u_j|^2+\lambda_j^{p-2}|\nabla u_j|^p\right)\dx\rightarrow\underset{B^+(0,\sigma)}{\int}|\nabla u|^2\dx
\end{equation*}
for each $\sigma\in (0,1)$. This implies, in turn, that $\nabla u_j\rightarrow\nabla u$ strongly in $\LL^2(B^+(0,\sigma),\R^{N\times n})$ and, when $p>2$, $\lambda_j^{p-2}\underset{B^+(0,\sigma)}{\int}|\nabla u_j|^p\rightarrow 0$.
 
We can now take limit inferior in (\ref{eccontr}) and apply Corollary \ref{corollequivnormaldec} to obtain that
\begin{equation*}
c_*\underset{B^+(0,\tau)}{\mint}|\nabla u-(\nabla u)_{0,\tau}|^2\dx+\tau^{2\delta\alpha}\geq \underset{B^+(0,\tau)}{\mint}|\nabla u-(\nabla_\mathrm{n} u)_{0,\tau}\otimes\mathrm{e_n}|^2\dx+\tau^{2\delta\alpha}\geq C\tau^{2\delta\alpha}, 
\end{equation*}
which is a contradiction to (\ref{ec2forcont}) if we take $C>\gamma_0\,c_*+1$ and $r=\tau$. 

This proves that, given $m>0$, there exists a $C=C(m)>0$ with the property that, for every $\delta\in (0,1)$ and for each $\tau\in(0,\frac{1}{2})$, there is an $\varepsilon=\varepsilon(m,\tau,\delta)>0$ such that, if $x\in \Gamma$, $B^+(x_0,r)\subseteq B^+$ and  $E(x,r)<\varepsilon$, then
\begin{equation*}
E(x,\tau r)<C\tau^{2\delta\alpha}E(x,r).
\end{equation*}
\end{proof}

\begin{proof}[Proof of Proposition \ref{PropInteriorDecay}]
The proof for the decay in the interior case in the context of $\W^{1,q}$-local minimizers can be found in \cite{Kristensen} for homogeneous integrands. For the non-homogeneous case and, in particular for $\mathrm{BMO}$-local minimizers, the main modification is the one required to obtain the Caccioppoli inequality for interior points, as in Step 2.2. Hence, we do not elaborate any further on this point. We refer the interested reader to \cite[Chapter 3]{JCCThesis} for further details regarding the homogeneous setting. 
\end{proof}
The proof of the main theorem of this section is now straightforward. 
\begin{proof}[Proof of Theorem \ref{BdryRegularityTheorem}]
For $x_0\in \Omega_0$, a standard iteration of Proposition \ref{PropInteriorDecay} enables us to conclude that, for $\delta\in (0,1)$, 
\begin{equation}\label{FinalDecay}
E(x,r)\leq Lr^{2\delta\alpha}
\end{equation}
in a neighbourhood of $x_0$ for some constant $L:=L(\delta,\alpha, m)$, where $m$ and $R>0$ are such that $|(\nabla\overline{u})_{x_0,R}|<m$ and $0< r\leq R$. 
\begin{equation*}
E(x_0,r)\leq Lr^{2\delta\alpha}.
\end{equation*}
On the other hand, if $x_0\in \Sigma_0$, the regularity assumption made on $\Omega$ enables us to find a neighbourhood $B(x_0,R_0)$ and a diffeomorphism $\Phi: \overline{B^+}\rightarrow \overline{\Omega(x_0,R_0)} $ of class $C^{1,\alpha}$ such that $\Phi(\Gamma)=B(x_0,R_0)\cap\partial\Omega$. Then, by considering $\tilde{F}\colon B^+\rightarrow\R$ given by
\begin{equation*}
\tilde{F}(x,z):=F(\Phi(x),z(\nabla\,\Phi(x))^{-1})|\mathrm{det}\nabla\Phi(x)|,
\end{equation*}
we have that $\tilde{u}:=\overline{u}\circ\Phi$ is a corresponding local minimizer of the functional 
\begin{equation*}
v\mapsto \underset{B^+}{\int}\tilde{F}(\cdot,\nabla v)\dx.
\end{equation*}
Observe that $\tilde{F}$ satisfies all assumptions $(\mathrm{\overline{H}}0)$-$(\mathrm{\overline{H}}2)$ and $(\mathrm{{HC}})$. Hence, we can assume without loss of generality that $\Omega=B^+$ and that $x_0\in \Gamma$.

Furthermore, we can use Corollary \ref{corollequivnormaldec} to conclude that also 
\begin{equation}
\underset{r\rightarrow 0}{\liminf} \,E(x_0,r)\rightarrow 0, 
\end{equation}
with $E(x_0,r)$ given as in (\ref{excessboundary}). Whereby, we can now iterate Proposition \ref{PropBoundaryDecay} to obtain some $R>0$ so that

\begin{equation*}
E(x_0,r)\leq Lr^{2\delta\alpha}
\end{equation*}
for $0<r<R$. Observe that this decay estimate is only available for points in $\Gamma$ and appears exclusively in terms of half balls. Hence, in order to obtain (\ref{FinalDecay}) in a neighbourhood of $x_0\in\Gamma$, we need to combine the estimates obtained for interior and half balls. We refer the interested reader to \cite[Theorem 2.3 \& Section 3.6]{Grotowskinonlin} and \cite[Section 5]{Becksubquadlowdim} for details on this procedure. 

By Campanato-Meyers characterization of H\"{o}lder continuity, and considering that $\Phi$ and $g$ are assumed to be of class $C^{1,\alpha}$, this implies that $\overline{u}\in C^{1,\beta}_{\mathrm{loc}}(\overline{\Omega},R^{N\times n})$ for every $\beta\in [0,\alpha)$.
\end{proof}

\section{A sufficiency theorem on the strong local minimality of Lipschitz extremals with VMO derivative}\label{SectGralizedSuff}
As it has been mentioned in the introduction, an outstanding question in connection with the problem of sufficient conditions for strong local minimizers concerns the \textit{a priori} regularity assumptions that the extremal is required to satisfy in order to obtain the vectorial version of Weierstrass' sufficient conditions. The proof of the sufficiency theorem obtained by Grabovsky and Mengesha \cite{GyM} in its original version, as well as the alternative strategy that we suggest here, make use of the $C^1$ regularity up to the boundary that the extremal satisfies in order to reach the conclusion. 

However, given that strong local minimizers can only be shown to be partially regular \cite{Kristensen}, having to assume full regularity is not an ideal situation. On the other hand, the example constructed by Kristensen and Taheri in \cite[Section 7]{Kristensen} implies that Lipschitz regularity is not enough to ensure the strong local minimality of an extremal at which the second variation is strictly positive.

The result that we present in this section aims at improving the regularity assumptions imposed on the extremal to obtain strong local minimality under Dirichlet boundary conditions. Considering that Lipschitz regularity has been shown be insufficient, we prove the following result for Lipschitz extremals whose derivative satisfies, in addition, being of vanishing mean oscillation. This is a natural assumption in view of the available characterizations of regular points for strong local minimizers. However, we remark that this result appears as a straightforward consequence of Theorem \ref{BdryRegularityTheorem} and, hence, it does not remove  yet the intrinsic nature of the full regularity assumption made on Theorem \ref{theoGM}.

\begin{theorem}\label{theosufficVMO}
Let $\Omega$ be a bounded domain of class $C^{1,\alpha}$ and $g\in C^{1,\alpha}(\overline{\Omega},\R^N)$ for a given  $\alpha\in(0,1)$. Assume that $F:\overline{\Omega}\times\R^{N\times n}\rightarrow\R$ satisfies  $(\mathrm{{H}C})$ and  $(\overline{\mathrm{H}}0)$-$(\overline{\mathrm{H}}2)$ for some $p\in [2,\infty)$. Let ${u}\in \W^{1,\infty}_g(\Omega,\R^{N\times n})$ be an extremal at which the second variation is strictly positive, i.e., satisfying (\ref{uextremalVMO}) and (\ref{secvarposVMO}), and such that $\nabla{u}\in\mathrm{VMO}_{\nabla g}(\Omega,\R^{N\times n})$. Then, $u$ is an $\LL^p$-local minimizer. 
\end{theorem}
\begin{proof}
Theorems \ref{morethanWLMVMO} and \ref{BdryRegularityTheorem} readily imply that $u\in C^{1,\beta}(\overline{\Omega},\R^N)$ for every $\beta\in[0,\alpha)$. By Theorem \ref{theoGM} and Remark \ref{RemarkSuffNonhom}, it follows immediately that $u$ is an $\LL^p$-local minimizer. 
\end{proof}

\section*{Acknowledgements}
Part of this work was developed while I was doing my doctoral studies at the University of Oxford. I am grateful for uncountable fruitful discussions with Jan Kristensen, which have largely contributed to the ideas in this paper. I also wish to thank John Ball, Lisa Beck, Nicola Fusco and Konstantinos Koumatos for their useful remarks, as well as to the anonymous referees, whose helpful suggestions improved this manuscript. I am grateful to CONACyT, Mexico, and to the Schlumberger Foundation, Faculty for the Future, for the financial support provided during the initial stages of this project, as well as to the University of Augsburg, where this work was concluded.
\newpage
%

\def\cprime{$'$} \def\cprime{$'$}

\end{document}